\documentclass[10pt]{article}
\usepackage{latexsym}
\usepackage{indentfirst}
\usepackage{amssymb,amsfonts,amsmath,amsthm}
\usepackage{graphicx}
\usepackage{mathrsfs}
\usepackage{array}
\usepackage{color}
\usepackage{epstopdf}
\usepackage[all]{xy}
\usepackage{hyperref}
\usepackage{url}

\usepackage{geometry}
\newtheorem{thm}{Theorem}[section]
\newtheorem{prop}[thm]{Proposition}
\newtheorem{lem}[thm]{Lemma}
\newtheorem{conj}[thm]{Conjecture}
\newtheorem{cor}[thm]{Corollary}

\newtheorem{exam}[thm]{Example}

\newtheorem{rmk}[thm]{Remark}
\newtheorem{dfn}[thm]{Definition}

\numberwithin{equation}{section}

\newcommand{\frakB}{{\mathfrak B}}

\newcommand{\frakM}{{\mathfrak M}}

\newcommand{\frakS}{{\mathfrak S}}

\newcommand{\bC}{{\mathbb C}}

\newcommand{\bM}{{\mathbb M}}
\newcommand{\bN}{{\mathbb N}}

\newcommand{\bQ}{{\mathbb Q}}

\newcommand{\bZ}{{\mathbb Z}}

\newcommand{\calI}{{\mathcal I}}

\newcommand{\calO}{{\mathcal O}}

\newcommand{\rA}{{\mathrm A}}

\newcommand{\rC}{{\mathrm C}}

\newcommand{\rH}{{\mathrm H}}

\newcommand{\rL}{{\mathrm L}}
\newcommand{\rM}{{\mathrm M}}

\newcommand{\rR}{{\mathrm R}}

\newcommand{\rW}{{\mathrm W}}

\newcommand{\Zp}{{\bZ_p}}

\newcommand{\Cp}{{\bC_p}}

\newcommand{\Ainf}{{\mathrm{A_{inf}}}}

\newcommand{\End}{{\mathrm{End}}}           
\newcommand{\Gal}{{\mathrm{Gal}}}           
\newcommand{\id}{{\mathrm{id}}}             
\newcommand{\Ima}{{\mathrm{Im}}}            
\newcommand{\Ker}{{\mathrm{Ker}}}           
\newcommand{\Mod}{{\mathrm{Mod}}}           
\newcommand{\Rep}{{\mathrm{Rep}}}           
\newcommand{\RGamma}{{\mathrm{R\Gamma}}}    
\newcommand{\Rlim}{{\mathrm{R}\underleftarrow{\lim}}} 
\newcommand{\Spa}{{\mathrm{Spa}}}           
\newcommand{\Spf}{{\mathrm{Spf}}}           
\newcommand{\Vect}{{\mathrm{Vect}}}          

\newcommand{\GL}{{\mathrm{GL}}}             

\newcommand{\cris}{{\mathrm{cris}}}         
\newcommand{\cyc}{{\mathrm{cyc}}}           
\newcommand{\pd} {{\mathrm{pd}}}                           
\newcommand{\perf}{\mathrm{perf}}           

\usepackage{relsize}
\usepackage[bbgreekl]{mathbbol}
\DeclareSymbolFontAlphabet{\mathbb}{AMSb} 
\DeclareSymbolFontAlphabet{\mathbbl}{bbold}
\newcommand{\Prism}{{\mathlarger{\mathbbl{\Delta}}}} 

\newcommand{\ya}{{\rangle}}
\newcommand{\za}{{\langle}}

\begin{document}
\title{On the Hodge--Tate crystals over $\calO_K$}

\author{Yu Min\footnote{Y. M.: Department of Mathematics, Imperial College London, London SW7 2RH, UK. {\bf Email:} y.min@imperial.ac.uk} and Yupeng Wang\footnote{Y. W.: Morningside Center of Mathematics No.505, Chinese Academy of Sciences, ZhongGuancun East Road 55, Beijing, 100190, China. {\bf Email:} wangyupeng@amss.ac.cn}}

\maketitle

\begin{abstract}
Let $K$ be a complete discretely valued field of mixed characteristic $(0,p)$ with perfect residue field. We prove that the category $\Vect((\calO_K),\overline\calO_{\Prism})$ of Hodge--Tate crystals on $(\calO_K)_{\Prism}$ is equivalent to the category of pairs $(M,\phi_M)$ consisting of a finite free $\calO_K$-module $M$ and an $\calO_K$-linear endomorphism $\phi_M$ of $M$ satisfying some nilpotency condition. We then show that the cohomology of a Hodge--Tate crystal can be computed in terms of its associated linear operator $\phi_M$. After inverting $p$, we construct a fully faithful functor from the category of rational Hodge--Tate crystals to the category of semi-linear $\bC_p$-representations of $G_K$ and conjecture that the linear operator $\phi_M$ is essentially the classical Sen operator.

 As applications, under some mild assumption (e.g. $p\geq 3$), we show the crystalline Breuil--Kisin modules admit ``nilpotent connections'' and give an explicit description of prismatic crystals. The ``connection'' is conjectured predicting Hodge--Tate weights of associated crystalline representations.
\end{abstract}

\tableofcontents

\section{Introduction}\label{Sec-Introduction}
  \addtocontents{toc}{\protect\setcounter{tocdepth}{0}}

\subsection{Overview}
In their groundbreaking work \cite{BS-a}, Bhatt and Scholze constructed the prismatic cohomology, which is of ``motivic nature" in the sense that it can specialize to most existing $p$-adic cohomology theories such as \'etale, de Rham and crystalline cohomology. Similar to the theory of crystalline cohomology, one can also consider crystals on the relative/absolute prismatic site. It turns out that these coefficient objects contain important geometric and arithmetic information. For example, in \cite{BS-b}, Bhatt and Scholze proved an equivalence between the category of Laurent $F$-crystals over the absolute prismatic site of a $p$-adic formal scheme and the category of \'etale $\bZ_p$-local systems on its generic fiber (see also \cite{Wu} and \cite{MW} for related results). They also got an equivalence between the category $\Vect^{\varphi}((\calO_K)_{\Prism},\calO_{\Prism})$ of prismatic $F$-crystals over the absolute prismatic site $(\calO_K)_{\Prism}$ and the category of crystalline $\bZ_p$-representations of the absolute Galois group $G_K$ of $K$, where $K$ is a complete discretely valued field of mixed characteristic
$(0, p)$ with perfect residue field. This result was later reproved and generalized to the semi-stable case by Du and Liu in \cite{DL} by using the theory of $(\varphi,\tau)$-modules and the logarithmic prismatic site of Koshikawa \cite{Ko}.

Note that one of the basic tools of studying prismatic theory with constant coefficient is the Hodge--Tate specialization. The same spirit applies to the case of non-constant coefficients. The main goal of this paper is to study the category $\Vect((\calO_K)_{\Prism},\overline\calO_{\Prism})$ of the Hodge--Tate crystals on the absolute prismatic site $(\calO_K)_{\Prism}$. These objects are interesting in their own right and indeed provide important information about prismatic crystals. We will give an explicit description of the categories $\Vect((\calO_K)_{\Prism},\overline\calO_{\Prism})$ and $\Vect((\calO_K)_{\Prism},\overline\calO_{\Prism}[\frac{1}{p}])$ and show their close relation with semi-linear $\bC_p$-representations of $G_K$, especially the theory of Sen operator.


On the other hand, crystals on the relative prismatic site are closely related to Higgs bundles as shown in \cite{MT}, \cite{Tian} at least in the local case. This suggests the possibility of using prismatic theory to study $p$-adic non-abelian Hodge theory. This paper is actually our first step towards this purpose and is devoted to the case of $\Spf(\calO_K)$. Some results have been generalised to the higher dimensional case in our subsequent paper \cite{MW-b}.
We also want to mention that independent of our work, Bhatt and Lurie have also studied some aspects of the Hodge--Tate crystals on the absolute prismatic site. For more information, we refer to our Remark \ref{BL's work}.

\subsection{Main results}
Now we state our main results. The readers are referred to Subsection \ref{notation} for notations. 
\subsubsection{Fundamental results on Hodge--Tate crystals}
As we have mentioned in the overview, the purpose of this paper is to study Hodge--Tate crystals over $\calO_K$, which are defined as follows:
 \begin{dfn}[Definition \ref{Dfn-Hodge--Tate crystal}]
    By a {\bf Hodge--Tate crystal} on $(\calO_K)_{\Prism}$, we mean a sheaf $\bM$ of $\overline \calO_{\Prism}$-modules such that for any $(A,I)\in(\calO_K)_{\Prism}$, $\bM((A,I))$ is a finite projective $A/I$-module and that for any morphism $(A,I)\to (B,J)$, the following natural map is an isomorphism
    \[\bM((A,I))\otimes_AB\xrightarrow{\cong}\bM((B,I)).\]
    Let $\Vect((\calO_K)_{\Prism},\overline \calO_{\Prism})$ denote the category of Hodge--Tate crystals on $(\calO_K)_{\Prism}$. One can define {\bf rational Hodge--Tate crystals} by using $\overline \calO_{\Prism}[\frac{1}{p}]$ instead of $\overline \calO_{\Prism}$ and denote the corresponding category by $\Vect((\calO_K)_{\Prism},\overline \calO_{\Prism}[\frac{1}{p}])$.
 \end{dfn}
 Our first main result is an explicit description of the categories $\Vect((\calO_K)_{\Prism},\overline \calO_{\Prism})$, $\Vect((\calO_K)_{\Prism},\overline \calO_{\Prism}[\frac{1}{p}])$:
 \begin{thm}[Theorem \ref{Thm-HTC}]\label{Intro-equivalence}
   The evaluation at the Breuil--Kisin prism $(\frakS,(E))$ induces an equivalence from the category $\Vect((\calO_K)_{\Prism},\overline \calO_{\Prism})$ (resp. $\Vect((\calO_K)_{\Prism},\overline \calO_{\Prism}[\frac{1}{p}])$) of Hodge--Tate crystals (resp. rational Hodge--Tate crystals) to the category of pairs $(M,\phi_M)$, where $M$ is a finite free $\calO_K$-module (resp. finite free $K$-module) and $\phi_M$ is an $\calO_K$-linear (resp. $K$-linear) endomorphism of $M$ satisfying 
   \[\lim_{n\to+\infty}\prod_{i=0}^{n}(\phi_M+iE'(\pi)) = 0.\]
 \end{thm}
 
In fact, the Breuil--Kisin prism $(\frakS,(E))$ is a cover of the final object of ${\rm Shv}((\calO_K)_{\Prism})$. So we can interpret Hodge--Tate crystals as stratifications with respect to the \v Cech nerve
$(\frakS^{\bullet},(E))$ associated with $(\frakS,(E))$. Although $\frakS^{\bullet}$ is extremely complicated, the structure of $\frakS^{\bullet}/E$ turns out to be easier (see Lemma \ref{Equ-structure morphism on variables}).
 
Thanks to Theorem \ref{Intro-equivalence}, one can use some simple linear-algebraic objects to study Hodge--Tate crystals. In particular, we hope to read the cohomology of Hodge--Tate crystals from their associated pairs. This is our second main result.
\begin{thm}[Theorem \ref{Thm-prismatic cohomology}]\label{Intro-prism coho}
  Let $\bM\in \Vect((\calO_K)_{\Prism},\overline{\calO}_{\Prism})$ or $\Vect((\calO_K)_{\Prism},\overline \calO_{\Prism}[\frac{1}{p}])$  with associated pair $(M,\phi_M)$. Then there exists a quasi-isomorphism
  \[\rR\Gamma_{\Prism}(\bM)\simeq [M\xrightarrow{\phi_M}M].\]
\end{thm}
\begin{rmk}\label{BL's work}
  The explicit description of Hodge--Tate crystals was also obtained by Bhatt and Lurie. See \cite[Theorem 3.5.8]{BL-a} and \cite[Example 9.6]{BL-b} for relevant results. In particular, when $\calO_K = \rW(k)$ is unramified, they also computed the cohomology of Hodge--Tate crystals in \cite[Proposition 3.5.11]{BL-a}. Our approach is different from theirs and provides a uniform way to handle both unramified and ramified cases.
\end{rmk}

As an application, we get the cohomological dimension of prismatic crystals in $\Vect((\calO_K),\calO_{\Prism})$. This can be viewed as an evidence of a conjecture of Bhatt--Lurie \cite[Conjecture 10.1]{BL-b}.
\begin{thm}[Theorem \ref{non-reduced coh}]\label{Intro-coho dim}
 Let $\bM\in \Vect((\calO_K)_{\Prism},\calO_{\Prism})$ be a prismatic crystal. Then we have $\rH^i((\calO_K)_{\Prism},\bM)=0$ for all $i>1$.
\end{thm}

\subsubsection{Semi-linear $\bC_p$-representations}
 As an important motivation of our study, we show that (rational) Hodge--Tate crystals can be understood as semi-linear $\Cp$-representations of $G_K$. The latter is also known as the $v$-bundles on $\Spa(K,\calO_K)_v$. It is worthwhile to mention that rationally, the perfect absolute prismatic site of a smooth $p$-adic formal scheme is roughly the same as the $v$-site of its generic fiber. So we can use absolute perfect prismatic site as a bridge to connect Hodge--Tate crystals to $v$-bundles. 
\begin{thm}[Theorem \ref{rational crystal as representation}]\label{Intro-crystal vs rep}  
  The evaluation at the Fontaine prism $(A_{\inf}=W(\calO_{\hat K_{\cyc,\infty}^{\flat}}),(\xi))$ induces an equivalence from the category $\Vect((\calO_K)^{\perf}_{\Prism},\overline \calO_{\Prism}[\frac{1}{p}])$ to the category $\Rep_{\hat G}(\hat K_{\cyc,\infty})$ of semi-linear $\hat K_{\cyc,\infty}$-representations of $\hat G$.
\end{thm}
  Recall that by almost purity theorem due to Faltings, there is a canonical equivalence of categories
  \[\Rep_{\hat G}(\hat K_{\cyc,\infty})\simeq \Rep_{G_K}(\Cp).\]
  Combining this with Theorem \ref{Intro-crystal vs rep}, the restriction functor
  \[\Vect((\calO_K)_{\Prism},\overline \calO_{\Prism}[\frac{1}{p}])\to \Vect((\calO_K)^{\perf}_{\Prism},\overline \calO_{\Prism}[\frac{1}{p}])\]
  gives rise to a natural functor
  \[V: \Vect((\calO_K)_{\Prism},\overline \calO_{\Prism}[\frac{1}{p}])\to\Rep_{G_K}(\Cp).\]
  An intereting result is that this functor $V$ turns out to be fully faithful. Indeed, we have
  \begin{thm}[Theorem \ref{representation comes from crystal}, Proposition \ref{Prop-fully faithful}]\label{Intro-crystal as rep}
    Let $\bM\in \Vect((\calO_K)_{\Prism},\overline \calO_{\Prism}[\frac{1}{p}])$ be a rational Hodge--Tate crystal with the associated pair $(M,\phi_M)$. Then the semi-linear $\Cp$-representation $V(\bM)$ is determined by the cocycle
    \[g\mapsto (1-c(g)\lambda \pi (1-\zeta_p)E'(\pi))^{-\frac{\phi_M}{E'(\pi)}},\]
    where $\lambda\in\calO_{\hat K_{\cyc,\infty}}^{\times}$ is the image of $\frac{\xi}{E([\pi^{\flat}])}\in A_{\inf}$ modulo $\xi$ and $c(g)\in \Zp$ is determined by $g(\pi^{\flat}) = \epsilon^{c(g)}\pi^{\flat}$ for any $g\in G_K$. Moreover, the functor
    \[V: \Vect((\calO_K)_{\Prism},\overline \calO_{\Prism}[\frac{1}{p}])\to\Rep_{G_K}(\Cp)\]
    is fully faithful.
  \end{thm}
  
  By classical Sen theory, a semi-linear $\Cp$-representation $V$ is uniquely determined by a linear operator, which is known as the {\bf Sen operator} associated to $V$. We may then ask whether the Sen operator of a semi-linear $\Cp$-representation coming from a rational Hodge--Tate crystal can be read off directly from the associated linear-algebraic data provided by Theorem \ref{Intro-equivalence}. Indeed, we have the following conjecture.
  
\begin{conj}[Conjecture \ref{Conj-Sen operator}]\label{Intro-conj-Sen}
  Let $\bM$ be a rational Hodge--Tate crystal in $\Vect((\calO_K)_{\Prism},\overline \calO_{\Prism}[\frac{1}{p}])$ with the associated pair $(M,\phi_M)$ and corresponding semi-linear $\Cp$-representation $V(\bM)$ of $G_K$. Then the Sen operator of $V(\bM)$ is 
  \[\Theta_{V(\bM)} = \frac{-\phi_M}{E'(\pi)}.\]
\end{conj}
 Assuming Conjecture \ref{Intro-conj-Sen} is true, we then can get quasi-isomorphisms
 \[\rR\Gamma_{\Prism}(\bM)\otimes_K\bC_p\simeq [V(\bM)\xrightarrow{\Theta_{V(\bM)}}V(\bM)]\simeq \rR\Gamma(G_K,V(\bM))\otimes_K\bC_p\]
 where the second quasi-isomorphism is due to classical Sen theory.
 
 \begin{rmk}
  Recently, Conjecture \ref{Intro-conj-Sen} was proved by Hui Gao in \cite{Gao-b} by using the Sen theory formulated via locally analytic vectors due to Berger--Colmez \cite{BC}. When $\calO_K=W(k)$, Bhatt and Lurie have proved a similar result by using ${\rm WCart}^{\rm HT}_{\bZ_p}$ in \cite{BL-a}. 
 \end{rmk}

\subsubsection{Further applications}
 Assuming $K = K_{\cyc}\cap K_{\infty}$ (e.g. $p\geq 3$), we now give two further applications of our previous results.

The first one is to the theory of crystalline Breuil--Kisin modules, which are known to be equivalent to crystalline $\bZ_p$-representations by \cite{Gao}, \cite{EG}. Let $\frakM$ be a crystalline Breuil--Kisin module (see Definition \ref{Dfn-BK module}). One can define a ``$\tau$-connection'' \[\nabla_{\frakM}:\frakM \to M_{\inf} = \frakM\otimes_{\frakS}\Ainf\] satisfying several conditions (see Lemma \ref{tau connection on M}). Then this ``$\tau$-connection'' is topologically nilpotent in the following way.
 \begin{thm}[Theorem \ref{poly-nilpotency}]
   Let $\lambda:= \frac{\xi}{E([\pi^{\flat}])}$. Let $\frakM$ be a crystalline Breuil-Kisin module with a fixed $\frakS$-basis $e_1,\dots,e_l$. Denote by $\tilde A$ the matrix of $\nabla_{\frakM}=\frac{\tau-1}{\varphi^{-1}(\mu)u}:\frakM\to M_{\inf}$, then $\prod_{i=0}^{p-1}(\lambda^{-1}\tilde A+id_q(E))$ is $(p,[\pi^{\flat}])$-adic topologically nilpotent and $\lambda^{-1}\tilde A$ belongs to $\rM_l(k)$ modulo $(\xi,[\pi^{\flat}])$.
 \end{thm}
 To the best of our knowledge, this result is new in the theory of Breuil--Kisin modules and seems difficult to see without using prismatic theory.
 
 We also formulate the following conjecture concerning the Hodge--Tate weights of the crystalline representations through associated crystalline Breuil--Kisin modules, which can be implied by Conjecture \ref{Intro-conj-Sen}.
\begin{conj}[Conjecture \ref{Conj-weight conjecture}]
  Let $\frakM$ be a crystalline Breuil-Kisin module of rank $d$ and denote $\tilde A$ as before. Denote the Hodge-Tate weights of the associated crystalline $\Zp$-representation by $r_1\leq r_2\leq \dots\leq r_d$. Then
  \[\prod_{i=1}^d(\lambda^{-1}\tilde A+r_id_q(E))\]
  is $(p,[\pi^{\flat}])$-adic topologically nilpotent.
\end{conj}
 
Our second application is to give a linear-algebraic description of the category $\Vect((\calO_K)_{\Prism},\calO_{\Prism})$ of prismatic crystals by using a similar but more technical method used in \cite{MT}. Indeed, we define a category of prismatic $\frakS$-modules (see Definition \ref{Dfn-prismatic S module}), which can be understood as a ``no Frobenius'' version of the category of crystalline Breuil--Kisin modules. Then we get the following theorem:
 \begin{thm}[Theorem \ref{crystal as S module}]\label{Intro-prism mod}
  The evaluation at $(\frakS,(E))\in(\calO_K)_{\Prism}$ induces an equivalence from the category $\Vect((\calO_K)_{\Prism},\calO_{\Prism})$ of prismatic crystals to the category $\Mod_{\Prism}(\frakS)$ of prismatic $\frakS$-modules.
\end{thm}
 As a compensation of ignoring the Frobenius, we need to require the ``nilpotency" of the associated $\tau$-connections. A posteriori, the existence of the Frobenius ensures the ``nilpotency" of the $\tau$-connection after \cite{BS-b} and \cite{DL}. This is similar to the interesting phenomenon appearing in \cite{MT}. We hope our method of proving Theorem \ref{Intro-prism mod} may be helpful for studying the relative case.

\subsection{Notations}\label{notation} 
Let $K$ be a $p$-adic field, i.e. a complete discretely valued field of mixed characteristic $(0,p)$ with perfect residue field. Let $\bar K$ be a fixed algebraic closure of $K$ and $\bC_p=\widehat{\bar K}$. Throughout the paper, we fix a compatible system of $p$-power roots $\{\pi_n\}_{n\geq 0}$ of a uniformizer $\pi$ of $K$ in $\bar K$, i.e $\pi_0=\pi$ and $\pi_{n+1}^p=\pi_n$. We also fix a compatible system of primitive $p$-power roots $\{\zeta_{p^n}\}_{n\geq 0}$ of unity in $\bar K$. Let $\hat K_{\cyc}$ be the completion of $K_{\cyc}=\cup_n{K(\zeta_{p^n})}$ and $\hat K_{\infty}$ be the  completion of $K_{\infty}=\cup_n{K(\pi_n)}$. Let $\hat K_{\cyc,\infty}$ be the completion of $K_{\cyc,\infty}=K_{\infty}K_{\cyc}$. Denote $A_{\inf}=W(\calO_{\hat K_{\cyc,\infty}^{\flat}})$ and $\hat G={\rm Gal}(K_{\cyc,\infty}/K)$. In this paper, sometimes we need to assume $K_{\infty}\cap K_{\cyc}=K$. This condition always holds when $p\geq 3$ by \cite[Lemma 5.1.2]{Liu-a} and holds for a suitable choice of $\pi$ when $p=2$ by \cite[Lemma 2.1]{Wang}. Under this condition, we get an isomorphism $\Gamma =\Gal(K_{\cyc}/K)\cong \Gal(K_{\cyc,\infty}/K_{\infty})$. Moreover, we have an isomorphism
\[\hat G \cong \Zp\tau\rtimes\Gamma\]
where $\tau \in \Gal(K_{\cyc,\infty}/K_{\cyc})$ such that $\tau(\pi_n) = \zeta_{p^n}\pi_n$ for all $n$. For any $\gamma\in \Gamma$, $\gamma\tau\gamma^{-1} = \tau^{\chi(\gamma)}$, where $\chi:\Gamma\to\bZ_p^{\times}$ is the $p$-adic cyclotomic character. Denote by $\pi^{\flat}$ and $\epsilon$ the element $(\pi,\pi_1,\pi_2,\dots)$ and $(1,\zeta_p,\zeta_{p^2},\dots)$ in $\calO_{\hat K_{\cyc,\infty}^{\flat}}$, respectively. Then $\gamma(\epsilon) = \epsilon^{\chi(\gamma)}$ for any $\gamma\in\Gamma$ and $\tau(\pi^{\flat}) = \epsilon\pi^{\flat}$. Put $\mu = [\epsilon]-1\in A_{\inf}$. Note that there is a canonical surjection $\theta: A_{\inf}\to\calO_{\hat K_{\cyc,\infty}}$ whose kernel is principal and both $E([\pi^{\flat}])$ and $\xi = \frac{\mu}{\varphi^{-1}(\mu)}$ are its generators, where $E$ denotes the Eisenstein polynomial in
$\frakS = \rW(k)[[u]]$ of $\pi$. Then $\iota = \iota_{\pi^{\flat}}: (\frakS,(E))\to(A_{\inf},(\xi))$ induced by sending $u$ to $[\pi^{\flat}]$ is a morphism of prisms in $(\calO_K)_{\Prism}$, which turns out to be faithfully flat (for example, see \cite[Proposition 2.2.13]{EG}). When contexts are clear, we always regard $\frakS$ as a subring of $A_{\inf}$ via $\iota$. Also, when we write $(A,I)\in (\calO_K)_{\Prism}$, we always mean the object $(\Spf(\calO_K)\leftarrow \Spf(A/I)\to \Spf(A))\in (\Spf(\calO_K))_{\Prism}$. We denote by $(\calO_K)_{\Prism}^{\perf}$ the site of perfect prisms (that is, prisms whose associated Frobenius endomorphims are bijective) over $\calO_K$.

\subsection{Organizations}
 In Section 2, we collect some basic properties of the prism $(\frakS,(E))$, which will be used in next section. In Section 3, we define and study Hodge--Tate crystals. We show the desired equivalence as mentioned in Theorem \ref{Intro-equivalence} at first, compute prismatic cohomology of Hodge--Tate crystals and finally show Hodge--Tate crystals can be viewed as semi-linear $\Cp$-representations. In Section 4, we show the nilpotency of crystalline Breuil-Kisin modules. In Section 5, we give an explicit description of prismatic crystals. In Appendix, we complete the proof of Theorem \ref{Intro-prism coho} by some technical calculations.

\section*{Acknowledgments}
We would like to thank Ruochuan Liu, Matthew Morrow and Takeshi Tsuji for useful correspondences during the preparation of this work and its previous version. We want to thank Hui Gao for some explanations about the crystalline condition. We would also like to thank Bhargav Bhatt for informing us of his joint work with Jacob Lurie on absolute prismatic cohomology and suggesting us to use the name ``Hodge--Tate crystals" to replace our original one. Special thanks go to Shizhang Li for his careful reading of the early draft of this work and sharing many valuable comments. The first author is supported by China Postdoctoral Science Foundation E1900503.

 \addtocontents{toc}{\protect\setcounter{tocdepth}{2}}
\section{Preliminaries}

 In this section, we collect some useful results which are needed later. When we write $(-)_{\delta}$, we always mean the $(p,E)$-adic completion of the usual $\delta$-ideal $(-)_{\delta}$.

\begin{lem}[\emph{\cite[Lemma 2.1.7]{ALB}}]\label{ALB}
  Let $(A,I)$ be a prism and let $d\in I$ be distinguished. If $(p,d)$ is a regular sequence in $A$, then for all $r,s\geq 0, r\neq s$, the sequences $(p,\varphi^r(d))$ and $(\varphi^r(d),\varphi^s(d))$ are regular.
\end{lem}
\begin{lem}\label{BK Cover}
\begin{enumerate}
    \item The prism $(A_{\inf},\xi)$ is a cover of final objects of both topoi ${\rm Shv}((\calO_K)_{\Prism})$ and ${\rm Shv}((\calO_K)^{\perf}_{\Prism})$.
    
    \item The prism $(\frakS,(E))$ is a cover of the final object of the topos ${\rm Shv}((\calO_K)_{\Prism})$.
\end{enumerate}
\end{lem}
\begin{proof}
   Since there is a morphism of prisms $\iota:(\frakS,(E))\to(A_{\inf},(\xi))$, it is enough to show the first part of lemma. But this follows from a similar argument for the proof of \cite[Lemma 3.5]{Wu} by using \cite[Proposition 7.11]{BS-a}.
\end{proof}

 Let $(\frakS^{\bullet},(E))$ be the \v{C}ech nerve of the cover $(\frakS,(E))$ in $(\calO_K)_{\Prism}$. As argued in \cite[Example 2.6 (1)]{BS-b}, for any $n\geq 0$,
 \[\frakS^{n} = A^n\{\frac{u_0-u_1}{E(u_0)},\dots,\frac{u_0-u_n}{E(u_0)}\}_{\delta}^{\wedge_{(p,E(u_0))}},\]
 where $A^n = \rW(k)[[u_0,\dots,u_n]]$.
 By \cite[Lemma 2.24]{BS-a}, for any $0\leq i\leq n$, $E(u_i)$ and $E(u_0)$ differ by a unit in $\frakS^n$. In particular, they generate the same ideal.
 
 Before moving on, we discuss some properties of $\frakS^{\bullet}$.
\begin{lem}\label{Intersection}
  Keep notations as above. Then
  \[(\frac{u_0-u_1}{E(u_0)},\dots,\frac{u_0-u_n}{E(u_0)})_{\delta}\cap A^n = (u_0-u_1, \dots, u_0-u_n).\]
\end{lem}
\begin{proof}
  Denote the $\delta$-ideal $(\frac{u_0-u_1}{E(u_0)},\dots,\frac{u_0-u_n}{E(u_0)})_{\delta}\subset \frakS^n$ by $I^n_{\delta}$.
  Consider the following commutative diagram of morphisms of $(p,E)$-complete $\delta$-rings
  \begin{equation}\label{Diag-delta ring}
    \xymatrix@C=0.5cm{
    A^n\{X_1,\dots, X_n\}_{\delta}^{\wedge}\ar[rrrrr]^{X_j\mapsto E(u_0)X_j-(u_0-u_j),\forall j}\ar[d]^{X_j\mapsto 0,\forall j}&&&&& A^n\{X_1,\dots, X_n\}^{\wedge}_{\delta}\ar[d]\ar[rr]^{\quad X_j\mapsto 0,\forall j}&& A^n\ar[d]\\
    A^n\ar[rrrrr]&&&&& \frakS^n\ar[rr]&& \frakS^n/I^n_{\delta}.
  }
  \end{equation}
  By the construction of $\frakS^n$, both left and right squares above are push-out squares and hence so is their composition.
  In other words, we have the following push-out square
  \begin{equation*}
  \xymatrix@C=0.5cm{
    A^n\{X_1,\dots, X_n\}_{\delta}^{\wedge}\ar[rrrr]^{X_j\mapsto -(u_0-u_j),\forall j}\ar[d]^{X_j\mapsto 0,\forall j}&&&& A^n\ar[d]\\
    A^n\ar[rrrr]&&&& \frakS^n/I_{\delta}^n.
  }
  \end{equation*}
  Equivalently, we have an isomorphism $A^n/J_{\delta}^n\simeq\frakS^n/I_{\delta}^n$, where $J_{\delta}^n\subset A^n$ is the $\delta$-ideal generated by $\{u_0-u_j\}_{1\leq j\leq n}$.
  Thus we have $I_{\delta}^n\cap A^n = J_{\delta}^n.$
  Then the lemma follows from the fact that $J_{\delta}^n = (u_0-u_1, \dots, u_0-u_n)$.
\end{proof}

\begin{lem}\label{regular sequence}
  For any $n\geq j\geq 0$, the structure morphism $p_j:\frakS\rightarrow \frakS^n$ induced by
  \[
  \{0\} \to \{j\} \subset \{0,\dots,n\}
  \]
  is $(p,E(u_0))$-faithfully flat.
  In particular, $(p,E(u_0))$ is a regular sequence in $\frakS^n$.
\end{lem}
\begin{proof}
  Let $A^n$ be as in the proof of Lemma \ref{Intersection}.
  Then the sequence
  \[(E(u_0),u_0-u_1, \dots, u_0-u_n)\]
  is regular in $A^n$.
  Now, the lemma follows from \cite[Proposition 3.13]{BS-a}.
\end{proof}
\begin{lem}\label{divisiblity}
  For any $j\geq 1$, $n\geq 0$ and $m\geq 1$, $E(u_0)$ divides $\varphi^m(\delta_n(\frac{u_0-u_j}{E(u_0)}))$.
\end{lem}
\begin{proof}
  Without loss of generality, we may assume $j = 1$ and put $x = \frac{u_0-u_1}{E(u_0)}$.

  For $n=0$, we have $\varphi^m(x) = \frac{u_0^{p^m}-u_1^{p^m}}{\varphi^m(E)}$.
  As $E$ divides $u_0-u_1$, combining Lemma \ref{ALB} and Lemma \ref{regular sequence}, it also divides $\varphi^m(x)$.

  For any $n\geq 0$ and $m\geq 1$, as $p\varphi^m(\delta_{n+1}(x)) = \varphi^{m+1}(\delta_n(x))-\varphi^m(\delta_n(x))^p$, it follows from induction hypothesis that $E$ divides $p\varphi^{m}(\delta_{n+1}(x))$.
  Since $(p,E)$ is a regular sequence, $E$ also divides $\varphi^m(\delta_{n+1}(x))$ as desired.
\end{proof}
 As observed in \cite{Tian}, although the underlying $\delta$-rings of some certain prisms $(A,I)$ are complicated, their reductions modulo $I$ are much easier to handle directly. More precisely, we are going to prove the following crucial lemma on the structure of $\frakS^n/(E)$.
\begin{lem}\label{pd polynomial}
  For any $1\leq i\leq n$, denote $X_i$ the image of $\frac{u_0-u_i}{E(u_0)}\in \frakS^{n}$ modulo $(E)$. Then $\frakS^n/(E) \cong \calO_K\{X_1,\dots, X_n\}^{\wedge}_{\rm pd}$ is the free pd-polynomial ring in the variables $X_1,\dots, X_n$. Moreover, for any $0\leq i\leq n+1$, let $p_i:\frakS^n\to\frakS^{n+1}$ be the structure morphism induced by the order-preserving map \[\{0,\dots,n\}\to\{0,\dots,i-1,i+1\dots,n+1\}.\]
  Then via above isomorphisms for all $n$, we have
  \begin{equation}\label{Equ-structure morphism on variables}
      p_i(X_j) = \left\{
      \begin{array}{rcl}
           (X_{j+1}-X_1)(1-E'(\pi)X_1)^{-1}, & i=0  \\
           X_j, & j<i \\
           X_{j+1}, & 0<i\leq j.
      \end{array}
      \right.
  \end{equation}
\end{lem}
\begin{proof}
  Since $(p,E=E(u_0))$ is a regular sequence in $\frakS^n$, we see that $\frakS^n/(E)$ is $p$-torsion free.
  For any $j$ and any $m\geq 0$, we define $\overline \delta_{m,j}$ as the reduction of $\delta_m(\frac{u_0-u_j}{E})$ modulo $E$.
  Since for any $m\geq 1$,
  \[\varphi(\delta_{m-1}(\frac{u_0-u_j}{E})) = (\delta_{m-1}(\frac{u_0-u_j}{E}))^p+p\delta_m(\frac{u_0-u_j}{E}),\]
  By Lemma \ref{divisiblity}, we see that $p\overline \delta_{m+1,j} = -\overline{\delta_{m,j}}^p$ for any $m\geq 0$. By noting that $\overline \delta_{0,j} = X_j$, we conclude that for any $m\geq 0$,
  \[\overline \delta_{m,j} = (-1)^{1+p+\cdots+p^{m-1}}\frac{X_j^{p^m}}{p^{1+p+\cdots+p^{m-1}}}.\]
  Since for any $m\geq 0$, $X_j^{[p^m]}:=\frac{X_j^{p^m}}{p^m!}$\footnote{From now on, we denote by $X^{[n]}$ the $n$-th pd-power of $X$ (that is, $X^{[n]} = \frac{X^n}{n!}$). Similar remark applies to $X_1^{[n]}, X_2^{[n]}$, etc..} differs from $\overline \delta_{m,j}$ by a unit in $\Zp$, it is easy to check that $\frac{X_j^m}{m!}\in (\frakS^n/(E))[\frac{1}{p}]$ lies in $\frakS^n/(E)$.

  Let $\calO_K\{Y_1,\dots, Y_n\}^{\wedge}_{\pd}$ be the $p$-complete free pd-algebra in the variables $\{Y_j\}_{1\leq j\leq n}$ over $\calO_K$.
  Then there exists a well-defined pd-morphism
  \[\alpha: \calO_K\{Y_1,\dots, Y_n\}^{\wedge}_{\pd}\rightarrow \frakS^n/(E)\]
  sending $Y_j$ to $X_j$ for any $1\leq j\leq n$.
  The construction of $\frakS^n$ implies that $\alpha$ is surjective.
  Since both sides are $p$-complete, in order to see $\alpha$ is an isomorphism, we only need to see it is an injection after modulo $(p)$. This follows from Lemma \ref{injectivity}.

  It remains to prove the ``moreover'' part. We deal with the $i = 0$ case while the rest is obvious.
  Assume $E(u) = \sum_{i=0}^ea_iu^i$ and then
  \begin{equation*}
      \begin{split}
          E(u_0)-E(u_1) & = \sum_{i=0}^ea_i(u_0^i-u_1^i) \\
          & = \sum_{i=1}^ea_i(\sum_{j=0}^i\binom{i}{j}u_1^{i-j}(u_0-u_1)^j-u_1^i) \\
          & = \sum_{i=1}^ea_i(\sum_{j=1}^i\binom{i}{j}u_1^{i-j}(u_0-u_1)^j) \\
          & \equiv \sum_{i=1}^eia_iu_1^{i-1}(u_0-u_1) \mod (u_0-u_1)^2.
      \end{split}
  \end{equation*}
  So we get
  \begin{equation*}
      \frac{E(u_1)}{E(u_0)} = 1-\frac{E(u_0)-E(u_1)}{E(u_0)} \equiv 1 - E'(u_0)X_1 \mod E.
  \end{equation*}
  In other words, $\frac{E(u_0)}{E(u_1)}$ goes to $(1-E'(\pi)X_1)^{-1}$ modulo $E$. Now, for any $1\leq j\leq n$,
  \begin{equation*}
  \begin{split}
      p_0(X_j) & = p_0(\frac{u_0-u_j}{E(u_0)})\\
      & = \frac{u_1-u_{j+1}}{E(u_1)}\\
      & = (\frac{u_0-u_{j+1}}{E(u_0)}-\frac{u_0-u_1}{E(u_0)})\frac{E(u_0)}{E(u_1)}\\
      & \equiv (X_{j+1}-X_1)(1-E'(\pi)X_1)^{-1} \mod E.
  \end{split}
  \end{equation*}
  So we are done.
\end{proof}
\begin{lem}\label{injectivity}
   The morphism $\bar \alpha: \calO_K/p\{Y_1,\cdots,Y_n\}^{\wedge}_{\pd}\to \frakS^n/(E,p)$ is injective.
\end{lem}
\begin{proof}
  We prove the case $n=1$. The general case follows from the same argument.

  Note that $\frakS^1=A^1\{X\}^{\wedge}_{\delta}/(EX+(u_1-u_0))_{\delta}$, where $(EX+(u_1-u_0))_{\delta}$ is the closure of the ideal generated by $\{\delta_n(EX+(u_1-u_0))\}_{n\in \bN}$.

  We claim that $\delta_n(EX+(u_1-u_0))\equiv \delta_n(EX) \mod{(u_1-u_0)}$. We prove this by induction. When $n=0$, this is trivial. Assume the claim is true when $n=m$. Let $Z=EX+(u_1-u_0)$. Then we have
  \[
 \delta_{m+1}(Z)=\delta(\delta_m(Z))=\frac{\varphi(\delta_m(Z))-(\delta_m(Z))^p}{p}.
  \]
  Write $\delta_m(Z)=\delta_m(EX)+(u_1-u_0)Z_1$ for some $Z_1\in A^1\{X\}^{\wedge}_{\delta}$. Then
  \begin{equation*}
      \begin{split}
  &\frac{\varphi(\delta_m(Z))-(\delta_m(Z))^p}{p}= \frac{\varphi(\delta_m(EX)+(u_1-u_0)Z_1)-(\delta_m(EX)+(u_1-u_0)Z_1)^p}{p}\\
  \equiv & \frac{\varphi(\delta_m(EX))-(\delta_m(EX))^p+\varphi((u_1-u_0)Z_1)-((u_1-u_0)Z_1)^p}{p} \mod{(u_1-u_0)}\\
  \equiv & \delta_{m+1}(EX)+\delta((u_1-u_0)Z_1) \mod{(u_1-u_0)}\\
  \equiv &\delta_{m+1}(EX) \mod{(u_1-u_0)}.
      \end{split}
  \end{equation*}
  The last congruence follows from $\delta(u_1-u_0)\equiv 0\mod{(u_1-u_0)}$.

  So we can see
  \begin{equation*}
      \begin{split}
           \frakS^1/(E,p)= & A^1\{X\}^{\wedge}_{\delta}/(p,E,(u_1-u_0),\delta_1(EX),\cdots,\delta_n(EX),\cdots)\\
           = & \frakS\{X\}^{\wedge}_{\delta}/(p,E,\delta_1(EX),\cdots,\delta_n(EX),\cdots).
      \end{split}
  \end{equation*}
 Let $(\frakS\langle T\rangle^1, E)$ be the self-product of $(\frakS\langle T\rangle, E)$ in the category $(\calO_K\langle T\rangle/(\frakS,E))_{\Prism}$. Then by similar arguments as above, we have
 \begin{equation*}
 \begin{split}
         \frakS\langle T\rangle^1/(E,p)=&\frakS\langle T_0,T_1\rangle\{X\}^{\wedge}_{\delta}/(p,E,(T_1-T_0),\delta_1(EX),\cdots,\delta_n(EX),\cdots)\\
         =& \frakS\langle T\rangle\{X\}^{\wedge}_{\delta}/(p,E,\delta_1(EX),\cdots,\delta_n(EX),\cdots).
 \end{split}
 \end{equation*}

 So in particular, we have $\frakS^1/(E,p)\otimes_{\calO_K/p}\calO_K/p[T]\cong \frakS\langle T\rangle^1/(E,p)$. Now we consider the following commutative diagram

 \begin{equation*}
       \xymatrix@C=0.5cm{
    \calO_K/p\{Y_1\}^{\wedge}_{\pd}\ar[rrrr]^{}\ar[d]^{\otimes_{\calO_K/p}\calO_K/p[T]}&&&& \frakS^1/(p,E)\ar[d]^{\otimes_{\calO_K/p}\calO_K/p[T]}\\
    \calO_K/p[T]\{Y_1\}^{\wedge}_{\pd}\ar[rrrr]&&&& \frakS\langle T\rangle^1/(p,E).
  }
  \end{equation*}
 The vertical maps are injective due to the faithful flatness of the map $\calO_K/p\to \calO_K/p[T]$. The bottom map is injective due to \cite[Proposition 4.3]{Tian}. So we are done.
\end{proof}
\begin{exam}\label{Exam-structure morphism}
  Let $\calO_K\{X\}^{\wedge}_{\rm pd} = \frakS^1/E$ and $\calO_K\{X_1,X_2\}^{\wedge}_{\rm pd} = \frakS^2/E$. Denote by $p_{ij}:\frakS^1/E\to \frakS^2/E$ the structure morphism induced by $\{0,1\}\to\{i,j\}\subset \{0,1,2\}$. Then we have
  \begin{equation*}
      \begin{split}
         & p_{01}(X) = X_1,\\
         & p_{02}(X) = X_2,\\
         & p_{12}(X) = (X_2-X_1)(1-E'(\pi)X_1)^{-1}.
      \end{split}
  \end{equation*}
\end{exam}
 
 Now, we define a ``$\tau$-connection'' on $\frakS$.
 For any $f(u) = \sum_{n\geq 0}a_nu^n\in \frakS$, define
  \[
  d_q(f): = \sum_{n\geq 0}a_n[n]_qu^{n-1}\in A_{\inf},
  \]
  where $[n]_q = \frac{[\epsilon]^n-1}{[\epsilon]-1}$.
  Then it is easy to check that for any $f\in\frakS$,
\begin{equation}\label{Equ-connection on coefficients}
  (\tau-1)(f) = \mu ud_q(f).
\end{equation}
  Clearly, for any $f,g\in \frakS$, we have that
  $d_q(fg) = d_q(f)\tau(g)+fd_q(g)$
  and that $d_q(f)$ goes to $f'(\pi)$ modulo $\xi$.
 \begin{exam}\label{Exam-differential of E}
   For any $h\geq 1$, we have
   \begin{equation*}
   \begin{split}
       d_q(E^h) & = \frac{\tau(E)^h-E^h}{\mu u}\\
       & = \frac{1}{\mu u}((E+\mu ud_q(E))^h-E^h)\\
       & = \sum_{i=0}^{h-1}\binom{h}{i}E^i(\mu u)^{h-1-i}d_q(E)^{h-i}\\
       & = hE^{h-1}d_q(E)+\varphi^{-1}(\mu)u\sum_{i=0}^{h-2}\binom{h}{i}E^i\xi^{h-1-i}(\varphi^{-1}(\mu)u)^{h-2-i}d_q(E)^{h-i}.
   \end{split}
   \end{equation*}
   In particular, $\frac{d_q(E^h)}{E^{h-1}}\in A_{\inf}$ and $\frac{d_q(E^h)}{E^{h-1}}\equiv hd_q(E)\mod \varphi^{-1}(\mu)uA_{\inf}$.
 \end{exam}
  For any $n\geq 1$, let $\tau\in\hat G$ act on the first factor of $\frakS^{n}$. Then for any $0\leq j\leq n$, we have
\begin{equation}\label{Equ-connection on X}
\begin{split}
  (\tau-1)(\frac{u_0-u_j}{E(u_0)}) & = \frac{[\epsilon]u_0-u_j}{\tau(E(u_0))} -\frac{u_0-u_j}{E(u_0)}\\
  & = \frac{\mu u_0}{\tau(E(u_0))}-\frac{u_0-u_j}{E(u_0)}\frac{(\tau-1)(E(u_0))}{\tau(E(u_0))}\\
  & = \frac{\mu u_0}{\tau(E(u_0))}(1-\frac{u_0-u_j}{E(u_0)}d_q(E)).
\end{split}
\end{equation}
\begin{lem}\label{tau connection}
  The $\rW(k)$-linear morphism $\nabla:=\frac{\tau-1}{\varphi^{-1}(\mu)u}:\frakS\rightarrow \Ainf$ is well-defined on $\frakS$ such that for any $f,g\in \frakS$, $\nabla(f) = \xi d_q(f)$ and that
  \[\nabla(fg) = \nabla(f)\tau(g)+f\nabla(g).\]
  If we regard $\frakS^1$ as an $\frakS$-algebra via the morphism $p_1:\frakS\to \frakS^1$ induced by the map
  \[\{0\}\to\{0\}\subset\{0,1\},\]
  then $\nabla$ extends to a $\rW(k)$-linear morphism
  $\nabla:\frakS^1\to \tilde A_{\inf}^1$
  such that for any $f\in\frakS$ and any $x\in \frakS^1$,
  \begin{equation}\label{Equ-tau connection}
      \nabla(fx) = \tau(f)\nabla(x)+\nabla(f)x = f\nabla(x)+x\nabla(f)+\mu ud_q(f)\nabla(x),
  \end{equation}
  where $\tilde A_{\inf}^1 = A_{\inf}[[u_1]]\{\frac{u_0-u_1}{E(u_0)}\}^{\wedge}_{\delta}$ is the coproduct of $(A_{\inf},(\xi))$ and $(\frakS,(E))$ in $(\calO_K)_{\Prism}$ via identifying $u_0$ with $[\pi^{\flat}]$. The $\nabla$ is compatible with $\varphi$ in the following sense
  \begin{equation}\label{Equ-nabla commutes with phi}
      \nabla\varphi = \xi u_0^{p-1}\varphi\nabla.
  \end{equation}
\end{lem}
\begin{proof}
  The only non-trivial part is that $\nabla$ is well-defined on $\frakS^1$.
  Denote $\frac{u_0-u_1}{E(u_0)}$ by $X$. It suffices to check $(\tau-1)(\delta_n(X))$ is divided by $\varphi^{-1}(\mu)u_0$.

  The $n = 0$ case follows from (\ref{Equ-connection on X}) directly since $\xi$ and $E(u_0)$ generate the same ideal in $A_{\inf}$. Now, assume we have shown that$\varphi^{-1}(\mu)u_0$ divides $(\tau-1)(\delta_n(X))$ for some $n\geq 0$. Then for $n+1$, since $\varphi$ commutes with $\tau$, we have
  \begin{equation}\label{Equ-nabla commute with phi}
  \begin{split}
      (\tau-1)(p\delta_{n+1}(X)) & = (\tau-1)\varphi(\delta_n(X)) - (\tau-1)(\delta_n(X)^p)\\
      & = \varphi((\tau-1)\delta_n(X)) - (((\tau-1)(\delta_n(X))+\delta_n(X))^p-\delta_n(X)^p) \\
      & = \varphi(\varphi^{-1}(\mu)u_0\nabla(\delta_n(X)))-\sum_{i=0}^{p-1}\binom{p}{i}(\varphi^{-1}(\mu)u_0)^{p-i}\nabla(\delta_n(X))^{p-i}\delta_n(X)^i\\
      & = \mu u_0^p\varphi(\nabla(\delta_n(X)))-\sum_{i=0}^{p-1}\binom{p}{i}(\varphi^{-1}(\mu)u_0)^{p-i}\nabla(\delta_n(X))^{p-i}\delta_n(X)^i \\
      & \equiv \mu u_0^p\varphi(\nabla(\delta_n(X))) \mod \varphi^{-1}(\mu)u_0.
  \end{split}
  \end{equation}
  So $\varphi^{-1}(\mu)u_0$ divides $p(\tau-1)(\delta_{n+1}(X))$. By a similar argument used in the proof of Lemma \ref{regular sequence}, we see $(p,\varphi^{-1}(\mu)u_0)$ is a regular sequence in $\tilde A_{\inf}^1$. This implies $\varphi^{-1}(\mu)u_0$ divides $(\tau-1)(\delta_{n+1}(X))$ as desired.
\end{proof}
\begin{cor}\label{reduced tau connection}
  The $\nabla$ induces an $\calO_K$-linear ``$\tau$-connection''
  \[\nabla:\calO_K\{X\}^{\wedge}_{\pd}\to\calO_{\hat K_{\cyc,\infty}}\{X\}^{\wedge}_{\pd},\]
  such that
  \begin{equation}\label{Equ-tau on X}
      \nabla(X) =  \lambda(1+(\zeta_p-1)\pi\lambda E'(\pi))^{-1}(1-XE'(\pi))
  \end{equation}
  and that for any $f,g\in \calO_K\{X\}^{\wedge}_{\pd}$,
  $\nabla(fg) = \nabla(f)\tau(g)+f\nabla(g)$, where $\lambda$ denotes the image of $\frac{\xi}{E([\pi^{\flat}])}$ modulo $\xi$.
\end{cor}
\begin{proof}
  A similar argument used in the proof of Lemma \ref{pd polynomial} shows $\tilde A^1_{\inf}/\xi \cong \calO_{\hat K_{\cyc,\infty}}\{X\}^{\wedge}_{\pd}$.
  By Lemma \ref{tau connection}, especially (\ref{Equ-tau connection}), we see that for any $f\in\frakS$ and any $x\in \frakS^1$,
  \[\nabla(fx)\equiv f\nabla(x) \mod \xi.\]
  This implies $\nabla$ is well-defined (by letting $f = E$) and $\calO_K$-linear (by letting $f = u$). It remains to check (\ref{Equ-tau on X}). By (\ref{Equ-connection on X}), we see
  \[\nabla(X) = \frac{\xi}{E(u_0)}\frac{E(u_0)}{\tau(E(u_0))}(1-\frac{u_0-u_1}{E(u_0)}d_q(E)).\]
  Note that
  \begin{equation}\label{Equ-tau E over E}
      \begin{split}
          \frac{\tau(E(u_0))}{E(u_0)}& = 1+\frac{(\tau-1)(E(u_0))}{E(u_0)}\\
          & = 1+\frac{\mu ud_q(E)}{E(u_0)}\\
          & = 1+\varphi^{-1}(\mu)u_0\lambda d_q(E).
      \end{split}
  \end{equation}
  We get $\nabla(X) = \lambda(1+(\zeta_p-1)\pi\lambda E'(\pi))^{-1}(1-XE'(\pi))$ as desired.
\end{proof}

\begin{exam}\label{Exam-nabla on pd algebra}

    For the further use, we give an explicit description of $\nabla(X^{[n]})$ for any $n\geq 0$. Note that $\calO_K\{X\}^{\wedge}_{\pd}$ is $p$-torsion free. So for any $f\in \calO_K\{X\}^{\wedge}_{\pd}$, $\pi(\zeta_p-1)\nabla(f) = (\tau-1)(f)$. Now, fix an $n\geq 1$ and then we have
    \begin{equation*}
        \begin{split}
            \tau(X^{[n]}) & = (X+(\zeta_p-1)\nabla(X))^{[n]}\\
            & = \sum_{i=0}^n\pi^{n-i}\frac{(\zeta_p-1)^{n-i}}{(n-i)!}\nabla(X)^{n-i}X^{[i]}\\
            & = X^{[n]}+\pi(\zeta_p-1)\sum_{i=0}^{n-1}\pi^{n-i-1}\frac{(\zeta_p-1)^{n-i-1}}{(n-i)!}\nabla(X)^{n-i}X^{[i]}.
        \end{split}
    \end{equation*}
    So we deduce that for any $n\geq 1$,
    \begin{equation}\label{Equ-Nabla on X^n}
        \nabla(X^{[n]}) = \sum_{i=0}^{n-1}\pi^{n-i-1}\frac{(\zeta_p-1)^{n-i-1}}{(n-i)!}\nabla(X)^{n-i}X^{[i]}.
    \end{equation}
    
    Recall that for any $m\geq 1$, $\nu_p(m!)\leq \frac{m-1}{p-1} = \nu_p((\zeta_p-1)^{m-1})$. So we deduce from (\ref{Equ-Nabla on X^n}) that 
    \begin{equation}\label{Equ-Nabla on X^n-mod pi}
        \nabla(X^{[n]}) \equiv X^{[n-1]}\nabla(X) \mod \pi.
    \end{equation}
    In other words, $\nabla$ behaves like a connection modulo $\pi$.
\end{exam}

\section{Hodge--Tate crystals}

 In this section, we study the category $\Vect((\calO_K)_{\Prism},\overline{\calO}_{\Prism})$ of Hodge--Tate crystals and investigate the cohomology of both Hodge--Tate crystals and prismatic crystals. We will give an explicit complex computing the former. For the latter, there will be a higher vanishing theorem. We remark that all the results on the Hodge--Tate crystals in this section hold true by the same arguments for crystals over $\overline \calO_{\Prism}[\frac{1}{p}]$, which we shall call {\bf rational Hodge--Tate crystals}.
\subsection{Hodge--Tate crystals}
  \begin{dfn}[Hodge--Tate crystals]\label{Dfn-Hodge--Tate crystal}
    By a {\bf Hodge--Tate crystal} on $(\calO_K)_{\Prism}$, we mean a sheaf $\bM$ of $\overline \calO_{\Prism}$-modules such that for any $(A,I)\in(\calO_K)_{\Prism}$, $\bM((A,I))$ is a finite projective $A/I$-module and that for any morphism $(A,I)\to (B,J)$, the following natural map is an isomorphism
    \[\bM((A,I))\otimes_AB\xrightarrow{\cong}\bM((B,I)).\]
    Denote by $\Vect((\calO_K)_{\Prism},\overline \calO_{\Prism})$ the category of Hodge--Tate crystals on $(\calO_K)_{\Prism}$.
  \end{dfn}

  \begin{dfn}[Stratification]\label{Dfn-stratification}
    Let $B^{\bullet}$ be a cosimplicial ring with $B:=B^0$. Let $M$ be a $B$-module. By a {\bf stratification} on $M$ with respect to $B^{\bullet}$, we mean a $B^1$-linear isomorphism \[\varepsilon: M\otimes_{B,p_0}B^1\to M\otimes_{B,p_1}B^1.\]
    Here, for $i=0,1$, $p_i$ is induced by the map
    $\{0\} \to\{1 - i\}\subset\{0,1\}$.
   
    We say a stratification $\varepsilon$ satisfies the {\bf cocycle condition} if 
    \begin{enumerate}
    \item $p^*_{12}(\varepsilon)p^*_{01}(\varepsilon) = p_{02}^*(\varepsilon)$, where for $ij\in \{01,02,12\}$, $p_{ij}$ is defined in Example \ref{Exam-structure morphism}.

    \item $\Delta^*(\varepsilon)=\id_M$, where $\delta:B^1\to B^0=B$ denotes the degeneracy morphism.
    \end{enumerate}
  \end{dfn}
  \begin{rmk}\label{Rmk-Tian}
     In Definition \ref{Dfn-stratification}, assume $B^{\bullet}$ is the underlying cosimplicial ring of the \v Cech nerve $(B^{\bullet},IB^{\bullet})$ induced by a cover $(B,I)$ of the final object in ${\rm Shv}((\calO_K)_{\Prism})$. 
     Then by \cite[Remark 3.5]{Tian}, for any stratification $\varepsilon$ on $M$ satisfying Item 1 of Definition \ref{Dfn-stratification}, Item 2 of Definition \ref{Dfn-stratification} is satisfied automatically.
  \end{rmk}
  The following Lemma is well-known.
  \begin{lem}[\emph{\cite[Corollary 3.9]{MT}},\emph{\cite[Proposition 2.7]{BS-b}}]\label{MT}
    Let $(B,I)$ be a cover of the final object of the topos $Shv((\calO_K)_{\Prism})$ with the associated cosimplicial prism $(B^{\bullet},IB^{\bullet})$. Then the category of Hodge--Tate crystals on $(\calO_K)_{\Prism}$ is equivalent to the category of finite projective $B/I$-modules $M$ on which there is a stratification satisfying the cocycle condition.
  \end{lem}
  Applying the lemma to $(B,I)=(\frakS,(E))$, we get the following corollary.
  \begin{cor}\label{crystal is stratification}
    The category of Hodge--Tate crystals on $(\calO_K)_{\Prism}$ is equivalent to the category of finite free $\calO_K$-modules $M$ on which there is a stratification with respect to $\frakS^{\bullet}/(E)$ satisfying the cocycle condition.
  \end{cor}
  
  Let $M$ be a finite free $\calO_K$-module endowed with a stratification $\varepsilon$ with respect to $(\frakS^{\bullet},(E))$. We want to give an explicit description on the cocycle condition of $\varepsilon$. Tor this purpose, we fix an $\calO_K$-basis $e_1,\dots,e_l$ of $M$ and write
  \[\varepsilon(\underline e) = \underline e\cdot \sum_{n\geq 0}A_nX_1^{[n]},\]
  where $A_n\in \rM_l(\calO_K)$ for any $n\geq 0$.
  Then we have that
  \begin{equation*}
     \begin{split}
        p_{12}^*(\varepsilon)\circ p_{01}^*(\varepsilon)(\underline e) & =  p_{12}^*(\varepsilon)(\underline e\sum_{n\geq 0}A_nX_1^{[n]})\\
        & = \sum_{n,m\geq 0}\underline e A_mA_n(X_2-X_1)^{[m]}(1-E'(\pi)X_1)^{-m}X_1^{[n]}\\
        & = \sum_{n\geq 0}\underline eX_2^{[n]}\sum_{i,j\geq 0}A_{n+i}A_j(-1)^i(1-E'(\pi)X_1)^{-n-i}X_1^{[i]}X_1^{[j]},
     \end{split}
  \end{equation*}
  that
  \begin{equation*}
      p_{02}^*(\varepsilon)(\underline e) = \sum_{n\geq 0}\underline eA_nX_2^{[n]},
  \end{equation*}
  and that 
  \begin{equation*}
      \Delta^*(\varepsilon)(\underline e) = \underline eA_0.
  \end{equation*}
  So the stratification satisfies the cocycle condition if and only if $A_0=I$ and for any $n\geq 0$, 
  \begin{equation}\label{Equ-determine coefficients-I}
      \sum_{i,j\geq 0}A_{n+i}A_j(-1)^i(1-E'(\pi)X_1)^{-n-i}X_1^{[i]}X_1^{[j]} = A_n.
  \end{equation}
\begin{lem}\label{determine coefficients}
  Let $(A_n)_{n\geq 0}$ be a sequence in $\rM_l(\calO_K)$ and $\alpha\in\calO_K$. If $A_0=I$, then the following are equivalence:
  \begin{enumerate}
      \item For any $n\geq 0$, $\sum_{i,j\geq 0}A_{n+i}A_j(-1)^i(1-\alpha X)^{-n-i}X^{[i]}X^{[j]} = A_n.$
      \item For any $n\geq 0$, $A_{n+1} = \prod_{i=0}^n(i\alpha+A_1)$.
  \end{enumerate}
\end{lem}
\begin{proof}
  Note that $X$ is topologically nilpotent in $\calO_K\{X\}^{\wedge}_{\rm pd}$, so we have
  \begin{equation*}
  \begin{split}
  & \sum_{i,j\geq 0}A_{n+i}A_j(-1)^i(1-\alpha X)^{-n-i}X^{[i]}X^{[j]}\\
  = & \sum_{i,j\geq 0}A_{n+i}A_j(-1)^i\sum_{l\geq 0}(-\alpha)^l\binom{-n-i}{l}X^lX^{[i]}X^{[j]} \\
  = & \sum_{i,j,l\geq 0}A_{n+i}A_j(-1)^i\alpha^l\binom{n+i-1+l}{l}X^lX^{[i]}X^{[j]} \\
  = & \sum_{N\geq 0}X^{[N]}\sum_{i+j+l=N}A_{n+i}A_j(-1)^i\alpha^l\binom{n+i-1+l}{l}\frac{N!}{i!j!}
  \end{split}
  \end{equation*}
  \begin{enumerate}
      \item Assume Item 1 is true. Then the coefficient of $X^{[1]}$ must be zero. So for any $n\geq 1$, we get
      \[-A_{n+1}+A_nA_1+A_nn\alpha = 0\]
      Equivalently, we have $A_{n+1} = A_n(n\alpha+A_1)$, so Item 2 follows by induction.

      \item Assume Item 2 is true. We only need to show that the equation (\ref{Equ-determine coefficients-I}) holds for any $n\geq 0$.
      Since all $A_i$'s commute and for any $n$, there are only finitely many $A_i$'s appearing in the coefficient of $X^n$, so we may assume $A_{n+1} = \prod_{i=0}^n(i\alpha+Y)$ and check the desired equality (\ref{Equ-determine coefficients-I}) in the ring $K[[X,Y]]$ of formal series.

      Put $F(X,Y) = (1-\alpha X)^{-\frac{Y}{\alpha}}$. It is well-defined in $K[[X,Y]]$ and of the form
      \begin{equation}\label{Equ-determine coefficients-II}
      \begin{split}
      F(X,Y) & = (1-\alpha X)^{-\frac{Y}{\alpha}}\\
      & = \sum_{i\geq 0}\binom{-\frac{Y}{\alpha}}{i}(-\alpha X)^i\\
      & = \sum_{i\geq 0}(-1)^i\alpha^i(-\frac{Y}{\alpha})(-\frac{Y}{\alpha}-1)\cdots(-\frac{Y}{\alpha}-i+1)X^{[i]}\\
      & = \sum_{i\geq 0}Y(Y+\alpha)\cdots(Y+(i-1)\alpha)X^{[i]} \\
      & = \sum_{i\geq 0}A_iX^{[i]}
      \end{split}
      \end{equation}
      Applying $\partial_X^n:=\frac{\partial^n}{\partial X^n}$ to both sides of the above equality, we get
      \begin{equation}\label{Equ-determine coefficients-III}
      \begin{split}
          \sum_{i\geq 0}A_{n+i}X^{[i]} & = \partial^n_XF(X,Y)\\
          & = Y(Y+\alpha)\cdots(Y+(n-1)\alpha)(1-\alpha X)^{-\frac{Y}{\alpha}-n}\\
          & =A_n(1-\alpha X)^{-\frac{Y}{\alpha}-n}.
      \end{split}
      \end{equation}
  \end{enumerate}
  In particular, after replacing $X$ by $\frac{-X}{1-\alpha X}$ in (\ref{Equ-determine coefficients-III}), we get
  \begin{equation}\label{Equ-determine coefficients-IV}
      \sum_{i\geq 0}A_{n+i}(-1)^i(1-\alpha X)^{-i}X^{[i]} = A_n(1-\alpha X)^{\frac{Y}{\alpha}+n}.
  \end{equation}
  Now, the desired equation (\ref{Equ-determine coefficients-I}) follows from (\ref{Equ-determine coefficients-II}) and (\ref{Equ-determine coefficients-IV}) directly.
\end{proof}
 So we have the following proposition.
\begin{prop}\label{matrix of stratification}
  Keep notations as above. Then $(M,\varepsilon)$ induces a crystal $\bM\in\Vect((\calO_K)_{\Prism},\overline \calO_{\Prism})$ if and only if
  $\lim_{n\to+\infty}\prod_{i=0}^n(iE'(\pi)+A_1) = 0$. In this case, we have $A_0 = I$ and for any $n\geq 0$, $A_{n+1} = \prod_{i=0}^n(iE'(\pi)+A_1)$.
\end{prop}
\begin{proof}
  By Corollary \ref{crystal is stratification} and the previous calculations, we see that $(M,\varepsilon)$ induces a crystal if and only if $A_0 = I$ and $\lim_{n\to+\infty}A_n = 0$ such that for any $n\geq 0$,
  \[\sum_{i,j\geq 0}A_{n+i}A_j(-1)^i(1-E'(\pi)X_1)^{-n-i}X_1^{[i]}X_1^{[j]} = A_n.\]
  Then the proposition follows from the Lemma \ref{determine coefficients}.
\end{proof}
 In summary, we have shown the following theorem, which gives an explicit description of Hodge--Tate crystals over $\calO_K$:
 \begin{thm}\label{Thm-HTC}
   The evaluation at $(\frakS,(E))$ induces an equivalence from the category $\Vect((\calO_K)_{\Prism},\overline \calO_{\Prism})$ of Hodge--Tate crystals to the category of pairs $(M,\phi_M)$, where $M$ is a finite free $\calO_K$-module and $\phi_M$ is an $\calO_K$-linear endomorphism of $M$ satisfying 
   \[\lim_{n\to+\infty}\prod_{i=0}^{n}(\phi_M+iE'(\pi)) = 0.\]
 \end{thm}
 \begin{rmk}\label{Rmk-stratification and Sen}
  For a Hodge--Tate crystal $\bM$ with the associated pair $(M,\phi_M)$ and the induced stratification $\varepsilon$ on $M$, as suggested in Lemma \ref{determine coefficients}, we have
  \[\varepsilon = (1-E'(\pi)X)^{-\frac{\phi_M}{E'(\pi)}}.\]
  The right hand side is well-defined since $\lim_{n\to+\infty}\prod_{i=0}^n(\phi_M+iE'(\pi)) = 0$.
\end{rmk}

\subsection{Absolute prismatic cohomology for Hodge--Tate crystals}\label{SSec-compare cohomologies}
  In this section, we will compute the prismatic cohomology $\RGamma_{\Prism}(\bM)$ for a Hodge--Tate crystal $\bM\in \Vect((\calO_K)_{\Prism},\overline \calO_{\Prism})$.
  We first show that we can compute $\RGamma_{\Prism}(\bM)$ in terms of the \v Cech-Alexander complex. To see this, we begin with the following lemma.
 \begin{lem}\label{vansihing lemma}
   Let $\bM$ be a Hodge--Tate crystal in $\Vect((\calO_K)_{\Prism},\overline \calO_{\Prism})$. Then for any $\frakB = (B,J)\in(\calO_K)_{\Prism}$ and any $i\geq 1$, we have
   \[\rH^i(\frakB,\bM) = 0.\]
 \end{lem}
 \begin{proof}
   This follows from the same argument in the proof of \cite[Lemma 3.11]{Tian}.
 \end{proof}
 \begin{cor}[\v Cech-Alexander complex]
   Let $\bM\in \Vect((\calO_K)_{\Prism},\overline \calO_{\Prism})$ be a Hodge--Tate crystal. Then the prismatic cohomology $\RGamma_{\Prism}(\bM)$ can be computed by the \v Cech-Alexander complex
   \begin{equation}\label{Equ-Cech-Alexander complex}
       \bM(\frakS,(E))\xrightarrow{d^0}\bM(\frakS^1,(E))\xrightarrow{d^1} \bM(\frakS^2,(E))\to\cdots,
   \end{equation}
   where for any $n\geq 0$, $d^n=\sum_{i=0}^{n+1}(-1)^ip_i$ is induced by $p_i:\frakS^n\to\frakS^{n+1}$ defined in Lemma \ref{pd polynomial}.
 \end{cor}
 \begin{proof}
   Since $(\frakS,(E))$ is a cover of the final object of the topos ${\rm Shv}((\calO_K)_{\Prism})$, the result follows from Lemma \ref{vansihing lemma} combined with the \v Cech-to-derived spectral sequence.
 \end{proof}
 We want to express the \v Cech-Alexander complex (\ref{Equ-Cech-Alexander complex}) in an explicit way. Let $(M,\phi_M)$ be the pair associated to $\bM$ in the sense of Theorem \ref{Thm-HTC}. As in the previous subsection, we fix an $\calO_K$-basis $e_1,\dots, e_l$ of $M$ and denote by $A$ the matrix of $\phi_M$.

 From now on, we denote by $q_i:\frakS\to\frakS^n$
 the structure morphism induced by the map $\{0\}\to\{i\}\subset\{0,\dots,n\}$. Then one can identify $\bM(\frakS^n,(E))$ with $M\otimes_{\frakS,q_0}\frakS^n$ via the canonical isomorphism
 \[M\otimes_{\frakS,q_0}\frakS^n\xrightarrow{\cong}\bM(\frakS^n,(E)).\]
 By Lemma \ref{pd polynomial} and Proposition \ref{matrix of stratification}, for any $1\leq i\leq n$ and any 
 \[\vec f(X_1,\dots,X_n)\in(\calO_K\{X_1,\dots,X_n\}^{\wedge}_{\pd})^l\cong(\frakS^n/(E))^{l},\] 
 we have
 \begin{equation}\label{Equ-differential in explicit way-I}
 \begin{split}
     &p_0(\underline e\vec f(X_1,\dots,X_n))\\
     =&\underline e(1-E'(\pi) X_1)^{-\frac{A}{E'(\pi)}}\vec f((1-E'(\pi)  X_1)^{-1}(X_2-X_1),\dots,(1-E'(\pi)X_1)^{-1}(X_{n+1}-X_1))
 \end{split}
 \end{equation}
 and
 \begin{equation}\label{Equ-differential in explicit way-II}
     p_i(\underline e\vec f(X_1,\dots,X_n)) = \underline e\vec f(X_1,\dots,X_{i-1},X_{i+1},\dots,X_{n+1}).
 \end{equation}
 So the \v Cech-Alexander complex reduces to the following complex $\check \rC\rA(M):$
 \begin{equation}\label{Equ-Cech-Alexander complex-II}
     M\xrightarrow{d^0}M\otimes_{\calO_K}\calO_K\{X_1\}^{\wedge}_{\pd}\xrightarrow{d^1}M\otimes_{\calO_K}\calO_K\{X_1,X_2\}^{\wedge}_{\pd}\to \cdots
 \end{equation}
 with the differentials
 \begin{equation}\label{Equ-differential in explicit way-III}
 \begin{split}
     &d^n(\underline e\vec f(X_1,\dots,X_n))\\ = & \underline e(1-E'(\pi) X_1)^{-\frac{A}{E'(\pi)}}\vec f((1-E'(\pi)  X_1)^{-1}(X_2-X_1),\dots,(1-E'(\pi)X_1)^{-1}(X_{n+1}-X_1))\\
     & + \sum_{i=1}^{n+1}(-1)^i\underline e\vec f(X_1,\dots,X_{i-1},X_{i+1},\dots,X_{n+1}).
 \end{split}
 \end{equation}
 Define
 \begin{equation}\label{Equ-A(X)-I}
     F_A(X) = X+\sum_{n\geq 1}\prod_{i=0}^{n-1}(A+E'(\pi)+iE'(\pi))X^{[n+1]}.
 \end{equation}
 Then it is well-defined in $\rM_l(\calO_K\{X\}^{\wedge}_{\pd})$ such that
 \begin{equation}\label{Equ-A(X)-II}
   I+A\cdot F_A(X) = (1-E'(\pi)X)^{-\frac{A}{E'(\pi)}}
 \end{equation}
 and that
 \begin{equation}\label{Equ-A(X)-III}
     \frac{\partial F_A}{\partial X}(X) = (1-E'(\pi)X)^{-\frac{A}{E'(\pi)}-1}.
 \end{equation}
 In particular, the following diagram
 \begin{equation*}
     \xymatrix@C=0.45cm{
     M\ar[d]^{\id_M}\ar[rr]^A&& M\ar[d]^{F_A(X_1)}\\
     M\ar[rr]^{d^0}&&M\{X_1\}^{\wedge}_{\pd}
     }
 \end{equation*}
 commutes and induces a natural morphism of complexes
 \begin{equation}\label{Equ-morphism of complex}
     \rho_M:[M\xrightarrow{A}M]\to \check \rC\rA(M).
 \end{equation}
 
 Now, we state and prove the main theorem in this subsection.
 \begin{thm}\label{Thm-prismatic cohomology}
     Keep notations as above. Then the morphism $\rho_M$ in (\ref{Equ-morphism of complex}) is a quasi-isomorphism. In particular, there exists a natural quasi-isomorphism
     \[\rR\Gamma_{\Prism}(\bM)\simeq [M\xrightarrow{\phi_M}M].\]
 \end{thm}
 \begin{proof}
   It is enough to show that $\rho_M$ induces isomorphisms on cohomology groups.
   We complete the proof by showing that
   \[\tau^{\leq 1}\rR\Gamma_{\Prism}(\bM)\simeq [M\xrightarrow{A}M]\]
   is a quasi-isomorphism
   and that
   \[\tau^{\geq 2}\rR\Gamma_{\Prism}(\bM)=0.\]
   To simplify notations, we put $\alpha:=E'(\pi) $ throughout the proof.

   By (\ref{Equ-A(X)-II}), we see $d^0:M\to M\{X_1\}^{\wedge}_{\pd}$ is given by
   $d^0(\underline e) = \underline eAF_A(X_1).$
   So we have
   \[\rH^0_{\Prism}(\bM)\cong \Ker(d^0) = \Ker(A).\]
   We claim that
   \[\Ker(d^1:M\{X_1\}^{\wedge}_{\pd}\to M\{X_1,X_2\}^{\wedge}_{\pd}) = F_A(X_1)M,\]
   which implies that $\rho_M$ induces a quasi-isomorphism
   \[\tau^{\leq 1}\rR\Gamma_{\Prism}(\bM)\simeq [M\xrightarrow{A}M]\]
   as desired.
   Using the expression of $d^0$ in (\ref{Equ-differential in explicit way-III}), one can check directly that $F_A(X_1)M$ is killed by $d^1$. So we only need to show that if $\underline e\vec f(X_1)\in\Ker(d^1)$, then there exists an $a\in\calO_K^l$ such that $\vec f(X_1) = F_A(X_1) a$. Write
   \[\vec f(X_1) = \sum_{n\geq 0} a_nX_1^{[n]}.\]
   Since it is killed by $d^1$, by (\ref{Equ-differential in explicit way-III}), we have
\begin{equation}\label{Equ-d1}
  \vec f((1-\alpha X_1)^{-1}(X_2-X_1)) = (1-\alpha X_1)^{\frac{A}{\alpha}}(\vec f(X_2)-\vec f(X_1));
\end{equation}
that is,
\begin{equation*}
\begin{split}
    (1-\alpha X_1)^{\frac{A}{\alpha}}\sum_{n\geq 0}a_n(X_2^{[n]}-X_1^{[n]})&=(1-\alpha X_1)^{\frac{A}{\alpha}}(\vec f(X_2)-\vec f(X_1)) \\
    &=\vec f((1-\alpha X_1)^{-1}(X_2-X_1))\\
    &= \sum_{n\geq 0}a_n(1-\alpha X_1)^{-n}\sum_{0\leq m\leq n}(-1)^{n-m}X^{[m]}_2X_1^{[n-m]}\\
    & = \sum_{m\geq 0}X_2^{[m]}\sum_{k\geq 0}(-1)^ka_{m+k}(1-\alpha X_1)^{-m-k}X_1^{[k]}.
\end{split}
\end{equation*}
Comparing the coefficients (including $X_1$) of $X_2^{[m]}$ for any $m\geq 1$, we deduce that
\begin{equation*}
\begin{split}
    a_m\sum_{n\geq 0}\prod_{i=0}^{n-1}(-A+i\alpha)X_1^{[i]}& = (1-\alpha X_1)^{\frac{A}{\alpha}}a_m\\
    &=\sum_{i\geq 0}(-1)^ia_{m+i}(1-\alpha X_1)^{-m-i}X_1^{[i]}.
\end{split}
\end{equation*}
Comparing the coefficients of $X_1^{[1]}$ on both sides, we get
\begin{equation*}
    -Aa_m = ma_m-a_{m+1}.
\end{equation*}
In other words, for any $m\geq 1$, we have
\begin{equation*}
    a_m=(A+(m-1)\alpha)a_{m-1}=\cdots=\prod_{i=1}^{m-1}(A+i\alpha)a_1.
\end{equation*}
Now, by letting $X_1=X_2$ in equation (\ref{Equ-d1}), we see that
\[a_0 = \vec f(0) = 0.\]
So we deduce that
\[\vec f(X_1) = \sum_{n\geq 1}\prod_{i=1}^{n-1}(A+i\alpha)a_1X_1^{[n]}= F_A(X_1)a_1.\]
This proves the claim.

 It remains to prove $\tau^{\geq 2}\rR\Gamma_{\Prism}(\bM) = 0$. We state the idea as follows:

 For any $\underline e\vec f(\underline X)\in\Ker(d^s)$ with $s\geq 2$, write $\vec f(\underline X) = \sum_{I\in\bN^s}a_I\underline X^{[I]}$. We first find a certain subset $\Lambda\subset \bN^s$ such that the coefficients $\{a_J\}_{J\in\bN^s}$ of $\vec f(\underline X)$ are uniquely determined by $\{a_J\}_{J\in\Lambda}$ in a canonical way (by using $\underline e\vec f(\underline X)\in\Ker(d^s)$ merely). Then we construct a $\vec g(\underline X)\in\calO_K\{X_1,\dots,X_{s-1}\}^{\wedge}_{\pd}$ such that if we write
 \[d^{s-1}(\underline e\vec g(\underline X))=\underline e\sum_{J\in\bN^s}c_J\underline X^{[J]},\]
 then $c_J=a_J$ for all $J\in\Lambda$. Since $d^{s-1}(\underline e\vec g(\underline X))$ is automatically killed by $d^s$, this forces that $c_J=a_J$ for all $J\in\bN^s$ and hence that $d^{s-1}(\underline e\vec g(\underline X)) = \underline e\vec f(\underline X)$ as desired.

 We only prove $\rH^2_{\Prism}(\bM) = 0$ here. The general case can be deduced from a similar but much more complicated argument, which is included in Appendix \ref{Appendix}.
 Now let $\vec f(\underline X) = \sum_{i,j\geq 0}a_{i,j}X_1^{[i]}X_2^{[j]}\in\calO_K\{X_1,X_2\}^{\wedge}_{\pd}$ such that $\underline e\vec f(\underline X)$ is killed by $d^2$. Then by (\ref{Equ-differential in explicit way-III}), we have
 \begin{equation}\label{Equ-Kernel s = 2}
\begin{split}
    &\sum_{i_1,i_2\geq 0}a_{i_1,i_2}(X_2^{[i_1]}X_3^{[i_2]}-X_1^{[i_1]}X_3^{[i_2]}+X_1^{[i_1]}X_2^{[i_2]})\\
    =&\sum_{j_1,j_2,l_1,l_2\geq 0}(1-\alpha X_1)^{-j_1-j_2-l_1-l_2-\frac{A}{\alpha}}a_{j_1+l_1,j_2+l_2}(-1)^{j_1+j_2}X_1^{[j_1]}X_1^{[j_2]}X_2^{[l_1]}X_3^{[l_2]}
\end{split}
\end{equation}
 Let $X_1=0$ in (\ref{Equ-Kernel s = 2}). Then we see that
 \[\sum_{i_2\geq 1}a_{0,i_2}(X_2^{[i_2]}-X_3^{[i_2]}) = 0,\]
 which implies that
 \begin{equation}\label{Equ-s=2-I}
     a_{0,l} = 0
 \end{equation}
 for any $l\geq 1$.
 Comparing the coefficients (including $X_2$ and $X_3$) of $X_1^{[1]}$ in (\ref{Equ-Kernel s = 2}), we get
 \begin{equation}\label{Equ-s=2-X1}
 \begin{split}
     \sum_{i_2\geq 0}a_{1,i_2}(X_2^{[i_2]}-X_3^{[i_2]})
     =\sum_{l_1,l_2\geq 0}((A+(l_1+l_2)\alpha)a_{l_1,l_2}-(a_{1+l_1,l_2}+a_{l_1,1+l_2})) X_2^{[l_1]}X_3^{[l_2]}.
 \end{split}
 \end{equation}
 In (\ref{Equ-s=2-X1}), considering the coefficients of $X_2^{[l_1]}X_3^{[l_2]}$ for $l_1,l_2\geq 1$, we get
 \begin{equation}\label{Equ-s=2-II}
     a_{1+l_1,l_2} = (A+(l_1+l_2)\alpha)-a_{l_1,l_2+1}.
 \end{equation}
 In (\ref{Equ-s=2-X1}), considering the coefficients of $X_2^{[l_1]}$ for $l_1\geq 1$, we get
 \begin{equation}\label{Equ-s=2-III}
     a_{l_1+1,0} = (A+l_1\alpha)a_{l_1,0}-a_{l_1,1}-a_{1,l_1}.
 \end{equation}
 Putting $X_2=X_3=0$ in (\ref{Equ-s=2-X1}) and using (\ref{Equ-s=2-I}), we get
 \begin{equation}\label{Equ-s=2-IV}
     a_{1,0}=Aa_{0,0}-a_{0,1} = Aa_{0,0}.
 \end{equation}
 Putting (\ref{Equ-s=2-I}),(\ref{Equ-s=2-II}),(\ref{Equ-s=2-III}) and (\ref{Equ-s=2-IV}) together, we conclude that
 \begin{equation}\label{Equ-s=2-V}
    \left\{
    \begin{array}{rcl}
      a_{1+l_1,l_2} = (A+(l_1+l_2)\alpha)a_{l_1,l_2}-a_{l_1,1+l_2},&l_1,l_2\geq 1\\
      a_{1+l_1,0}=(A+l_1\alpha)a_{l_1,0}-a_{1,l_1}-a_{l_1,1}, &l_1\geq 1\\
      a_{0,l}=0,&l\geq 1\\
      a_{1,0}=Aa_{0,0}.&
    \end{array}
    \right.
 \end{equation}
 Note that all relations in (\ref{Equ-s=2-V}) are independent of the choice of $\vec f$.

 We claim that $\{a_{i,j}\}_{i,j\geq 0}$ is uniquely determined by $a_{0,0}$ and $\{a_{1,j}\}_{j\geq 1}$. Indeed, the claim can be checked by induction on $(i,j)\in \bN^2$, where we endow $\bN^2$ with the lexicographic order such that
 \[(0,0)<(0,1)<\cdots<(1,0)<(1,1)<\cdots.\]
 When $i\leq 1$, there is nothing to prove. Assume we have shown the claim for any $(i,j)\leq (n,m)$. Here, we may assume $n\geq 2$. When $m = 0$, the claim follows from
 \[a_{n,0} = (A+(n-1)\alpha)a_{n-1,0}-a_{1,n-1}-a_{n-1,1}.\]
 When $m\geq 1$, the claim follows from
 \[a_{n,m} = (A+(n+m-1)\alpha)a_{n-1,m}-a_{n-1,m+1}.\]

 Now, we define $\vec g(X_1) = b_0+\sum_{n\geq 2}b_nX_1^{[n]}$ such that
 \[b_n = -\sum_{m=1}^{n-1}\prod_{i=m+1}^{n-1}(A+i\alpha+\alpha)a_{1,m},\]
 for any $n\geq 2$ and $b_0=a_{0,0}$, where we denote $\prod_{i=n}^{n-1}(A+i\alpha+\alpha)$ by $I$.
 Since $\prod_{i=0}^{n-1}(A+i\alpha)$ tends to $0$, we see that $\lim_{n\to+\infty}b_n=0$. So $\vec g(X_1)$ is a well-defined element in $\calO_K\{X_1\}^{\wedge}_{\pd}$.

 Write $d^1(\underline e\vec g(X_1)) = \underline e\sum_{i,j\geq 0}c_{i,j}X_1^{[i]}X_2^{[j]}$. Then we have
 \begin{equation}\label{Equ-s=2-VI}
   \begin{split}
       &\sum_{i,j\geq 0}c_{i,j}X_1^{[i]}X_2^{[j]}
       \\=&(1-\alpha X_1)^{-\frac{A}{\alpha}}\vec g((1-\alpha X_1)^{-1}(X_2-X_1))+\vec g(X_1)-\vec g(X_2)\\
       =&(1-\alpha X_1)^{-\frac{A}{\alpha}}\sum_{m\geq 0}b_m(1-\alpha X_1)^{-m}(X_2-X_1)^{[m]}+\sum_{m\geq 0}b_m(X_1^{[m]}-X_2^{[m]})\\
       =&\sum_{i,j\geq 0}(1-\alpha X_1)^{-i-j-\frac{A}{\alpha}}b_{i+j}(-1)^iX_1^{[i]}X_2^{[j]}+\sum_{m\geq 0}b_m(X_1^{[m]}-X_2^{[m]})\\
       =&\sum_{i,j\geq 0}\sum_{k\geq 0}\prod_{l=0}^{k-1}(A+(i+j+k)\alpha)X_1^{[k]}b_{i+j}(-1)^iX_1^{[i]}X_2^{[j]}+\sum_{m\geq 0}b_m(X_1^{[m]}-X_2^{[m]}).
   \end{split}
   \end{equation}
 Putting $X_1=X_2 = 0$ in (\ref{Equ-s=2-VI}), we get
 \[c_{0,0} = b_0=a_{0,0}.\]
 Comparing the coefficients of $X_1^{[1]}X_2^{[n]}$ for $n\geq 1$ in (\ref{Equ-s=2-VI}), we see that
   \[c_{1,1} = (A+(n+1)\alpha)b_1 -b_2 =a_{1,1}\]
   and that
   \begin{equation*}
   \begin{split}
       c_{1,n} &= (A+(n+1)\alpha)b_n-b_{n+1}\\
       &=\sum_{m=1}^{n}\prod_{i=m+1}^{n}(A+i\alpha+\alpha)a_{1,m}-(A+(n+1)\alpha)\sum_{m=1}^{n-1}\prod_{i=m+1}^{n-1}(A+i\alpha+\alpha)a_{1,m}\\
       &= \sum_{m=1}^{n}\prod_{i=m+1}^{n}(A+i\alpha+\alpha)a_{1,m}-\sum_{m=1}^{n-1}\prod_{i=m+1}^{n}(A+i\alpha+\alpha)a_{1,m}\\
       &=a_{1,n}
   \end{split}
   \end{equation*}
    for amy $n\geq 2$. Since $d^1(\underline e\vec g(X_1))$ is also killed by $d^2$, all relations in (\ref{Equ-s=2-V}) are still true if one replaces $a_{i,j}$'s by $c_{i,j}$'s. This forces $d^1(\underline e\vec g(X_1)) = \underline e\vec f(X_1,X_2)$ and hence $\rH^2_{\Prism}(\bM) = 0$ as desired.
 \end{proof}
  \begin{rmk}
     In the ramified case, the proof of Theorem \ref{Thm-prismatic cohomology} can be deduced from the work of \cite{Tian}. We give a sketch of the argument as follows:

     Consider the $p$-adic completion $\calO_K\za T\ya$ of the polynomial ring $\calO_K[T]$ and let $\frakS\za T\ya$ be a smooth lifting of $\calO_K\za T\ya$ over $\frakS$ with compatible $\delta$-structure such that $\delta(T) = 0$. Then the prism $(\frakS\za T\ya,(E))$ covers the final object of ${\rm Shv}((\calO_K\langle T\rangle/(\frakS,(E)))_{\Prism})$ with the induced \v Cech nerve $(C^{\bullet},EC^{\bullet})$ such that for any $n\geq 0$, 
     \[C^n = \frakS\za T_0,T_1,\dots,T_n\ya\{\frac{T_0-T_1}{E},\dots,\frac{T_0-T_n}{E}\}^{\wedge}_{\delta},\]
     where $T_i$ denotes the image of $T$ via the face map $q_i:\frakS\za T\ya\to C^n$ induced by $\{0\}\to\{i\}\subset\{0,1,\dots,n\}$ for any $0\leq i\leq n$. By \cite[Corollary 4.5]{Tian}, one can check that for any $n\geq 0$, 
     \[C^n/EC^n = \calO_K\za T\ya\{X_1,\dots,X_n\}^{\wedge}_{\pd}\]
     is the $p$-complete free pd-algebra in variables $X_i$'s, the images of $\frac{T_0-T_i}{E}$'s modulo $E$\footnote{We remark that our notations is slightly different from those in \cite{Tian}. Namely, our $X_j$ is indeed $-\frac{\xi_{1,j}+\cdots+\xi_{1,j}}{d}$ in loc.cit..}, such that the face maps are $\calO_K\za T\ya$-linear and induced by $p_0(X_i) = X_{i+1}-X_1$ and for any $j\geq 1$,
     \[
       p_j(X_i) = \left\{\begin{array}{rcl}
           X_i &  0\leq i<j\\
           X_{i+1} & i\geq j>0.
       \end{array}
       \right.
     \]
     Note that the above formula on $p_i$'s are compatible with (\ref{Equ-structure morphism on variables}) by letting $E'(\pi) $ go to $0$.
     Then \cite[Theorem 4.10]{Tian} says that giving a Hodge--Tate crystal (i.e. reduced crystal in loc.cit.) $\bM'$ on $(\calO_K\za T\ya/(\frakS,(E)))$ amounts to giving a finite projective $\calO_K\za T\ya$-module $M'$ together with a stratification 
     \[\epsilon': M'\otimes_{\calO_K\za T\ya,p_0}\calO_K\za T\ya\{X_1\}^{\wedge}_{\pd}\to M'\otimes_{\calO_K\za T\ya,p_1}\calO_K\za T\ya\{X_1\}^{\wedge}_{\pd}\]
     with respect to $C^{\bullet}/EC^{\bullet}$ satisfying the cocycle condition and amounts to giving a finite projective $\calO_K\za T\ya$-module $M'$ together with a topologically nilpotent operator $A\in\End_{\calO_K\za T\ya}(M')$ (i.e a topologically nilpotent Higgs field in the one-dimensional case). More precisely, one can check the stratification $(M',\epsilon')$ corresponding to the nilpotent operator $A$ on $M'$ is indeed given by 
     \[\epsilon'=\exp(AX_1) = \sum_{i\geq 0}A^iX_1^{[i]}.\]
     Moreover, Tian also showed in \cite[\S 4]{Tian} that one can use the \v Cech-Alexander's method to compute the prismatic cohomology of a Hodge--Tate crystal $\bM'$. Namely, let $M'$ be the finite projective $\calO_K\za T\ya$-module together with the nilpotent operator $A$ corresponding to $\bM'$ in the above sense. Then there exist quasi-isomorphisms
     \[\rR\Gamma((\calO_K\za T\ya/(\frakS,(E)))_{\Prism},\bM')\simeq \bM'(C^{\bullet},EC^{\bullet})\simeq [M\xrightarrow{A}M],\]
     where $\bM'(C^{\bullet},EC^{\bullet})$ denotes total complex of the evaluation of $\bM'$ on the \v Cech nerve $(C^{\bullet},EC^{\bullet})$.

     Now, let $\bM$ be a Hodge--Tate crystal on $(\calO_K)_{\Prism}$ with the corresponding pair $(M,\phi_M)$ as in the Theorem \ref{Thm-HTC}. Note that in the ramified case, $E'(\pi)\equiv 0\mod \pi$. So $\phi_M$ is topologically nilpotent and thus induces a Hodge--Tate crystal $\bM'$ on $(\calO_K\za T\ya/(\frakS,(E)))_{\Prism}$ with associated finite projective $\calO_K\za T\ya$-module $M\otimes_{\calO_K}\calO_K\za T\ya$ and nilpotent operator $\phi_M\otimes\id_{\calO_K\za T\ya}$ via base-change along $\calO_K\to\calO_K\za T\ya$. By Remark \ref{Rmk-stratification and Sen}, the stratification induced by $\bM$ is given by $\varepsilon = (1-E'(\pi)X_1)^{-\frac{\phi_M}{E'(\pi)}}$, which coincides with the stratification $\epsilon'=\exp(\phi_MX_1)$ with respect to $C^{\bullet}/EC^{\bullet}$ induced by $\bM'$ modulo $\pi$. In particular, we can identify \v Cech-Alexander complexes
     \[\check \rC\rA(M)\otimes_{\calO_K}k[T] = \bM'(C^{\bullet},EC^{\bullet})\otimes_{\calO_K\za T\ya}k[T]\]
     and then get a quasi-isomorphism
     \[\check \rC\rA(M)\otimes^{\rL}_{\calO_K}k[T]\simeq [M\xrightarrow{\phi_M}M]\otimes_{\calO_K}^{\rL}k[T].\]
     Since $k\to k[T]$ is faithfully flat, we get a quasi-isomorphism 
     \[[M\xrightarrow{\phi_M}M]\otimes_{\calO_K}^{\rL}\calO_K/\pi\simeq \check \rC\rA(M)\otimes_{\calO_K}^{\rL}\calO_K/\pi\]
     and then obtain Theorem \ref{Thm-prismatic cohomology} by derived Nakayama's lemma.
 \end{rmk}

 \begin{rmk}\label{Rmk-Homotopy}
   Using (\ref{Equ-A(X)-II}) and (\ref{Equ-A(X)-III}), we see the following diagram
   \begin{equation*}
       \xymatrix@C=0.45cm{
         M\ar[rr]^{d^0}\ar[d]^{\id_M}&& M\{X_1\}^{\wedge}_{\pd}\ar[d]^{\frac{\partial}{\partial X_1}(0)}\\
         M\ar[rr]^A&& M
       }
   \end{equation*}
   commutes. So $\frac{\partial}{\partial X}(0)$ induces a morphism of complexes
   \[\rho_M':\check \rC\rA(M)\to[M\xrightarrow{A}M].\]
   One can check that
   \[\rho'_M\circ\rho_M = \id:[M\xrightarrow{A}M]\to[M\xrightarrow{A}M].\]
   So Theorem \ref{Thm-prismatic cohomology} implies $\rho_M'$ is also a quasi-isomorphism. We believe that the morphism
   \[\rho_M\circ\rho_M':\check \rC\rA(M)\to\check \rC\rA(M)\]
   is homotopic to the identity $\id_{\check \rC\rA(M)}$ but do not know how to prove it.
 \end{rmk}

\subsection{Absolute prismatic cohomology for prismatic crystals}
Similar to Hodge--Tate crystals, we define the category $\Vect((\calO_K)_{\Prism},\calO_{\Prism})$ of {\bf prismatic crystals}; that is, the category of sheaves $\bM$ of $\calO_{\Prism}$-modules such that for any prism $(A,I)\in (\calO_K)_{\Prism}$, the evaluation $\bM(A,I)$ is a finite projective $A$-module satisfying similar condition on base-changes as in Definition \ref{Dfn-Hodge--Tate crystal}.
In this subsection, we turn to studying the prismatic cohomology of a crystal $\bM\in \Vect((\calO_K)_{\Prism},\calO_{\Prism})$. Our goal in this section is to prove the following theorem.

\begin{thm}\label{non-reduced coh}
 Let $\bM\in \Vect((\calO_K)_{\Prism},\calO_{\Prism})$. We have $H^i((\calO_K)_{\Prism},\bM)=0$ for all $i>1$.
\end{thm}
\begin{rmk}
  This can be viewed as a special case of \cite[Conjecture 10.1]{BL-b}.
\end{rmk}
  In order to prove this theorem, we consider a restricted site $(\calO_K)^{\prime}_{\Prism}$. Its underlying category is the full subcategory of $(\calO_K)_{\Prism}$ spanned by the prisms admitting maps from the prism $(\frakS,(E))$. The coverings in $(\calO_K)^{\prime}_{\Prism}$ are inherited from the site $(\calO_K)_{\Prism}$. Note that $(\calO_K)^{\prime}_{\Prism}$ is not the relative prismatic site $(\calO_K/(\frakS,(E)))_{\Prism}$.

  We still denote $\bM$ the induced crystal of $\bM$ on $(\calO_K)^{\prime}_{\Prism}$. Then we have the following lemma.

\begin{lem}\label{prime&non-prime}
For any crystal $\bM\in \Vect((\calO_K)_{\Prism})$, we have
\begin{enumerate}
\item $\rR\Gamma((\calO_K)^{\prime}_{\Prism},\bM)\simeq \rR\Gamma((\calO_K)_{\Prism},\bM)\simeq \check \rC\rA(\bM)$,
\item $\rR\Gamma((\calO_K)^{\prime}_{\Prism},\bM/E^n)\simeq \rR\Gamma((\calO_K)_{\Prism},\bM/\calI^n)\simeq \check \rC\rA(\bM/\calI^n)$ for any $n\geq 1$,
\end{enumerate}
where $\calI$ is the sheaf on $(\calO_K)_{\Prism}$ sending a prism $(A,I)$ to $I$ and $\check \rC\rA(\bM)$ (resp. $\check \rC\rA(\bM/\calI^n)$) is the \v Cech-Alexandre complex of $\bM$ (resp. $\bM/\calI^n$) associated with the prism $(\frakS,(E))$.
\end{lem}

\begin{proof}
Since the prism $(\frakS,(E))$ covers the final objects of both the topos ${\rm Shv}((\calO_K)_{\Prism})$ and the topos ${\rm Shv}((\calO_K)^{\prime}_{\Prism})$. So we have
 \[
 \rR\Gamma((\calO_K)_{\Prism},\bM)\simeq \Rlim_i \rR\Gamma((\frakS^i,(E)),\bM)\simeq \rR\Gamma((\calO_K)^{\prime}_{\Prism},\bM)
 \]
 and
 \[
 \rR\Gamma((\calO_K)_{\Prism},\bM/\calI^n)\simeq \Rlim_i \rR\Gamma((\frakS^i,(E)),\bM/\calI^n)\simeq \rR\Gamma((\calO_K)^{\prime}_{\Prism},\bM/E^n),
 \]
 where $(\frakS^{\bullet},(E))$ is the \v Cech nerve associated with the prism $(\frakS,(E))$.
 Then it remains to prove that $R\Gamma((\frakS^i,(E)),\bM)$ and $\rR\Gamma((\frakS^i,(E)),\bM/\calI^n)$ are discrete for any $i$. The former follows from the vanishing of the higher \v Cech cohomology of $\calO_{\Prism}$ and the \v Cech-to-derived spectral sequence.

 For the latter, let $((\frakS^i,(E))\to (A,(E))$ be a flat cover. Then the \v Cech nerve of $\bM/\calI^n$ associated with the cover $((\frakS^i,(E))\to (A,(E))$ is $\bM/\calI^n(\frakS^i,(E))\otimes_{\frakS^i/E^n}C^n$, where $C^n$ is the \v Cech nerve of $\calO_{\Prism}/\calI^n$ associated with the cover $((\frakS^i,(E))\to (A,(E))$. We also let $C$ denote the \v Cech nerve of $\calO_{\Prism}$ associated with the cover $((\frakS^i,(E))\to (A,(E))$. So we have
 \[
 C\otimes_{\frakS^i}^{\rL}\frakS^i/E^n\simeq C^n
 \]
 for any $n\geq 1$.

 This implies
 \[
 C^n\otimes_{\frakS^i/E^n}^{\rL}\frakS^i/E\simeq C^1.
 \]
 Note that we have the following exact triangle
 \[
 C^{n-1}\xrightarrow{\times E} C^n\to C^n\otimes_{\frakS^i/E^n}^{\rL}\frakS^i/E=C^1.
 \]
 By taking cohomology, we get a long exact sequence
 \[
 \cdots\to \rH^i(C^{n-1})\xrightarrow{\times E}\rH^i(C^n)\to \rH^i(C^1)\to \rH^{i+1}(C^{n-1})\to \cdots.
 \]
 As $\rH^k(C^1)=0$ for any $k\geq 1$, we see that $\rH^k(C^n)=0$ for any $k\geq 1$. For $k=1$, this is because $E\rH^1(C^{n-1})=\rH^1(C^n)$.

 Now by using the \v Cech-to-derived spectral sequence, we see \[\rR\Gamma((\frakS^i,(E)),\bM/I^n)=\Gamma((\frakS^i,(E)),\bM/I^n),\] 
 which implies that
 \[
 \rR\Gamma((\calO_K)_{\Prism},\bM/\calI^n)\simeq \Rlim_i \rR\Gamma((\frakS^i,(E)),\bM/\calI^n)\simeq \Rlim_i \Gamma((\frakS^i,(E)),\bM/\calI^n).
 \]
 So we are done.
\end{proof}

 \begin{lem}
 For any $j>1$ and $n\geq 1$, the cohomology group \[\rH^j((\calO_K)^{\prime}_{\Prism},\bM/E^n)=0.\]
 \end{lem}

 \begin{proof}
We consider the short exact sequences of abelian sheaves on $(\calO_K)^{\prime}_{\Prism}$
  \[
  0\to \bM/E^n\xrightarrow{\times E} \bM/E^{n+1}\to \bar M:=\bM/E\to 0
  \]
  for any $n\geq 1$. Then we get the exact triangles
  \[
  \rR\Gamma((\calO_K)^{\prime}_{\Prism},\bM/E^n)\to \rR\Gamma((\calO_K)^{\prime}_{\Prism},\bM/E^{n+1})\to \rR\Gamma((\calO_K)^{\prime}_{\Prism},\bM).
  \]

  By Theorem \ref{Thm-prismatic cohomology} and Lemma \ref{prime&non-prime}, we know $\tau^{\leq 1}\rR\Gamma((\calO_K)^{\prime}_{\Prism},\bar M)\simeq \rR\Gamma((\calO_K)^{\prime}_{\Prism},\bar M)$. Then by induction, we see that $\tau^{\leq 1}\rR\Gamma((\calO_K)^{\prime}_{\Prism},\bM/E^n)\simeq \rR\Gamma((\calO_K)^{\prime}_{\Prism},\bM/E^n)$ for any $n\geq 1$.

 \end{proof}

 \begin{lem}
 For any crystal $\bM\in \Vect((\calO_K)_{\Prism}^{\prime},\calO_{\Prism})$, we have $\bM\simeq\Rlim\bM/E^n$.
 \end{lem}

 \begin{proof}
 Since an inductive limit of faithfully flat maps of prisms is a faithfully flat map of prisms (see \cite[Remark 2.4]{BS-b}), the topos ${\rm Shv}((\calO_K)^{\prime}_{\Prism})$ is replete. Then we have $\lim_n\bM/E^n\simeq\Rlim_n\bM/E^n$ by \cite[Proposition 3.1.10]{BS}. Then this lemma follows from the fact that $\bM$ is $E$-adically complete.
 \end{proof}

 \begin{proof}[Proof of Theorem \ref{non-reduced coh}]
 Since $\rR\Gamma$ commutes with derived inverse limit, we have
 \[
 \rR\Gamma((\calO_K)^{\prime}_{\Prism},\bM)\simeq \Rlim_n \rR\Gamma((\calO_K)^{\prime}_{\Prism},\bM/E^n).
 \]
 Now we consider the exact triangle of abelian sheaves on $(\calO_K)^{\prime}_{\Prism}$
 \[
 \bM/E\xrightarrow{\times E^n}\bM/E^{n+1}\to \bM/E^n.
 \]
 By taking the cohomology, we see that the natural map \[\rH^1((\calO_K)^{\prime}_{\Prism},\bM/E^{n+1})\to \rH^1((\calO_K)_{\Prism}^{\prime},\bM/E^n)\]
 is surjective. Then by  \cite[\href{https://stacks.math.columbia.edu/tag/07KY}{Tag 07KY}]{SP}, the cohomology group 
 \[\rH^k((\calO_K)^{\prime}_{\Prism},\bM)=\rH^k((\calO_K)_{\Prism},\bM)=0\]
 for any $k\geq 2$. This finishes the proof of Theorem \ref{non-reduced coh}.
 \end{proof}

\subsection{Rational Hodge--Tate crystals and $\Cp$-representations}
In this subsection, we want to assign to each (rational) Hodge--Tate crystal a semi-linear $\Cp$-representation of $G_K$. To this end, let us study the rational Hodge--Tate crystals on the site of perfect prisms $(\calO_K)_{\Prism}^{\perf}$ at first.

\begin{dfn}[Rational Hodge--Tate crystals on perfect site]\label{Dfn-rational Hodge--Tate crystal}
    By a {\bf rational Hodge--Tate crystal on} $(\calO_K)^{\perf}_{\Prism}$, we mean a sheaf $\bM$ of $\overline \calO_{\Prism}[\frac{1}{p}]$-modules such that for any $(A,I)\in(\calO_K)^{\perf}_{\Prism}$, $\bM(A)$ is a finite projective $A/I[\frac{1}{p}]$-module and that for any morphism $(A,I)\to (B,J)$, there is a canonical isomorphism
    \[\bM(A)\otimes_AB\xrightarrow{\cong}\bM(B).\]
    Denote by $\Vect((\calO_K)^{\perf}_{\Prism},\overline \calO_{\Prism}[\frac{1}{p}])$ the category of rational Hodge--Tate crystals on $(\calO_K)^{\perf}_{\Prism}$.
\end{dfn}
 Let $(B,I)$ be a cover of the final object of the topos ${\rm Shv}((\calO_K)^{\perf}_{\Prism})$ and $V$ be a finite projective $B/I[\frac{1}{p}]$-module. Similar to Definition \ref{Dfn-stratification}, one can define stratification satisfying the cocycle condition on $V$ such that the analogue of Lemma \ref{crystal is stratification} is true. Note that, by Lemma \ref{BK Cover}, $(A_{\inf},(\xi))$ is a cover of initial object of $(\calO_K)_{\Prism}^{\perf}$. Denote by $A_{\perf}^{\bullet}$ the associated \v Cech nerve in $(\calO_K)^{\perf}_{\Prism}$.
\begin{cor}\label{crystal is stratification-II}
    The category of rational Hodge--Tate crystals on $(\calO_K)^{\perf}_{\Prism}$ is equivalent to the category of finite dimensional $\hat K_{\cyc,\infty}$-vector spaces $V$ on which there is a stratification satisfying the cocycle condition.
\end{cor}
\begin{proof}
  This follows from the $v$-descent of vector bundles on perfectoid spaces over $\bQ_p$ (cf. \cite[Lemma 17.1.8]{SW} and \cite[Theorem 3.5.8]{KL}).
\end{proof}
  We want to give an explicit description on rational Hodge--Tate crystals on $(\calO_K)_{\Prism}^{\perf}$.
\begin{prop}\label{galois descent}
  There is a canonical isomorphism of cosimplicial rings
  \[\overline \calO_{\Prism}[\frac{1}{p}](A_{\perf}^{\bullet})\cong \rC(\hat G^{\bullet},\hat K_{\cyc,\infty}).\]
\end{prop}
\begin{proof}
  This follows from similar arguments used in the proof of \cite[Lemma 2.13]{MW}. So we just give an outline here. For any $i\geq 0$, let $S^i$ be the perfectoidization of the $p$-complete tensor product of $(i+1)$-copies of $\calO_{\hat K_{\cyc,\infty}}$ over $\calO_K$. By \cite[Theorem 3.10]{BS-a}, $A_{\perf}^{\bullet}/\xi = S^{\bullet}$. It is enough to show $S^{\bullet}[\frac{1}{p}] = C(\hat G^{\bullet},\hat K_{\cyc,\infty})$. Let $R^{i}$ be the $p$-complete normalization of $S^i$ in $S^i[\frac{1}{p}]$. Then $X^i: = \Spa(S^i[\frac{1}{p}],R^i)$ is the initial object in the category of perfectoid spaces over $\Spa(K,\calO_K)$ to which there are $(i+1)$-arrows from $X = X^0 = \Spa(\hat K_{\cyc,\infty},\calO_{\hat K_{\cyc,\infty}})$. So we see that
  \[X^i \cong X\times_K\times\dots\times_KX\]
  is the fibre product of $(i+1)$-copies of $X$ over $K$, which turns out to be
  \[\Spa(C(\hat G^i,\hat K_{\cyc,\infty}),C(\hat G^i,\calO_{\hat K_{\cyc,\infty}})).\]
  This implies $S^{\bullet}[\frac{1}{p}] = C(\hat G^{\bullet},\hat K_{\cyc,\infty})$ as desired.
\end{proof}
\begin{exam}\label{Exam-galois descent}
	We identify $(A_{\perf}^i/\xi)[\frac{1}{p}]$ with $\rC(\hat G^i,\hat K_{\cyc,\infty})$ for $i\in \{0,1,2\}$ via the isomorphisms in Proposition \ref{galois descent}. Then
	\begin{equation*}p_j: \hat K_{\cyc,\infty} = \rC(\hat G^0,\hat K_{\cyc,\infty})\to \rC(\hat G^1,\hat K_{\cyc,\infty})
	\end{equation*}
	for $j\in \{0,1\}$ is given by
		\begin{equation*}
			p_j(x)(g) = \left\{
			\begin{array}{rcl}
				g(x), & {\rm if} ~j=0\\
				x, &{\rm if}~j=1
			\end{array}   .
			\right.
		\end{equation*}
		for any $x\in \hat K_{\cyc,\infty}$ and $g\in \hat G$.
		
		The degeneracy morphism
		\[\Delta:\rC(\hat G^1,\hat K_{\cyc,\infty})\to \hat K_{\cyc,\infty}\]
		is given by $\Delta(f) = f(1)$, for any continuous function $f:\hat G\to\hat K_{\cyc,\infty}$.
		
		For $ij\in \{01,02,12\}$, for any $f\in \rC(\hat G^1,\hat K_{\cyc,\infty})$ and $g_0,g_1\in \Gamma$,
		\begin{equation*}
			p_{ij}(f)(g_0,g_1) = \left\{
			\begin{array}{rcl}
				f(g_0), & {\rm if} ~ij=01\\
				f(g_0g_1), & {\rm if}~ij=02\\
				g_0(f(g_1)), & {\rm if} ~ij=12
			\end{array}.
			\right.
		\end{equation*}
\end{exam}
\begin{thm}\label{rational crystal as representation}
  The evaluation at $(A_{\inf},(\xi))$ induces an equivalence from the category $\Vect((\calO_K)^{\perf}_{\Prism},\overline \calO_{\Prism}[\frac{1}{p}])$ to the category $\Rep_{\hat G}(\hat K_{\cyc,\infty})$ of representations of $\hat G$ over $\hat K_{\cyc,\infty}$.
\end{thm}
\begin{proof}
  This follows from Lemma \ref{crystal is stratification-II}, Example \ref{Exam-galois descent} and the Galois descent.
\end{proof}
\begin{rmk}\label{Rmk-rational crystal as representations}
  Note that by Faltings' almost purity theorem, the following categories
  \[\Rep_{\Gamma}(\hat K_{\cyc})\xrightarrow{\simeq}\Rep_{\hat G}(\hat K_{\cyc,\infty})\xrightarrow{\simeq}\Rep_{G_K}(\Cp)\]
  are equivalent. So Theorem \ref{rational crystal as representation} gives an equivalence between $\Vect((\calO_K)^{\perf}_{\Prism},\overline \calO_{\Prism}[\frac{1}{p}])$ and $\Rep_{G_K}(\Cp)$.
\end{rmk}

Now one can regard a Hodge--Tate crystal $\bM\in\Vect((\calO_K)_{\Prism},\overline \calO_{\Prism})$ as a $\hat K_{\cyc,\infty}$-representation of $\hat G$ via the composition of functors
\[V:\Vect((\calO_K)_{\Prism},\overline \calO_{\Prism})\to\Vect((\calO_K)^{\perf}_{\Prism},\overline \calO_{\Prism}[\frac{1}{p}])\xrightarrow{\simeq}\Rep_{\hat G}(\hat K_{\cyc,\infty}).\]
We want to give an explicit description of the representation $V(\bM)$ associated to $\bM$. To do so, we want to compare stratifications on $\bM(\frakS)$ and $\bM(A_{\inf})[\frac{1}{p}]$ via the canonical morphism
\[\frakS^{\bullet}/(E)\to A_{\perf}^{\bullet}/(E)[\frac{1}{p}]\cong \rC(\hat G^{\bullet},\hat K_{\cyc,\infty})\]
of cosimplicial rings.
We want to determine the image of
\[X\in\calO_K\{X\}^{\wedge}_{\pd}\cong \frakS^1/(E)\]
in $\rC(\hat G^1,\hat K_{\cyc,\infty})$ at first.
\begin{prop}\label{image of X}
  Let $\lambda$ be the image of $\frac{\xi}{E([\pi^{\flat}])}$ in $\calO_{\hat K_{\cyc,\infty}}$. Then
\[X(g) = c(g)\pi\lambda(1-\zeta_p),\]
where $c(g)\in\bZ_p$ is determined by $g(\pi^{\flat}) = \epsilon^{c(g)}\pi^{\flat}$.
\end{prop}
\begin{proof}
  Recall that $X$ is the image of $\frac{u_0-u_1}{E(u_0)}$. By the proof of Proposition \ref{galois descent}, as functions in $C(\hat G,A_{\inf})$, $u_0(g) = [\pi^{\flat}]$ and $u_1(g) = g([\pi^{\flat}])$. Therefore, we have
  \[X(g) = \frac{[\pi^{\flat}]-[\epsilon]^{c(g)}[\pi^{\flat}]}{E([\pi^{\flat}])} = -\frac{\xi}{E([\pi^{\flat}])}\varphi^{-1}(\mu)\frac{1-[\epsilon]^{c(g)}}{1-[\epsilon]}[\pi^{\flat}].\]
  Modulo $\xi$, we see $X(g) = c(g)\lambda\pi(1-\zeta_p)$ as desired.
\end{proof}
\begin{thm}\label{representation comes from crystal}
  Let $\bM\in\Vect((\calO_K)_{\Prism},\overline \calO_{\Prism})$ with associated pair $(M,\phi_M)$ in the sense of Theorem \ref{Thm-HTC}. Let $V(\bM) = M\otimes_{\calO_K}\hat K_{\cyc,\infty}$ be the induced representation of $\hat G$ over $\hat K_{\cyc,\infty}$. Then $V(\bM)$ is determined by the cocycle $U: \hat G\to \GL_l(\hat K_{\cyc,\infty})$ satisfying
  \begin{equation*}
  \begin{split}
      U(g) & = I+\sum_{n\geq 1}\frac{(c(g)\pi\lambda(1-\zeta_p))^n}{n!}\prod_{k=0}^{n-1}(kE'(\pi)+\phi_M)\\
      & = (1-c(g)E'(\pi)\pi\lambda(1-\zeta_p))^{-\frac{\phi_M}{E'(\pi)}},
  \end{split}
  \end{equation*}
  where $c(g)\in\Zp$ is determined by $g(\pi^{\flat}) = \epsilon^{c(g)}\pi^{\flat}$. In particular, when $K_{\cyc}\cap K_{\infty} = K$, if we write $g = \tau^i\gamma^j$ for some $i\in\Zp$ and $j\in\bZ_p^{\times}$, then we have
  \begin{equation}\label{Equ-cocycle for rep}
  \begin{split}
      U(\tau^i\gamma^j) = (1-E'(\pi)i\pi\lambda(1-\zeta_p))^{-\frac{\phi_M}{E'(\pi)}}.
  \end{split}
  \end{equation}
\end{thm}
\begin{proof}
  By Remark \ref{Rmk-stratification and Sen}, the corresponding stratification $\varepsilon$ on $M$ is given by
  \[\varepsilon = (1-E'(\pi)X)^{-\frac{\phi_M}{E'(\pi)}}.\]
  Then the theorem follows from Proposition \ref{image of X}.
\end{proof}
\begin{rmk}\label{Rmk-action on lambda}
  Note that $\gamma([\epsilon]) = [\epsilon]^{\chi(\gamma)}$. Combining this with (\ref{Equ-tau E over E}), we see that
  \[g(\lambda) = \chi(g)\frac{\zeta_p-1}{\zeta_p^{\chi(g)}-1}\lambda(1-c(g)\lambda(1-\zeta_p)\pi E'(\pi))^{-1}.\]
  In particular, when $K_{\cyc}\cap K_{\infty} =K$, we have
  \begin{equation}\label{Equ-action on lambda}
      \begin{split}
         & \tau(\lambda) = \lambda(1-\lambda(1-\zeta_p)\pi E'(\pi))^{-1};\\
         &  \gamma(\lambda) = \lambda\chi(\gamma)\frac{\zeta_p-1}{\zeta_p^{\chi(\gamma)}-1}.
      \end{split}
  \end{equation}
\end{rmk}

 Clearly, the functor $V: \Vect((\calO_K)_{\Prism},\overline \calO_{\Prism}) \to \Rep_{\hat G}(\hat K_{\cyc,\infty})$ is not fully faithful. However we can show the full faithfulness after inverting $p$. More precisely, the restriction to the perfect site also induces a functor (still denoted by $V$)
 \[V: \Vect((\calO_K)_{\Prism},\overline \calO_{\Prism}[\frac{1}{p}])\to \Rep_{\hat G}(\hat K_{\cyc,\infty})\cong \Rep_{G_K}(\Cp).\]
 One can show this functor is indeed fully faithful.
 
 \begin{prop}\label{Prop-fully faithful}
    Let $\bM\in\Vect((\calO_K)_{\Prism},\overline \calO_{\Prism}[\frac{1}{p}])$ with associated semi-linear $\Cp$-representation $V$ of $G_K$. Then we have a canonical isomorphism
    \[\rH^0_{\Prism}(\bM)\simeq V^{G_K}.\]
    In particular, the natural functor
   \[\Vect((\calO_K)_{\Prism},\overline \calO_{\Prism}[\frac{1}{p}])\to\Vect((\calO_K)^{\perf}_{\Prism},\overline \calO_{\Prism}[\frac{1}{p}])\simeq\Rep_{G_K}(\Cp)\]
   is fully faithful.
 \end{prop}
 \begin{proof}
   It is enough to show $\rH^0_{\Prism}(\bM) = V(\bM)^{\hat G}$ as the latter is exactly $V^{G_K}$ by construction.
   Let $(M,\phi_M)$ be the pair associated to $\bM$ in the sense of Theorem \ref{Thm-HTC}. By Theorem \ref{Thm-prismatic cohomology}, we get
   \[\rH^0_{\Prism}(\bM)\simeq M^{\phi_M=0}.\]
   Using Theorem \ref{representation comes from crystal}, it is clear that
   \[M^{\phi_M=0}\subset V(\bM)^{\hat G}.\]
   It remains to show the above inclusion is indeed an isomorphism.
   
   For any $v\in V(\bM)^{\hat G}$, by Theorem \ref{representation comes from crystal} again, we see that
   \[\phi_M(v)=\frac{1}{\lambda\pi(1-\zeta_p)}\lim_{g\to 1,c(g)\neq 0}\frac{U(g)-I}{c(g)}v=0.\]
   This forces that
   \[V(\bM)^{\hat G}\subset V(\bM)^{\phi_M=0}=M^{\phi_M=0}\otimes_K\hat K_{\cyc,\infty}.\]
   By some standard arguments, we see that the natural map
   \[V(\bM)^{\hat G}\otimes_K\hat K_{\cyc,\infty}\to V(\bM)\]
   is an injection. So we get inclusions
   \[M^{\phi_M=0}\otimes_K\hat K_{\cyc,\infty}\to V(\bM)^{\hat G}\otimes_K\hat K_{\cyc,\infty}\to V(\bM)^{\phi_M=0}=M^{\phi_M=0}\otimes_K\hat K_{\cyc,\infty}.\]
   Now the result follows by counting dimensions.
 \end{proof}

In \cite{Sen}, Sen showed that the following categories
\[\Rep_{\Gamma}(K_{\cyc})\to\Rep_{\Gamma}(\hat K_{\cyc})\to\Rep_{G_K}(\Cp)\]
are equivalent. Moreover, any semi-linear $\Cp$-representation $V$ of $G_K$ is uniquely determined by a linear operator $\Theta_{V}$ defined over $K$, which is known as the {\bf Sen operator}. More precisely, if we denote by $V_0$ the associated $K_{\cyc}$-representation of $\Gamma$, then for any $v\in V_0$, there is an open subgroup $\Gamma_v$ of $\Gamma$ such that for any $\gamma\in \Gamma_v$,
\[\gamma(v) = \exp(\log\chi(\gamma)\Theta_{V})(v).\]
For semi-linear $\Cp$-representations associated to Hodge--Tate crystals $\bM$, we give the following conjecture on their Sen operators.

\begin{conj}\label{Conj-Sen operator}
  Let $\bM$ be a rational Hodge--Tate crystal in $\Vect((\calO_K)_{\Prism},\overline \calO_{\Prism}[\frac{1}{p}])$ with associated pair $(M,\phi_M)$. Let $V$ be the associated $\Cp$-representation, then the Sen operator of $V$ is $\Theta_V=-\frac{\phi_M}{E'(\pi)}$.
\end{conj}
 
 At the end of this section, let us give a corollary of Conjecture \ref{Conj-Sen operator}: 
 
 If Conjecture \ref{Conj-Sen operator} is true, then by Theorem \ref{Thm-prismatic cohomology}, for a rational Hodge--Tate crystal $\bM$ with associated pair $(M,\phi_M)$ and corresponding $\Cp$-representation $V$ of $G_K$, we have quasi-isomorphisms
 \begin{equation}\label{Equ-quasi-isomorphisms}
     \rR\Gamma_{\Prism}(\bM)\otimes_K\Cp\simeq [M\xrightarrow{\phi_M}M]\otimes_K\Cp\simeq [V\xrightarrow{\Theta_V}V].
 \end{equation}
 On the other hand, by classical Sen theory, there is a natural quasi-isomorphism
 \[[V\xrightarrow{\Theta_V}V]\simeq \rR\Gamma(G_K,V)\otimes_K\Cp.\]
 So, we get the following \'etale comparison result:
 \begin{prop}\label{Prop-etale comparison}
   If Conjecture \ref{Conj-Sen operator} is true, then for any rational Hodge--Tate crystal $\bM$ with associated $\Cp$-representation $V(\bM)$ of $G_K$, there is a quasi-isomorphism
   \begin{equation}\label{Equ-EtaleComparison}
       \rR\Gamma_{\Prism}(\bM)\otimes_K\Cp\simeq \rR\Gamma(G_K,V(\bM))\otimes_K\Cp.
   \end{equation}
 \end{prop}
 \begin{rmk}\label{Gao Hui}  Conjecture \ref{Conj-Sen operator} was proved recently by Hui Gao \cite[Theorem 1.1.9]{Gao-b} by using the Sen theory formulated via locally analytic vectors developed in \cite{BC}.
 \end{rmk}

\section{The nilpotency of crystalline Breuil-Kisin modules}\label{4}
 Throughout this section, we always assume $K_{\cyc}\cap K_{\infty} = K$. We will attach a $\tau$-connection to a crystalline Breuil--Kisin module and show that it is ``topologically nilpotent'' in some certain sense. We also formulate a conjecture concerning the Hodge-Tate weights of the crystalline representation associated with a crystalline Breuil-Kisin module and the $\tau$-connection on it. As an evidence, we will give an example in the extended Fontaine-Laffaille case.
\begin{dfn}\label{Dfn-BK module}
\begin{enumerate}
  \item By a {\bf Breuil-Kisin module}, we mean a finite free $\frakS$-module $\frakM$ together with a semi-linear morphism $\varphi_{\frakM}:\frakM\to\frakM$ whose linearization is an isomorphism inverting $E$; that is, $\varphi_{\frakM}$ induces an isomorphism
  \[\varphi^*\frakM[\frac{1}{E}]\xrightarrow{\cong}\frakM[\frac{1}{E}].\]
  Denote the category of Breuil-Kisin modules by ${\rm BK}(\frakS)$.

  \item By a {\bf crystalline Breuil--Kisin module of height $h$}, we mean a Breuil--Kisin module $\frakM$ together with a continuous $\hat G$-action on $M_{\inf}:= \frakM\otimes_{\frakS}A_{\inf}$, which commutes with $\varphi$ such that the following conditions hold:
  \begin{enumerate}
      \item The induced morphism $\varphi^*\frakM\to\frakM$ is injective whose cokernel is killed by $E^h$.

      \item $\frakM\subset M_{\inf}^{\Gamma}$.

      \item {\bf crystalline condition:} $(\tau-1)(\frakM)\subset \varphi^{-1}(\mu)uM_{\inf}$.
  \end{enumerate}


  Denote by ${\rm BK}_h^{\cris}(\frakS)$ the category of crystalline Breuil-Kisin modules and define ${\rm BK}^{\cris}(\frakS) = \cup_{h\geq 0}{\rm BK}^{\cris}_h(\frakS)$.
\end{enumerate}
\end{dfn}
  A deep theorem in integral $p$-adic Hodge theory says that the crystalline Breuil-Kisin modules can be viewed as crystalline $\Zp$-representations of $G_K$. More precisely, we have the following theorem.
 \begin{thm}[\emph{\cite[Theorem 7.1.10]{Gao},\cite[Theorem F.11]{EG}}]\label{Gao-EG}
   There is an equivalence between the category ${\rm BK}^{\cris}(\frakS)$ and the category $\Rep_{G_K}^{\cris}(\Zp)$ of crystalline $\Zp$-representations of $G_K$ with non-negative Hodge-Tate weights. More precisely, the equivalence identifies ${\rm BK}^{\cris}_h(\frakS)$ with the category $\Rep_{G_K}^{\cris,h}(\Zp)$ of crystalline $\Zp$-representations whose Hodge-Tate weights lie in $[0,h]$.
\end{thm}
\begin{rmk}
	The definition above is slightly different from the original one in \cite{Gao}. But they are actually equivalent.
\end{rmk}
On the other hand, it was proved by Bhatt-Scholze \cite{BS-b} at first and then by Du-Liu \cite{DL} in a different way that the category $\Rep_{G_K}^{\cris}(\Zp)$ is equivalent to the category of prismatic $F$-crystals $\Vect^{\varphi}((\calO_K)_{\Prism},\calO_{\Prism})$ (c.f. \cite[Definition 4.1]{BS-b}). So we have the following theorem.
\begin{thm}[\emph{\cite[Theorem 5.6]{BS-b},\cite[Theorem 4.1.10]{DL}}]\label{DL}
    The evaluation at $(\frakS,(E))$ induces an equivalence from the category of prismatic $F$-crystals $\Vect^{\varphi}_h((\calO_K)_{\Prism},\calO_{\Prism})$ of height $h$ (cf. \cite[Definition 4.1.5 (2)]{DL}) to the category ${\rm BK}^{\cris}_h(\frakS)$ of crystalline Breuil-Kisin modules of height $h$.
\end{thm}
Therefore, for a crystalline Breuil-Kisin module $\frakM$, we may regard it as a prismatic $F$-crystal. In particular, it induces a Hodge--Tate crystal $\bM$ via the natural functor
\[\Vect^{\varphi}((\calO_K)_{\Prism},\calO_{\Prism})\to \Vect((\calO_K)_{\Prism},\overline \calO_{\Prism}).\]
Let $\overline \frakM$ be the corresponding finite free $\calO_K$-module endowed with a stratification $\varepsilon$ satisfying the cocycle condition. Then we see that
\[\overline \frakM = \frakM/E\frakM\]
and the $\hat G$-action on $\overline \frakM$ is induced by the $\hat G$-action on $\frakM$.

  Suggested by crystalline condition, one can define a ``$\tau$-connection'' on $\frakM$ by
  \[\nabla_{\frakM} = \frac{\tau-1}{\varphi^{-1}(\mu)u}:\frakM\to M_{\inf}\]
  which is compatible with the one in Lemma \ref{tau connection} in the following sense.
\begin{lem}\label{tau connection on M}
Keep notations as in Lemma \ref{tau connection}.

  \begin{enumerate}
      \item $\nabla_{\frakM}$ is $\rW(k)$-linear and for any $f\in \frakS$ and any $x\in \frakM$,
      \[\nabla_{\frakM}(fx) = \nabla(f)\tau(x)+f\nabla_{\frakM}(x).\]

      \item $\nabla_{\frakM}$ extends to a $\rW(k)$-linear morphism $\nabla_{\frakM}:\frakM\otimes_{\frakS,p_1}\frakS^1\to M_{\inf}\otimes_{A_{\inf}}\tilde A_{\inf}^1$ such that for any $f\in \frakS^1$ and any $x\in \frakM$,
      \begin{equation}\label{Equ-tau connection on M}
          \nabla_{\frakM}(fx) = \nabla(f)\tau(x)+f\nabla_{\frakM}(x).
      \end{equation}

      \item $\nabla_{\frakM}$ is compatible with $\varphi_{\frakM}$ in the following sense
      \begin{equation}\label{Equ-commute with phi on M}
        \nabla_{\frakM}\varphi_{\frakM} = \xi u_0^{p-1}\varphi_{\frakM}\nabla_{\frakM}.
      \end{equation}
  \end{enumerate}
\end{lem}
\begin{proof}
    This follows from the definition of $\nabla_{\frakM}$ combined with Lemma \ref{tau connection}.
\end{proof}

\begin{thm}\label{poly-nilpotency}
    Let $\lambda:= \frac{\xi}{E([\pi^{\flat}])}$. Let $\frakM$ be a crystalline Breuil-Kisin module with a fixed $\frakS$-basis $e_1,\dots,e_l$. Denote by $\tilde A$ the matrix of $\nabla_{\frakM}:\frakM\to M_{\inf}$. Then $\prod_{i=0}^{p-1}(\lambda^{-1}\tilde A+id_q(E))$ is $(p,[\pi^{\flat}])$-adically topologically nilpotent and $\lambda^{-1}\tilde A$ belongs to $\rM_l(k)$ modulo $(\xi,[\pi^{\flat}])$.
\end{thm}
\begin{proof}
  Let $B$ be the image of $\lambda^{-1}\tilde A$ modulo $\xi$. It is enough to show $B\in \rM_l(k)$ module $\pi$ and $\prod_{i=0}^{p-1}(B+iE'(\pi))$ is nilpotent modulo $\pi$. However, by (\ref{Equ-cocycle for rep}), we have
  \[B\equiv -A \mod \pi.\]
  In particular,
  \[\prod_{i=0}^{p-1}(B+iE'(\pi))\equiv \prod_{i=0}^{p-1}(-A+iE'(\pi)) \equiv -\prod_{i=0}^{p-1}(A+iE'(\pi)) \mod \pi.\]
  Now the theorem follows from Proposition \ref{matrix of stratification}.
\end{proof}
  Inspired by Theorem \ref{poly-nilpotency}, we make a more precise conjecture, which says somehow one can read the Hodge-Tate weights of a crystalline representation from the associated crystalline Breuil-Kisin module directly.
\begin{conj}\label{Conj-weight conjecture}
  Let $\frakM$ be a crystalline Breuil-Kisin module of rank $d$ and denote $\tilde A$ as before. Denote the Hodge-Tate weights of the associated crystalline $\Zp$-representation by $r_1\leq r_2\leq \dots\leq r_d$, then
  \[\prod_{i=1}^d(\lambda^{-1}\tilde A+r_id_q(E))\]
  is $(p,[\pi^{\flat}])$-adic topologically nilpotent.
\end{conj}

\begin{prop}\label{Prop-Conj imply conj}
    Conjecture \ref{Conj-weight conjecture} is a corollary of Conjecture \ref{Conj-Sen operator}.
\end{prop}
\begin{proof}
    Keep notations as in the proof of Theorem \ref{poly-nilpotency} and then we are reduced to showing that $X:=\prod_{i=1}^d(-A+r_iE'(\pi))$ is $\pi$-adically nilpotent. However, Conjecture \ref{Conj-Sen operator} suggests that $-\frac{A}{E'(\pi)}$ should be the Sen operator of the $\Cp$-representation associated to (the crystaline $\Zp$-representation induced by the given crystalline Breuil--Kisin module). Since $-r_i$'s are exactly eigenvalues of $-\frac{A}{E'(\pi)}$, we conclude that $X$ is actually nilpotent.
\end{proof}
 
  Since Conjecture \ref{Conj-Sen operator} was settled by Gao as mentioned in Remark \ref{Gao Hui}, Conjecture \ref{Conj-weight conjecture} should be a consequence of his work \cite{Gao}. However, it is worth giving a direct proof of Conjecture \ref{Conj-weight conjecture} in the Fontaine--Laffaille case, based on a structure theorem of Gee--Liu--Savitt \cite{GLS}.
\begin{exam}[Extended Fontaine-Laffaille case]\label{FL}
 Assume $K$ is unramified and $h\leq p$. In this case, $E(u) = u-\pi$ for some uniformizer $\pi\in K$. In particular, $d_q(E)=1$. We claim that $\prod_{i=1}^d(\lambda^{-1}\tilde A+[r_i]_q)$ is nilpotent modulo $(p,u)$, which means the Conjecture \ref{Conj-weight conjecture} holds in this case. Put $\lambda = \frac{\xi}{E}$.

  By \cite[Theorem 4.1]{GLS}, one may assume the matrix $\Phi$ of $\varphi_{\frakM}$ is of the form $XDY$ for $X,Y\in \GL_d(\frakS)$ with $Y\equiv I \mod p$ and $D = {\rm Diag}(E^{r_1}, \dots, E^{r_d})$.

  Since $\tilde A\tau(\Phi)+\nabla(\Phi) = \xi u^{p-1}\Phi\varphi(A)$, we get
\begin{equation*}
\begin{split}
  \tilde A\tau(X)\tau(D)+\xi d_q(X)\tau(D)+X\xi d_q(D) \equiv \xi u^{p-1}XD\varphi(\tilde A) \mod p.
\end{split}
\end{equation*}
  Recall that $\xi = \lambda E\equiv \lambda u\mod p$ and that
  \[d_q(E^r) = d_q((u-\pi)^r) = \sum_{i=1}^r[i]_q(-\pi)^{r-i}\binom{r}{i}u^{i-1}\equiv [r]_qu^{r-1}\mod p.\]
  we get
\begin{equation*}
  X^{-1}\frac{\tilde A}{\lambda}\tau(X)+u X^{-1}d_q(X)+U\equiv D\varphi(\tilde A)V \mod p,
\end{equation*}
  where $U = {\rm Diag}(\frac{[r_1]_qu^{r_1}}{\tau(u)^{r_1}}, \dots, \frac{[r_d]_qu^{r_d}}{\tau(u)^{r_d}})$ and $V = {\rm Diag}(\frac{u^p}{\tau(u)^{r_1}}, \dots, \frac{u^p}{\tau(u)^{r_d}})$.

  Since $0\leq r_1\leq \dots \leq r_d \leq p$, we deduce that $D\varphi(A)V$ is a strictly upper-triangular matrix modulo $(p,u)$ and that $\tau(X)\equiv X\mod u$. Since
  \[U\equiv L: = {\rm Diag}([r_1]_q, \dots, [r_d]_q)\mod u,\]
  we conclude that
  \[X^{-1}\frac{\tilde A}{\lambda}X+L\equiv D\varphi(\tilde A)V\mod (p,u).\]
  Put $T = -L+D\varphi(\tilde A)V$ and then we see that
  \[X^{-1}\prod_{i=1}^d(\frac{\tilde A}{\lambda}+[r_i]_q)X \equiv
  \prod_{i=1}^d([r_i]_q+T) \mod (p,u).\]
  Since $\prod_{i=1}^d([r_i]_q+T)$ is strict upper-triangular modulo $(p,u)$, we conclude that $\prod_{i=1}^d(\lambda^{-1}\tilde A+[r_i]_q)$ is nilpotent modulo $(p,u)$ as desired.
\end{exam}

 \section{Prismatic crystals and prismatic $\frakS$-modules}

Throughout this section, we assume $K_{\cyc}\cap K_{\infty} = K$ again. We will give an explicit description of prismatic crystals in $\Vect((\calO_K)_{\Prism},\calO_{\Prism})$. More precisely, by replacing the Frobenius in the definition of crystalline Breuil-Kisin modules with the $\tau$-connection, we can define the category of prismatic $\frakS$-modules. This category will be shown to be equivalent to the category $\Vect((\calO_K)_{\Prism},\calO_{\Prism})$. The idea of using the $\tau$-connection is inspired by \cite{MT}.
\begin{dfn}\label{Dfn-prismatic S module}
  By a {\bf prismatic $\frakS$-module of rank $r$}, we mean a finite free $\frakS$-module $\frakM$ of rank $r$ together with a continuous $\hat G$-action on $M_{\inf} = \frakM\otimes_{\frakS,\iota}A_{\inf}$ satisfying the following conditions:
  \begin{enumerate}
      \item $\frakM\subset M_{\inf}^{\Gamma}$.

      \item $(\tau-1)(\frakM)\subset\varphi^{-1}(\mu)uM_{\inf}$.

      \item Fix an $\frakS$-basis $e_1,\dots,e_r$ of $\frakM$, let $\tilde A$ be the matrix of the ``$\tau$-connection''
      \[\nabla_{\frakM}=\frac{\tau-1}{\varphi^{-1}(\mu)u}:\frakM\to M_{\inf}.\]
      Then the following conditions are satisfied:
      \begin{enumerate}
          \item $\lambda^{-1}\tilde A$ belongs to $\rM_r(k)$ modulo $(\xi,[\pi^{\flat}])$;
          \item {\bf Pseudo-nilpotency condition}: $\prod_{i=0}^{p-1}(\lambda^{-1}\tilde A+id_q(E))$ is $(\xi,[\pi^{\flat}])$-adically topologically nilpotent.
      \end{enumerate}
  \end{enumerate}
  We denote by $\Mod_{\Prism}(\frakS)$ the category of prismatic $\frakS$-modules.
\end{dfn}

\begin{thm}\label{crystal as S module}
  The evaluation at $(\frakS,(E))\in(\calO_K)_{\Prism}$ induces an equivalence from the category $\Vect((\calO_K)_{\Prism},\calO_{\Prism})$ of prismatic crystals to the category $\Mod_{\Prism}(\frakS)$ of prismatic $\frakS$-modules.
\end{thm}

\begin{rmk}\label{Rmk-Recover BS}
  It seems that one can use Theorem \ref{crystal as S module} to recover Theorem \ref{DL}. We only provide some ideas on how to construct a prismatic $F$-crystals of height $h$ from a crystalline Breuil--Kisin module $\frakM$ of height $h$ under the assumption that Theorem \ref{poly-nilpotency} holds for $\nabla_{\frakM}$: Indeed, in this case, the crystalline Breuil--Kisin module $\frakM$ induces a prismatic $\frakS$-module after forgetting height $h$ and the $\varphi$-action $\varphi_{\frakM}$, and hence induces a prismatic crystal $\bM$ by Theorem \ref{crystal as S module}. Similarly, $\varphi^*\frakM$ will also induces a prismatic crystal $\varphi^*\bM$ (note that $\nabla_{\varphi^*\frakM} \equiv 0 \mod (E(u),u)$). Then the morphism $\varphi_{\frakM}:\varphi^*\frakM \to \frakM$ lifts to a morphism $\varphi_{\bM}:\varphi^*\bM\to \bM$. Since $(\frakS,(E))$ covers the final object of ${\rm Shv}((\calO_K)_{\Prism})$, by evaluating at $(\frakS,(E))$ and using the height condition on $\frakM$, one can check that $\varphi_{\bM}$ is always injective with cokernel being killed by $\calI^h$. So we get a prismatic $F$-crystal $\bM$ under the above assumption. 
  
  It is worth pointing out that the authors know from Hui Gao that the above assumption is always true by using general theory of $(\varphi,\tau)$-modules.
\end{rmk}

 The rest of this section is devoted to proving Theorem \ref{crystal as S module}. We first show how to construct a prismatic crystal out of a prismatic $\frakS$-module. This is more difficult than the other direction and we will adapt the idea used in \cite{MT}.
 
\subsection{Prismatic $\frakS$-modules induce Prismatic crystals}

 Fix a prismatic $\frakS$-module $\frakM$ of rank $r$. Let $\tau$ only act on the first component of $\frakS^1$. As in Lemma \ref{tau connection on M}, one can define a morphism
  \[\nabla: \frakM\otimes_{\frakS,p_1}\frakS^1\rightarrow \frakM\otimes_{\frakS,p_1}\Ainf\widehat \otimes_{\rW}\frakS\{\frac{u_0-u_1}{E(u_0)}\}_{\delta}^{\wedge}.\]
  Modulo $E(u_0)$, this reduces to a ``$\tau$-connection''
  \[\nabla: \frakM/E\otimes_{\calO_K}\calO_K\{X\}^{\wedge}_{\pd}\rightarrow \frakM/E\otimes_{\calO_K}\calO_{\widehat K_{\cyc,\infty}}\{X\}^{\wedge}_{\pd}.\]
  such that
\begin{equation}\label{Equ-reduced nabla}
  \nabla(f(X)m) = \nabla(f)m+f(\tau(X))\nabla(m).
\end{equation}

  Put $\lambda=\frac{\xi}{E}$ and $\lambda' = \frac{\xi}{\tau(E)}$, which belong to $\rA_{\inf}^{\times}$. We often confuse $\lambda$ and $\lambda'$ with their images in $A_{\inf}/\xi$ for simplicity. In particular, as elements in $\Ainf/\xi$, we have
  \begin{equation}\label{Equ-Lambdaprime}
      \lambda' = \lambda(1+(\zeta_p-1)\pi\lambda E'(\pi))^{-1}.
  \end{equation}
  By equation (\ref{Equ-tau on X}), we see that
  \[\nabla(X) = \lambda'-\lambda' E'(\pi) X\]
  and
  \[\tau(X) = \lambda' (\zeta_p-1)\pi+(1-\lambda' (\zeta_p-1)\pi E'(\pi))X.\]

  Fix a set of basis $\{e_1, \dots, e_r\}$ of $\frakM$ and denote the matrix of $\nabla$ by $\tilde A$. We will also confuse $\tilde A$ and its image in $M_r(A_{\inf}/\xi)$. Then for any $\vec f(X)\in (\calO_K\{X\}^{\wedge}_{\pd})^r$, we have
\begin{equation}\label{Equ-reduced nabla-II}
  \nabla(\underline e\vec f(X)) = \underline e\tilde A\vec f(\lambda' \pi(\zeta_p-1)+(1-\lambda' \pi(\zeta_p-1)E'(\pi))X)+\underline e\nabla \vec f(X).
\end{equation}
  If we write $\vec f(X) = \sum_{i=0}^{\infty}\vec a_iX^{[i]}$, then the equation (\ref{Equ-reduced nabla-II}) reduces to the following one
\begin{equation}\label{Equ-reduced nabla-III}
\begin{split}
  \nabla(\underline e\vec f(X)) =& \underline e \sum_{i\geq 0}((1-\lambda' \pi(\zeta_p-1)E'(\pi))^i\sum_{j\geq 0}\frac{(\lambda'\pi(\zeta_p-1))^j}{j!}\tilde A\vec a_{i+j}X^{[i]}+\vec a_i\nabla(X^{[i]})).
\end{split}
\end{equation}
 Recall that $\nabla$ behaves as a connection modulo $\pi$ (cf. Example \ref{Exam-nabla on pd algebra}). Therefore, we see that modulo $\pi$,
 \begin{equation}\label{Equ-reduced nabla-IV}
     \nabla(\underline e\vec f(X)) = \underline e \sum_{i\geq 0}(\tilde A\vec a_i-i\lambda' E'(\pi)+\lambda' \vec a_{i+1})X^{[i]}.
 \end{equation}
 


\begin{lem}\label{Lem-Solution modulo u}
  For any $\vec g(X)\in \calO_K\{X\}^{\wedge}_{\pd}$, there exists an $\vec f(X)$ such that $\nabla(\underline e\vec f(X))=\underline e\vec g(X)$ modulo $\pi$. Moreover, the morphism $\calO_K\{X\}^{\wedge}_{\pd}\rightarrow \calO_K$ sending $X$ to $0$ induces an isomorphism from $(\frakM\otimes_{\frakS,p_1}\calO_K\{X\}^{\wedge}_{\pd}/\pi)^{\nabla = 0}$ to $\frakM/(E(u_0),u_0)$.
\end{lem}
\begin{proof}
  Write $\vec f(X) = \sum_{i=0}^{\infty}\vec a_iX^{[i]}$. Then we deduce from (\ref{Equ-reduced nabla-IV}) that 
  \[\underline e \sum_{i\geq 0}(\tilde A\vec a_i-i\lambda' E'(\pi)+\lambda' \vec a_{i+1})X^{[i]} = \underline e \vec g(X).\]
  
  When $\vec g(X) = 0$, it is easy to see that for any $i\geq 0$, $\vec a_i = \prod_{j=0}^{i-1}(jE'(\pi)-\frac{\tilde A}{\lambda'})\vec a_0$. This is well-defined due to the nilpotent condition in Definition \ref{Dfn-prismatic S module} (3) and implies the ``moreover'' part of result.

  In general, we write $\vec g(X) = \sum_{i\geq 0}\vec b_iX^{[i]}$. Put $\vec a_0 = 0$ and $\vec a_i = \sum_{j=0}^{i-1}\prod_{k=j+1}^{i-1}(kE'(\pi)-\frac{\tilde A}{\lambda'})\frac{\vec b_j}{\lambda'}$ for any $i\geq 1$. Then $\vec f(X) = \sum_{i\geq 0}\vec a_iX^{[i]}$ is well-defined satisfying $\nabla(\underline e\vec f(X)) = \underline e\vec g(X)$.
\end{proof}

\begin{rmk}\label{Rmk-Pseudo nilpotency is necessary}
 From the proof of Lemma \ref{Lem-Solution modulo u}, the ``nilpotency'' of $\tilde A$ is necessary to ensure the surjectivity of the morphism
  \[(\frakM\otimes_{\frakS,p_1}\calO_K\{X\}^{\wedge}_{\pd}/\pi)^{\nabla = 0} \rightarrow \frakM/(E(u_0),u_0).\]
\end{rmk}

\begin{lem}\label{Lem-Solution modulo E}
  For any $\vec g(X)\in \calO_K\{X\}^{\wedge}_{\pd}$, there exists an $\vec f(X)$ such that $\nabla(\underline e\vec f(X))=\underline e\vec g(X)$.
  Moreover, the morphism $\calO_K\{X\}^{\wedge}_{\pd}\rightarrow \calO_K$ sending $X$ to $0$ induces an isomorphism from $(\frakM\otimes_{\frakS,p_1}\calO_K\{X\}^{\wedge}_{\pd})^{\nabla = 0}$ to $\frakM/E(u_0)$.
\end{lem}
\begin{proof}
  Assume $\vec g(X)=\sum_{i\geq 0}\vec b_iX^{[i]}\in \calO_K\{X\}^{\wedge}_{\pd}$.
  We shall show the existence of $\vec f(X)$ by constructing a sequence $\{\vec f_n(X)\}_{n\geq 0}$ in $\calO_K\{X\}^{\wedge}_{\pd}$ inductively such that for any $N\geq 0$,
  \[\nabla(\sum_{i=0}^N\underline e\pi^i\vec f_i(X))\equiv \underline e\vec g(X) \mod \pi^{N+1}.\]

  By Lemma \ref{Lem-Solution modulo u}, one can choose $\vec f_0(X)$ such that the above formula holds for $N=0$.
  Assume we have constructed desired $\vec f_0(X), \dots, \vec f_n(X)$.
  Define $\vec h_n(X)$ such that
  \[\underline e \vec h_n(X) =\pi^{-n-1}(\underline e\vec g(X)- \nabla(\sum_{i=0}^n\underline e\pi^i\vec f_i(X))).\]
  Then by Lemma \ref{Lem-Solution modulo u} again, there exists $\vec f_{n+1}(X)$, such that \[\nabla(\underline e\vec f_{n+1}(X))\equiv\underline e \vec h_n(X)\mod \pi.\]
  It is easy to check
  \[\nabla(\sum_{i=0}^{n+1}\underline e\pi^i\vec f_i(X))\equiv \underline e\vec g(X) \mod \pi^{n+2}\]
  as expected.

  It remains to show the morphism
  \[\alpha: N: = (\frakM\otimes_{\frakS,p_1}\calO_K\{X\}^{\wedge}_{\pd})^{\nabla = 0}\rightarrow \frakM/E(u_0)\]
  sending $X$ to $0$ is an isomorphism.

  For any $\underline e \vec a\in \frakM$, put $\vec f_{\vec a}(X) = \sum_{i\geq 0}\prod_{k=0}^{i-1}(kE'(\pi)-A_0)\vec a X^{[i]}$, where $A_0\in\rM_l(\frakS)$ coincides with $\lambda^{-1}\tilde A$ modulo $(\xi,[\pi^{\flat}])$.
  Then according to the proof of Lemma \ref{Lem-Solution modulo u}, we have $\nabla(\underline e \vec f_{\vec a}(X))\equiv 0\mod \pi$.
  By what we have showed above, there exists a $\vec g_{\vec a}(X)$ such that $-\pi\nabla(\underline e\vec g_{\vec a}(X)) = \nabla(\underline e\vec f_{\vec a}(X))$.
  Let $\vec h_{\vec a}(X) = \vec f_{\vec a}(X)+\pi\vec g_{\vec a}(X)$.
  We see $\underline e\vec h_{\vec a}(X)\in N$ satisfying $\vec h_{\vec a}(0) = \vec a+\pi\vec g_{\vec a}(0)$.
  As a consequence, $\alpha(N)+\pi \frakM = \frakM$, which implies the surjectivity of $\alpha$.

  To see the injectivity of $\alpha$, we need to show that
  $\underline e\vec f(X) \in N$ implies $\vec f(X)=0$.
  Indeed, it follows from the last criterion of Lemma \ref{Lem-Solution modulo u} that $\vec f(X)\equiv 0\mod \pi$.
  So there is some $\vec g(X)$ such that $\pi \vec g(X) = \vec f(X)$.
  According to the $\pi$-torsion freeness of $\frakM\otimes_{\frakS,p_1}\calO_K\{X\}^{\wedge}_{\pd}$,  $\underline e\vec g(X)\in N$ is also annihilated by $\nabla$.
  By iteration, one can deduce from the above argument that $\vec f(X)\equiv 0\mod \pi^n$ for all $n\geq 0$, which forces $\vec f(X) = 0$.
\end{proof}
\begin{cor}\label{Cor-Approximation}
  For any $x\in \frakM\otimes_{\frakS,p_1}\frakS^1$ and any $m\in \frakM$, there exists a $y\in \frakM\otimes_{\frakS,p_1}\frakS^1$ such that $\nabla(y)\equiv x \mod (E(u_1))$ and that $\Delta(y)\equiv m \mod (E(u_0))$, where $\Delta:\frakS^1\rightarrow\frakS$ is the ``multiplication'' map.
\end{cor}
\begin{proof}
  By Lemma \ref{Lem-Solution modulo E}, one can find some $y_1$ such that $\nabla(y_1)\equiv x\mod (E(u_0))$.
  Then by Lemma \ref{Lem-Solution modulo E} again, one can find some $y_2$ such that $\nabla(y_2)\equiv 0\mod (E(u_0))$ and that $\Delta(y_2)\equiv m-\Delta(y_1) \mod (E(u_0))$.
  Then $y = y_1+y_2$ is a desired element.
\end{proof}
\begin{prop}\label{Prop-Key isomorphism}
  The multiplication map $\Delta:\frakS^1\rightarrow\frakS$ induces an isomorphism
  \[(\frakM\otimes_{\frakS,p_1}\frakS^1)^{\tau=1}\rightarrow \frakM\]
  of $\frakS$-modules, where the former is viewed as an $\frakS$-module via $p_0:\frakS\rightarrow \frakS^1$.
\end{prop}
\begin{proof}
  This can be checked by the same argument as in the proof of Lemma \ref{Lem-Solution modulo E}.
  More precisely, for any $m\in \frakM$, we shall construct a sequence $\{x_n\}_{n\geq 0}$ in $\frakM\otimes_{\frakS,p_1}\frakS^1$ such that for any $N\geq 0$,
\begin{equation*}
  \left\{
  \begin{array}{rcl}
  \nabla(\sum_{i=0}^N(x_iE(u_1)^i))\equiv 0 \mod (E(u_0))^{N+1}\\
  \Delta(\sum_{i=0}^N(x_iE(u_1)^i))\equiv m \mod (E(u_0))^{N+1}
  \end{array}.
  \right.
\end{equation*}
  Then $x = \sum_{n\geq 0}x_nE(u_1)^n$ is well-defined in $\frakM\otimes_{\frakS,p_0}\frakS^1$ by $(E(u_0))$-adic completeness, which is $\tau$-stable and goes to $m$ via $\Delta$.
  This shows the surjectivity of $\Delta$.

  The existence of $x_0$ follows from Lemma \ref{Lem-Solution modulo E}.
  If we have constructed $x_0, \dots, x_n$, then $y_n:=-\frac{\nabla(\sum_{i=0}^nx_iE(u_1)^i)}{E(u_1)^{n+1}}$ and $m_n:=\frac{m-\Delta(\sum_{i=0}^nx_iE(u_1)^i)}{E(u_1)^{n+1}}$ are both well-defined and belong to $\frakM\otimes_{\frakS,p_0}\frakS^1$ and $\frakM$, respectively.
  By Corollary \ref{Cor-Approximation}, one can choose $x_{n+1}$ such that
\begin{equation*}
  \left\{
  \begin{array}{rcl}
  \nabla(x_{n+1})\equiv y_n \mod (E(u_0))\\
  \Delta(x_{n+1})\equiv m_n \mod (E(u_0))
  \end{array}.
  \right.
\end{equation*}
  It is easy to check such an $x_{n+1}$ satisfies all desired conditions.

  It remains to show the injectivity of $\Delta$.
  Let $x\in (\frakM\otimes_{\frakS,p_1}\frakS^1)^{\nabla = 0}$ be annihilated by $\Delta$. By Lemma \ref{Lem-Solution modulo E}, $E(u_1)$ divides $x$. It is clear that $\frac{x}{E(u_1)}$ is also killed by $\nabla$. So by iteration, one shows that for any $n\geq 0$, $x\equiv 0\mod E(u_1)^n$, which forces $x= 0$.
\end{proof}
  Now, the inverse of the isomorphism given in Proposition \ref{Prop-Key isomorphism} induces a morphism of $\frakS$-modules
  \[\frakM\rightarrow (\frakM\otimes_{\frakS,p_1}\frakS^1)^{\tau=1}\hookrightarrow \frakM\otimes_{\frakS,p_1}\frakS^1.\]
  By scalar extension, this gives rise to an $\frakS^1$-linear morphism
\begin{equation}\label{Equ-Stratification}
  \epsilon: \frakM\otimes_{\frakS,p_0}\frakS^1\rightarrow \frakM\otimes_{\frakS,p_1}\frakS^1.
\end{equation}

Now, let $(A_{\inf}^{\bullet},\xi A_{\inf}^{\bullet})$ be the cosimplicial prism induced by the cover $(A_{\inf},(\xi))$ of the final object ${\rm Shv}((\calO_K)_{\Prism})$. For $i=0,1$, let $\tau_i:\frakS^1\to A_{\inf}^1$ be the analogue of the morphism $\tau:\frakS\to A_{\inf}$ on the $(i+1)$-th component.
\begin{lem}\label{Lem-Stratification is tau-equivariant}
  The $\frakS^1$-linear morphism $\epsilon$ is compatible with the $(\tau_0^{\Zp},\tau_1^{\Zp})$-action. In other words, for any $a,b\in \Zp$, the following diagram
  \[\xymatrix@C=0.45cm{
    \frakM\otimes_{\frakS,p_0}\frakS^1\ar[rr]^{(\tau_0^a,\tau_1^b)}\ar[d]^{\epsilon}&&
    \frakM\otimes_{\frakS,p_0}A_{\inf}^1\ar[d]^{\epsilon}\\
    \frakM\otimes_{\frakS,p_1}\frakS^1\ar[rr]^{(\tau_0^a,\tau_1^b)}&& \frakM\otimes_{\frakS,p_1}A_{\inf}^1
  }\]
  commutes.
\end{lem}
\begin{proof}
  We use the same argument in \cite[Proposition 3.19]{MT}.
  By the construction of $\epsilon$, we only need to show $\epsilon(\tau_1^b(m)\otimes 1) = (\tau_0^a,\tau_1^b)(\epsilon(m\otimes 1))$.
  Since for any $c\in \Zp$ and $l\in \frakS^1$, $\Delta((\tau_0^c,\tau_1^c)(l)) = \tau^c(\Delta(l))$, we see that for any $x\in M\otimes_{\frakS,p_1}\frakS_1$, $\Delta((\tau_0^c,\tau_1^c)(x)) = \tau^c(\Delta(x))$.
  In particular, for $x = \epsilon(m\otimes 1)\in (\frakM\otimes_{\frakS,p_1}\frakS_1)^{\tau=1}$, we get
  $\Delta((\tau_0^a,\tau_1^b)(\epsilon(m\otimes 1)))  = \Delta((\tau_0^a,\tau_1^b)(x))
   = \Delta((\tau_0^b,\tau_1^b)(x))
   = \tau^b(\Delta(x))
   = \tau^b(m).$
  By construction of $\epsilon$, this amounts to $\epsilon(\tau_1^b(m)\otimes 1) = (\tau_0^a,\tau_1^b)(\epsilon(m\otimes 1))$ as desired.
\end{proof}
  Let $\tau_i (i=0,1,2)$ be the $\tau$-action on the $(i+1)$-th component of $\frakS^2\rightarrow A_{\inf}^2$.
  Define $\nabla_0 = \frac{\tau_0-1}{\varphi^{-1}(\mu)u_0}$ and $\nabla_1 = \frac{\tau_1-1}{\varphi^{-1}(\mu)u_1}$.
  Let $X_1$ and $X_2$ be the image of $\frac{u_0-u_1}{E(u_0)}$ and $\frac{u_0-u_2}{E(u_0)}$ in $\frakS^2/E$.
  Then we have that $\nabla_0(X_i) = \lambda'-\lambda' E'(\pi)X_i$ and $\tau_0(X_i) = \lambda' (\zeta_p-1)\pi+ (1-\lambda'\pi(\zeta_p-1)E'(\pi)) X_i$ for $i=1,2$, and that $\nabla_1(X_1) = -\lambda $, $\nabla_1(X_2) = 0$ and $\tau_1(X_1) = X_1-\lambda (\zeta_p-1)\pi$, $\tau_1(X_2) = X_2$.
  Clearly, the matrix of $\nabla_0$ and $\nabla_1$ on $\frakM\otimes_{\frakS,q_0}\frakS^2/E$ are given by $\tilde A$ and $0$, respectively.
  \begin{rmk}\label{Rmk-Nabla_1 as connection modulo pi}
      Similar to (the last paragraph of) Example \ref{Exam-nabla on pd algebra}, one can show that $\nabla_1$ behaves as a connection modulo $\pi$. To see this, it is enough to check that for any $m\geq 1$, $\nabla_1(X_1^{[m]}) \equiv X_1^{[m-1]}\nabla_1(X_1)\mod \pi$. However, we see that
      \begin{equation*}
          \begin{split}
              \nabla_1(X_1^{[m]}) & = \frac{1}{m!}\sum_{i=0}^{m-1}\tau_1(X_1)^iX_1^{m-i}\nabla_1(X_1)\\
              & = \frac{1}{m!}\frac{\tau_1(X_1)^m-X_1^m}{\tau_1(X_1)-X_1}\nabla_1(X_1)\\
              & = \frac{1}{m!}\frac{(X_1-\lambda (\zeta_p-1)\pi)^m-X_1^m}{-\lambda (\zeta_p-1)\pi}\nabla_1(X_1)\\
              & = \frac{1}{m!}\nabla_1(X_1)\sum_{i=0}^{m-1}\binom{m}{i}(-\lambda(\zeta_p-1)\pi)^{m-1-i}X_1^i\\
              & = \nabla_1(X_1)\sum_{i=0}^{m-2}(-\lambda \pi)^{m-1-i}\frac{(\zeta_p-1)^{m-1-i}}{(m-i)!}X_1^{[i]}+X_1^{[m-1]}\nabla_1(X_1)\\
              &\equiv X_1^{[m-1]}\nabla_1(X_1) \mod \pi.
          \end{split}
      \end{equation*}
      The last congruence follows from the fact that $\nu_p((\zeta_p-1)^k) = \frac{k}{p-1}\geq \nu_p((k+1)!)$ for any $k\geq 0$.
  \end{rmk}
\begin{prop}\label{Prop-Key injection}
  The morphism $\Delta: (\frakM\otimes_{\frakS,q_0}\frakS^2)^{\tau_0=\tau_1=1}\rightarrow \frakM$ induced by the multiplication $\Delta: \frakS^2\rightarrow \frakS$ is an injection.
\end{prop}
\begin{proof}
  We proceed as in the proof of Proposition \ref{Prop-Key isomorphism}.
  At first, we claim that $\Delta$ induces an isomorphism modulo $(E(u_0),u_0)$; that is,
  \[N:=(\frakM\otimes_{\frakS,q_0}\frakS^2/(E(u_0),u_0))^{\nabla_0=\nabla_1=0}\rightarrow \frakM/(E(u),u)\]
  is an isomorphism.

  Let $\vec f(X_1,X_2) = \sum_{i,j\geq 0}\vec f_{ij}X_1^{[i]}X_2^{[j]}\in \calO_K\{X_1,X_2\}^{\wedge}_{\pd}/\pi$ such that $\underline e\vec f(X_1,X_2)\in N$.
  Since $\nabla_1(\underline e\vec f(X_1,X_2))=0$, we deduce from Remark \ref{Rmk-Nabla_1 as connection modulo pi} that
  \[\sum_{i,j\geq 0}\vec f_{ij}\lambda X_1^{[i-1]}X_2^{j} = 0.\]
  Thus, for any $i\geq 1$, $\vec f_{ij} = 0$.
  So $\vec f(X_1,X_2) = \vec f(0,X_2) = \sum_{j\geq 0}\vec f_{0j}X_2^{[j]}$.
  As $\nabla_0(\underline e\vec f(X_1,X_2)) = \nabla_0(\underline e\vec f(0,X_2))= 0$, by Lemma \ref{Lem-Solution modulo u}, the claim is true.

  Next, we claim that $\Delta$ induces an injection modulo $(E(u_2))$; that is, 
  \[(\frakM\otimes_{\frakS,q_0}\frakS^2/(E(u_2)))^{\nabla_0=\nabla_1=0}\rightarrow \frakM/(E(u))\]
  is an injection.
  
  Indeed, let $x\in (\frakM\otimes_{\frakS,q_0}\frakS^2/(E(u_2)))^{\nabla_0=\nabla_1=0}$ such that $\Delta(x) = 0$. Then the above argument shows that $x\equiv 0 \mod \pi$. So there is a $y\in \frakM\otimes_{\frakS,q_0}\frakS^2/(E(u_2))$ such that $x = \pi y$. Since $\nabla_0$ and $\nabla_1$ are $\calO_K$-linear, and both $\frakM\otimes_{\frakS,q_0}\frakS^2/(E(u_2))$ and $\frakM/(E(u))$ are $\pi$-torsion free, we see that $y$ satisfies the same condition as $x$, which implies that $y\equiv 0\mod \pi$. By iteration, we conclude that $x\equiv 0 \mod \pi^n$ for any $n\geq 1$ and hence that $x = 0$.
  
  Now, let $x \in (\frakM\otimes_{\frakS,q_0}\frakS^2)^{\tau_0=\tau_1=1}$ such that $\Delta(x) = 0$.
  By what we have proved above, $x\equiv 0\mod (E(u_2))$; that is, there is an $y\in \frakM\otimes_{\frakS,q_0}\frakS^2$ such that $E(u_2)y=x$.
  Since both $ \frakM\otimes_{\frakS,q_0}\frakS^2$ and $\frakM$ are $E$-torsion free, $y\in (\frakM\otimes_{\frakS,q_0}\frakS^2)^{\tau_0=\tau_1=1}$ and $\Delta(y)=0$. So $y\equiv 0\mod (E(u_2))$.
  By iteration, we show that for any $n\geq 1$, $x\equiv 0\mod E(u_2)^n$, which forces $x= 0$.
\end{proof}

\begin{prop}\label{Prop-Smallness condition gives stratification}
  The morphism $\epsilon$ constructed in  (\ref{Equ-Stratification}) is a stratification satisfying the cocycle condition.
\end{prop}
\begin{proof}
  The proof essentially appears in \cite[Proposition 3.18]{MT}.
  By construction, the base change of $\epsilon$ along the multiplication map $\Delta:\frakS^1\rightarrow \frakS$ is the identity morphism.

  By Lemma \ref{Lem-Stratification is tau-equivariant}, both sides are $\frakS^2$-linear which are compatible with $\tau_i$-actions for $i=0,1,2$.
  In particular, after restriction to $\frakM$, the both sides take values in $(\frakM\otimes_{\frakS,q_0}\frakS^2)^{\tau_0=\tau_1=1}$ and their compositions with the injection given in Proposition \ref{Prop-Key injection} are the identity maps.
  So we deduce $p^*_{01}(\epsilon)\circ p^*_{12}(\epsilon) = p^*_{02}(\epsilon)$.

  It remains to show that $\epsilon$ is an isomorphism. Indeed, its inverse is given by base changing it along the involution $\frakS^1\simeq \frakS^1$ swapping the two factors. We win!
\end{proof}
\begin{cor}\label{Cor-Hard direction}
    There exists an explicit functor 
    \[\bM:\Mod_{\Prism}(\frakS)\to\Vect((\calO_K),\calO_{\Prism})\]
    from the category of prismatic $\frakS$-modules to the category of prismatic crystals on $(\calO_K)_{\Prism}$.
\end{cor}
\begin{proof}
    For any prismatic $\frakS$-module $\frakM$, by Proposition \ref{Prop-Smallness condition gives stratification}, the morphism in (\ref{Equ-Stratification}) is a stratification on $M$ with respect to the cosimplicial ring $\frakS^{\bullet}$ satisfying the cocycle condition. Now, we can conclude by applying \cite[Proposition 2.7]{BS-b}.
\end{proof}

\subsection{Prismatic crystals induce prismatic $\frakS$-modules}
  Let $\bM\in\Vect((\calO_K)_{\Prism},\calO_{\Prism})$ be a prismatic crystal. Let $\frakM$ and $M_{\inf}$ be the evaluations of $\bM$ at $(\frakS,(E))$ and $(A_{\inf},(\xi))$, respectively. Then we have an isomorphism
  \[\frakM\otimes_{\frakS,\iota}A_{\inf}\cong M_{\inf}\]
  and the $\hat G$-action on $(A_{\inf},\xi)$ induces a $\hat G$-action on $M_{\inf}$ such that $\frakM\subset M_{\inf}^{\Gamma}$ (i.e. Item 1 of Definition \ref{Dfn-prismatic S module} holds true). Here we regard $\frakM$ as a sub-$\frakS$-module of $M_{\inf}$.
  We will see that such an $\frakM$ together with the $\hat G$-action on $M_{\inf}$ defines a prismatic $\frakS$-module once we show that Items 2 and 3 of Definition \ref{Dfn-prismatic S module} hold true. Let $\tau_0$ and $\tau_1$ be as in the previous subsection.

  Similar to Proposition Proposition \ref{Prop-Key isomorphism}, one can prove the following result:
\begin{prop}\label{Prop-Key isomorphism-II}
  The multiplication $\Delta:\frakS^1\rightarrow \frakS$ induces an isomorphism
  \[(\frakM\otimes_{\frakS,p_0}\frakS^1)^{\tau_0=1}\rightarrow \frakM\]
  of $\frakS$-modules.
\end{prop}
\begin{proof}
    Since we regard $\frakM$ as a sub-$\frakS$-module of $\frakM\otimes_{\frakS,p_0}\frakS^1$ via the face map $p_0$ induced by  $\{0\}\to\{1\}\subset\{0,1\}$, we see that $\tau_0$ acts on $\frakM$ trivially. So we are reduced to the case where $\frakM = \frakS$. In this situation, the result follows from Proposition \ref{Prop-Key isomorphism} (for $\frakM=\frakS$).
\end{proof}

\begin{cor}\label{Cor-Key isomorphism-II}
  Let $\epsilon: \frakM\otimes_{\frakS,p_0}\frakS^1\rightarrow \frakM\otimes_{\frakS,p_1}\frakS^1$ be a stratification induced by $\bM$. Then the multiplication $\Delta$ induces an isomorphism
  \[(\frakM\otimes_{\frakS,p_1}\frakS^1)^{\tau_0=1}\rightarrow \frakM\]
  of $\frakS$-modules.
\end{cor}
\begin{proof}
  The inverse of $\epsilon$ composed with the isomorphism in Proposition \ref{Prop-Key isomorphism-II} implies that
  \[\Delta:(\frakM\otimes_{\frakS,p_1}\frakS^1)^{\tau_0=1}\rightarrow \frakM\]
  is an isomorphism as desired.
\end{proof}

 \begin{prop}\label{Prop-Condition 2 and 3}
     Fix a prismatic crystal $\bM\in\Vect((\calO_K)_{\Prism},\calO_{\Prism})$ and let $\frakM$ be as above. Then the conditions 2 and 3 of Definition \ref{Dfn-prismatic S module} hold true for $\frakM$.
 \end{prop}
 \begin{proof}
  To see $(\tau-1)\frakM\subset \varphi^{-1}(\mu)uM_{\inf}$, by Corollary \ref{Cor-Key isomorphism-II}, it is enough to show that 
  \[(\tau_1-1)(\frakM\otimes_{\frakS,p_1}\frakS^1)^{\tau_0=1}\subset \varphi^{-1}(\mu)u_1(\frakM\otimes_{\frakS,p_1}\widetilde \frakS^1)^{\tau_0=1},\]
  where $\widetilde \frakS^1 = \frakS\widehat \otimes_{\rW(k)}\Ainf\{\frac{u_0-u_1}{E(u_1)}\}^{\wedge}_{(p,E)}$ (and then pull back above inclusion along $\Delta$).
  To this end, by noting that $\tau_1-1$ acts on $\frakM$ trivially (as $p_1$ is induced by $\{0\}\to\{0\}\subset\{0,1\}$), we are reduced to the case for $\frakM = \frakS$. But this follows from Lemma \ref{tau connection} immediately.

  Now, the above argument allows us to define a ``$\tau$-connection'' 
  \[\nabla_{\frakM}:=\frac{\tau-1}{\varphi^{-1}(\mu)u}: \frakM\to M_{\inf}\]
  as required in Item 3 of Definition \ref{Dfn-prismatic S module}. It remains to check that the conditions (a) and (b) there are satisfied. However, since $\bM$ induces a Hodge--Tate crystal in $\Vect((\calO_K)_{\Prism},\overline \calO_{\Prism})$, the proof of Theorem \ref{poly-nilpotency} also works in this situation. We are done.
 \end{proof}

 \begin{cor}\label{Cor-Easy direction}
     There exists an explicit functor
     \[\bM^{-1}:\Vect((\calO_K)_{\Prism},\calO_{\Prism})\to \Mod_{\Prism}(\frakS)\]
     from the category of prismatic crystals on $(\calO_K)_{\Prism}$ to the category of prismatic $\frakS$-modules.
 \end{cor}
 \begin{proof}
     This follows from Proposition \ref{Prop-Condition 2 and 3} together with arguments at the beginning of this subsection.
 \end{proof}
  Now, we are able to prove our main theorem: 
 \begin{proof}[Proof of Theorem \ref{crystal as S module}]
     The result follows from Corollary \ref{Cor-Hard direction} and Corollary \ref{Cor-Easy direction} by checking $\bM$ and $\bM^{-1}$ are quasi-inverses of each other (from their explicit constructions).
 \end{proof}
 
\section{Appendix}\label{Appendix}
In Appendix, we show that for a Hodge--Tate crystal $\bM\in\Vect((\calO_K)_{\Prism},\overline \calO_{\Prism})$,
\[\tau^{\geq 2}\rR\Gamma_{\Prism}(\bM) = 0.\]
We follow the notations adapted in Subsection \ref{SSec-compare cohomologies}. Put $\alpha = E'(\pi)$ and for any $s\geq 2$ and $1\leq i\leq s$, define $E_i\in\bN^s$ the generator of the $i$-th component of $\bN^s$ such that for any $i_1,\dots,i_s\geq 0$,
\[a_{i_1,\dots,i_s} = a_{i_1E_1+\dots i_sE_s}.\]

By (\ref{Equ-differential in explicit way-III}), for any $\vec f(\underline X)=\sum_{I\in\bN^s}a_{I}\underline X^{[I]}$,
\begin{equation}\label{Equ-differentials for s>2}
    \begin{split}
        &d^s(\underline e\vec f(\underline X))\\
        =&\underline e\sum_{j_1,\dots,j_s,l_1,\dots,l_s\geq 0}(1-\alpha X_1)^{-(j_1+\dots+j_s)-(l_1+\dots+l_s)-\frac{A}{\alpha}}a_{j_1+l_1,\dots,j_s+l_s}(-1)^{j_1+\dots+j_s}X_1^{[j_1]}\cdots X_1^{[j_s]}X_2^{[l_1]}\cdots X_{s+1}^{[l_s]}\\
        &+\underline e\sum_{k=1}^{s+1}(-1)^k\sum_{i_1,\dots,i_s\geq 0}a_{i_1,\dots,i_s}X_1^{[i_1]}\cdots X_{k-1}^{[i_{k-1}]}X_{k+1}^{[i_k]}\cdots X_{s+1}^{[i_s]}.
    \end{split}
\end{equation}
In particular, if $\underline e\vec f(\underline X)$ belongs to $\Ker(d^s)$, then
\begin{equation}\label{Equ-kernel d for s>2}
    \begin{split}
        &\sum_{j_1,\dots,j_s,l_1,\dots,l_s\geq 0}(1-\alpha X_1)^{-(j_1+\dots+j_s)-(l_1+\dots+l_s)-\frac{A}{\alpha}}a_{j_1+l_1,\dots,j_s+l_s}(-1)^{j_1+\dots+j_s}X_1^{[j_1]}\cdots X_1^{[j_s]}X_2^{[l_1]}\cdots X_{s+1}^{[l_s]}\\
        &+\sum_{k=1}^{s+1}(-1)^k\sum_{i_1,\dots,i_s\geq 0}a_{i_1,\dots,i_s}X_1^{[i_1]}\cdots X_{k-1}^{[i_{k-1}]}X_{k+1}^{[i_k]}\cdots X_{s+1}^{[i_s]} = 0.
    \end{split}
\end{equation}
\begin{lem}\label{Lem-kernel-constant-s>2}
  Keep notations as above. If $\underline e\vec f(\underline X)\in\Ker(d^s)$, then the coefficients $\{a_{0,l_2,\dots,l_s}\}_{l_2,\dots,l_s\geq 0}$ are uniquely determined by $\{a_{0,\dots,0,j_{2i+1},\dots,j_s}\}_{i\geq 1;j_{2i+1}\geq 1,j_{2i+2},\dots,j_s\geq 0}$.
  More precisely, for any $i\geq 1$, if $j_{2i},\dots,j_s\geq 1$,then
      $a_{0,\dots,0,j_{2i},\dots,j_s}=0$, and
  if $j_{2i}\geq 1$ and $j_{2i+1}\cdots j_s = 0$, then
  $a_{0,\dots,0,j_{2i},\dots,j_s}$ is a linear combination of finite elements in $\{a_{0,\dots,0,k_{2i+1},\dots,k_s}\}_{k_{2i+1}\geq 1,k_{2i+2},\dots,k_s\geq 0}$ with coefficients in $\bZ$, which are independent of $\vec f$.
\end{lem}
\begin{proof}
  Let $X_1=0$ in (\ref{Equ-kernel d for s>2}), we get
  \begin{equation}\label{Equ-constant of kernel-I}
      \sum_{k=2}^{s+1}(-1)^k\sum_{l_1,\dots,l_{k-2},l_k,\dots,l_s\geq 0}a_{0,l_1,\dots,l_{k-2},l_k,\dots,l_s}X_2^{[l_1]}\cdots X_{k-1}^{[l_{k-2}]}X_{k+1}^{[l_k]}\cdots X_{s+1}^{[l_s]} = 0.
  \end{equation}
  For any $r\geq 1$ and $2\leq t_1<\dots<t_r\leq s+1$, let $X_{t_1}=\dots=X_{t_r}=0$ and consider the coefficients of $\prod_{2\leq j\leq s+1,j\neq t_1,\dots,t_r}X_j^{[l_{j-1}]}$ with all $l_{j-1}\geq 1$. Then we have
  \begin{equation}\label{Equ-constant of kernel-II}
      \sum_{k=1}^r(-1)^{t_k}a_{\sum_{1\leq q\leq t_k-2;q+1\neq t_1,\dots,t_{k-1}}l_qE_{q+1}+\sum_{t_k\leq q\leq s;q+1\neq t_{k+1},\dots,t_r}l_qE_q} = 0.
  \end{equation}

  Now assume $i\geq 1$. Let $r=2i-1$ and $t_1=2,t_2=3,\dots,t_{2i-1}=2i$ in (\ref{Equ-constant of kernel-II}). We have
  \[\sum_{k=2}^{2i}(-1)^ka_{\sum_{2i\leq q\leq s}l_qE_q} = 0.\]
  This implies that $a_{0,\dots,0,l_{2i},\dots,l_s} = 0$ as desired.

  Let $r\geq 2i$ and $t_1=2,t_2=3,\dots,t_{2i-1}=2i, 2i+1\neq t_{2i}<\dots<t_{2r}$. We have
  \[\sum_{k=2}^{2i}(-1)^ka_{\sum_{2i\leq q\leq s;q+1\neq t_{2i},\dots,t_r}l_qE_q}+\sum_{h=2i}^r(-1)^{t_h}a_{\sum_{2i\leq q\leq t_h-2;q+1\neq t_{2i},\dots,t_{h-1}}l_qE_{q+1}+\sum_{t_h\leq q\leq s;q+1\neq t_{h+1},\dots,t_r}l_qE_q}=0.\]
  This implies that
  \[a_{\sum_{2i\leq q\leq s;q+1\neq t_{2i},\dots,t_r}l_qE_q}=\sum_{h=2i}^r(-1)^{t_h+1}a_{\sum_{2i\leq q\leq t_h-2;q+1\neq t_{2i},\dots,t_{h-1}}l_qE_{q+1}+\sum_{t_h\leq q\leq s;q+1\neq t_{h+1},\dots,t_r}l_qE_q}.\]
  Note that all $a_I$'s appearing in the right hand side have subscripts whose first $2i$ terms are zero. Since $l_{2i}\neq 0$ and $r\geq 2i$, we see that $a_{0,\dots,0,l_{2i},\dots,l_s}$ is a linear combination of $\{a_{0,\dots,0,k_{2i+1},\dots,k_s}\}_{k_{2i+1}\geq 1,k_{2i+2},\dots,k_s\geq 0}$ with coefficients in $\bZ$ as desired.
\end{proof}
\begin{lem}\label{Lem-kernel-X1-s>2}
  Keep notations as above. If $\underline e\vec f(\underline X)\in\Ker(d^s)$, then the coefficients $\{a_{l_1,l_2,\dots,l_s}\}_{l_1\geq 1,l_2,\dots,l_s\geq 0}$ are uniquely determined by $\{a_{1,j_{2},\dots,j_s}\}_{j_2,\dots,j_s\geq 0}$.
  More precisely, for any $I\in\bN_{\geq 1}\times\bN^{s-1}$ and any $J\in\{1\}\times \bN^{s-1}$, there exist integers $z_{I,J}^{(s)}$, which are independent of $\vec f$, such that
  \[a_I=\sum_{J\in \{1\}\times \bN^{s-1},|J|\leq |I|}z_{I,J}^{(s)}\prod_{i=|J|}^{|I|-1}(A+i\alpha)a_J,\]
  where we set $\prod_{i=|J|}^{|I|-1}(A+i\alpha) = 1$ if $|J|=|I|$.
\end{lem}
\begin{proof}
  In (\ref{Equ-kernel d for s>2}), considering the coefficients (including $X_2, \dots, X_{s+1}$) of $X_1^{[1]}$, we get
  \begin{equation}\label{Equ-X1 of kernel}
      \begin{split}
          &\sum_{L=(l_1,\dots,l_s)\in\bN^s}((A+|L|\alpha)a_L-\sum_{i=1}^sa_{L+E_i})X_2^{[l_1]}\cdots X_{s+1}^{[l_s]}\\
          =&\sum_{k=2}^{s+1}(-1)^{k-1}\sum_{l_1,\dots,l_{k-2},l_k\dots,l_s\geq 0}a_{1,l_1,\dots,l_{k-2},l_k,\dots,l_s}X_2^{[l_1]}\cdots X_{k-1}^{[l_{k-2}]}X_{k+1}^{[l_{k}]}\cdots X_{s+1}^{[l_s]}.
      \end{split}
  \end{equation}
  Comparing the coefficients of $X_2^{[l_1]}\cdots X_{s+1}^{[l_s]}$ for $l_1,\dots,l_s\geq 1$ in (\ref{Equ-X1 of kernel}), we see that
  \begin{equation}\label{Equ-X1 of kernel-I}
      a_{l_1+1,l_2,\dots,l_s} = (A+(l_1+\dots+l_s)\alpha)a_{l_1,\dots,l_2}-(a_{l_1,l_2+1,\dots,l_s}+\dots+a_{l_1,l_2,\dots,l_{s+1}}).
  \end{equation}
  For any $r\geq 1$ and $2\leq t_1<\dots<t_r\leq s+1$, comparing the coefficients of $\prod_{2\leq j\leq s+1;j\neq t_1,\dots,t_r}X_j^{[l_{j-1}]}$ with all $l_{j-1}\geq 1$, we get
  \begin{equation*}
      \begin{split}
          &(A+\sum_{1\leq q\leq s;1+q\neq t_1,\dots,t_r}l_q\alpha)a_{\sum_{1\leq q\leq s;1+q\neq t_1,\dots,t_r}l_qE_q}-(\sum_{j=1}^sa_{E_j+\sum_{1\leq q\leq s;1+q\neq t_1,\dots,t_r}l_qE_q})\\
          =&\sum_{k=1}^r(-1)^{t_k-1}a_{E_1+\sum_{1\leq q\leq t_k-2;q+1\neq t_1,\dots,t_{k-1}}l_qE_{q+1}+\sum_{t_k\leq q\leq s+1;1+q\neq t_{k+1},\dots,t_r}l_qE_q}.
      \end{split}
  \end{equation*}
  In other words, we have
  \begin{equation}\label{Equ-X1 of kernel-II}
      \begin{split}
          &a_{E_1+\sum_{1\leq q\leq s;1+q\neq t_1,\dots,t_r}l_qE_q}\\
          =&(A+\sum_{1\leq q\leq s;1+q\neq t_1,\dots,t_r}l_j\alpha)a_{\sum_{1\leq q\leq s;1+q\neq t_1,\dots,t_r}l_qE_q}-(\sum_{j=2}^sa_{E_j+\sum_{1\leq q\leq s;1+q\neq t_1,\dots,t_r}l_qE_q})\\
          &+\sum_{k=1}^r(-1)^{t_k}a_{E_1+\sum_{1\leq q\leq t_k-2;q+1\neq t_1,\dots,t_{k-1}}l_qE_{q+1}+\sum_{t_k\leq q\leq s;1+q\neq t_{k+1},\dots,t_r}l_qE_q}.
      \end{split}
  \end{equation}
  Now we complete the prove by induction on $I=(i_1,\dots,i_s)\in\bN_{\geq 1}\times \bN^{s-1}$, where $\bN_{\geq 1}\times \bN^{s-1}$ is endowed with the lexicographic order as in the proof of Theorem \ref{Thm-prismatic cohomology} for $s=2$ case.

  When $i_1=1$, there is nothing to prove. Now assume we have shown the result for all $a_I$'s with $I<N=(n_1,\dots,n_s)$. We may assume $n_1\geq 2$.

  If $n_2,\dots,n_s\geq 1$, by (\ref{Equ-X1 of kernel-I}), we have
  \[a_{N}=(A+(|N|-1)\alpha)a_{N-E_1}-(\sum_{i=2}^sa_{N-E_1+E_i}).\]
  This shows the result in this case by using the inductive hypothesis.

  If for some $r\geq 1$, the $0$ appears exactly $r$ times in $N$. We assume $2\leq h_1\leq \dots<h_r\leq s$ such that $n_{h_1} = \dots =n_{h_r} = 0$.
  Then put $t_i=h_i+1$ for any $1\leq i\leq r$ (in particular, $t_1\geq 3$) and put $l_1=n_1-1$ and $l_j = n_j$ for any $j\in\{1,\dots,s\}\setminus\{n_{h_1},\dots,n_{h_r}\}$ in (\ref{Equ-X1 of kernel-II}), we get
  \begin{equation*}
      \begin{split}
          a_N=&(A+(|N|-1)\alpha)a_{N-E_1}-(\sum_{j=2}^sa_{N-E_1+E_j})\\
          &-\sum_{k=1}^r(-1)^{h_k}a_{E_1+\sum_{1\leq q\leq h_k-1;q\neq h_1,\dots,h_{k-1}}n_qE_{q+1}-E_2+\sum_{h_k+1\leq q\leq s+1;q\neq h_{k+1},\dots,h_r}l_qE_q}.
      \end{split}
  \end{equation*}
  This combined with the inductive hypothesis gives the result in this case. So we complete the proof.
\end{proof}
  Note that in the $s=2$ case, $a_{1,0} = Aa_{0,0}$. So we want to generalize this to the $s\geq 2$ case.
\begin{lem}\label{Lem-kernel-const vs X1}
   Keep notations as above. Assume $\underline e\vec f(\underline X)\in\Ker(d^s)$, then for any $i\geq 1$, the coefficients $\{a_{1,0,\dots,0,j_{2i+1},\dots,j_s}\}_{j_{2i+1}\geq 1,j_{2i+2},\dots,j_s\geq 0}$ are uniquely determined by $\{a_{1,0,\dots,0,j_{2i+2},\dots,j_s}\}_{j_{2i+2}\geq 1,j_{2i+3},\dots,j_s\geq 0}$ and $\{a_{0,I}\}_{I\in\bN^{s-1}}$, where for $I=(i_2,\dots,i_s)\in\bN^{s-1}$, $a_{0,I}=a_{0,i_2,\dots,i_s}$.

   More precisely, put $I=(1,0,\dots,0,j_{2i+1},\dots,j_s)$. If $j_{2i+1},\dots,j_s\geq 1$, then
   \[a_{I}=(A+(|I|-1)\alpha)a_{I-E_1}-(\sum_{j=2}^sa_{I-E_1+E_j});\]
   and if $j_{2i+1}\geq 1$ and $j_{2i+2}\cdots j_s=0$, then
   for any $J\in \{1\}\times\{0\}^{2i}\times\bN_{\geq 1}\times\bN^{s-2i-2}$, there exists an integer $z_{I,J}^{(s)}$, which is independent of $\vec f$, such that
   \[a_I=\sum_{I\in\in \{1\}\times\{0\}^{2i}\times\bN_{\geq 1}\times\bN^{s-2i-2},|I|=|J|}z_{I,J}^{(s)}a_J+(A+(|I|-1)\alpha)a_{I-E_1}-(\sum_{j=2}^sa_{I-E_1+E_j}).\]
\end{lem}
\begin{proof}
  Assume $i\geq 1$. Let $r\geq 2i$ and $t_1=2,t_2=3,\cdots, t_{2i-1}=2i$ in (\ref{Equ-X1 of kernel-II}). Then we get
  \begin{equation}\label{Equ-kernel-const vs X1-I}
      \begin{split}
          &\sum_{k=2i}^r(-1)^{t_k-1}a_{E_1+\sum_{2i\leq q\leq t_k-2;1+q\neq t_{2i},\dots,t_{k-1}}l_qE_{q+1}+\sum_{t_k\leq q\leq s;1+q\neq t_{k+1},\dots,t_r}l_qE_q}\\
          =&(A+\sum_{2i\leq q\leq s;1+q\neq t_{2i},\dots,t_r}\alpha)a_{\sum_{2i\leq q\leq s;1+q\neq t_{2i},\dots,t_r}l_qE_q}-\sum_{j=2}^sa_{E_j+\sum_{2i\leq q\leq s;1+q\neq t_{2i},\dots,t_r}l_qE_q}.
      \end{split}
  \end{equation}
  In (\ref{Equ-kernel-const vs X1-I}), assume $r=2i$ and $t_{2i}=2i+1$. Then we get
  \[a_{1,0,\dots,0,l_{2i+1},\dots,l_s} = (A+\sum_{q=2i+1}^sl_q\alpha)a_{0,\dots,0,l_{2i+1},\dots,l_s}-(\sum_{j=2}^sa_{E_j+\sum_{2i+1\leq q\leq s}l_qE_q})\]
  as desired.

  In (\ref{Equ-kernel-const vs X1-I}), assume moreover $r\geq 2i+1$ and $t_{2i}=2i+1, t_{2i+1}\neq 2i+2$. Then we get
  \begin{equation}
  \begin{split}
      &a_{E_1+l_{2i+1}E_{2i+1}+\sum_{2i+2\leq q\leq s;1+q\neq t_{2i+1},\dots,t_r}l_qE_q}\\
      =&\sum_{k=2i+1}^r(-1)^{t_k}a_{E_1+l_{2i+1}E_{2i+2}+\sum_{2i+2\leq q\leq t_k-2;1+q\neq t_{2i+1},\dots,t_{k-1}}l_qE_{q+1}+\sum_{t_k\leq q\leq s;1+q\neq t_{k+1},\dots,t_r}l_qE_q}\\
      &+(A+\sum_{2i+1\leq q\leq s;1+q\neq t_{2i+1},\dots,t_r}\alpha)a_{\sum_{2i+1\leq q\leq s;1+q\neq t_{2i+1},\dots,t_r}l_qE_q}\\
      &-\sum_{j=2}^sa_{E_j+\sum_{2i+1\leq q\leq s;1+q\neq t_{2i+1},\dots,t_r}l_qE_q}.
  \end{split}
  \end{equation}
  So the lemma follows from the fact that $l_{2i+1}\neq 0$.
\end{proof}
  In summary, we get the following proposition.
\begin{prop}\label{Prop-rigidty for kernel}
  Let $\Lambda_s$ be the subset of $\bN^s$ consisting of indices $J=(j_1,\dots,j_s)$ satisfying one of the following conditions:
  \begin{enumerate}
      \item there exists some $i\geq 1$ such that $j_1=1$, $j_2=\cdots=j_{2i-1}=0$ and $j_{2i}\geq 1$.

      \item there exists some $i\geq 1$ such that $j_1=\cdots=j_{2i}=0$ and $j_{2i+1}\geq 1$.
  \end{enumerate}
  Assume $\underline e\vec f(\underline X)\in\Ker(d^s)$ and $\vec f(\underline X)=\sum_{I\in\bN^s}a_I\underline X^{[I]}$, then for any $J\in\Lambda_s$, there exists an integer $z_{I,J}^{(s)}$, which is independent of $\vec f$, such that
  \[a_I=\sum_{J\in\Lambda_s,|J|\leq |I|}z_{I,J}^{(s)}\prod_{i=|J|}^{|I|-1}(A+i\alpha)a_J.\]
\end{prop}
\begin{proof}
  The result follows from Lemma \ref{Lem-kernel-constant-s>2}, Lemma \ref{Lem-kernel-X1-s>2} and Lemma \ref{Lem-kernel-const vs X1} directly.
\end{proof}

 Now we describe how to construct a
 \[\vec g(\underline X)=\sum_{I\in\bN^{s-1}}b_I\underline X^{[I]} \in\calO_K\{X_1,\dots,X_{s-1}\}^{\wedge}_{\pd}\]
 for an element $\underline e\vec f(\underline X)\in\Ker(d^s)$ such that $d^{s-1}(\underline e\vec g(\underline X)) = \underline e\vec f(\underline X)$. Write
 \[d^{s-1}(\underline e\vec g(\underline X)) = \sum_{I\in\bN^s}c_I\underline X^{[I]}.\]
 Then (\ref{Equ-differentials for s>2}) shows that
 \begin{equation}\label{Equ-image for s>2}
    \begin{split}
        &\sum_{i_1,\dots,i_s\geq 0}c_{i_1,\dots,i_s}X_1^{[i_1]}\cdots X_s^{[i_s]}\\
        =&\sum_{j_1,\dots,j_{s-1},l_1,\dots,l_{s-1}\geq 0}(1-\alpha X_1)^{-(j_1+\dots+j_{s-1})-(l_1+\dots+l_{s-1})-\frac{A}{\alpha}}b_{j_1+l_1,\dots,j_{s-1}+l_{s-1}}(-1)^{j_1+\dots+j_{s-1}}\\
        &\cdot X_1^{[j_1]}\cdots X_1^{[j_{s-1}]}X_2^{[l_1]}\cdots X_{s}^{[l_{s-1}]}
        +\sum_{k=1}^{s}(-1)^k\sum_{i_1,\dots,i_{s-1}\geq 0}b_{i_1,\dots,i_{s-1}}X_1^{[i_1]}\cdots X_{k-1}^{[i_{k-1}]}X_{k+1}^{[i_k]}\cdots X_{s}^{[i_{s-1}]}.
    \end{split}
\end{equation}
\begin{lem}\label{Lem-constant-image-s>2}
  Keep notations as above. Then for any $i\geq 0$, one can recover $\{b_{0,\dots,0,l_{2i},\dots,l_{s-1}}\}_{l_{2i}\geq 1,l_{2i+1},\dots,l_{s-1}\geq 0}$ from $\{c_{0,\dots,0,j_{2i+1},\dots,j_s}\}_{i\geq 1;j_{2i+1}\geq 1,j_{2i+2},\dots,j_s\geq 0}$ and $\{b_{0,\dots,0,j_{2i+1},\dots,j_{s-1}}\}_{j_{2i+1}\geq 1,j_{2i+2},\dots,j_{s-1}\geq 0}$.

  More precisely, for any $I\in\{0\}^{2i-1}\times\bN_{\geq 1}\times\bN^{s-2i-1}$ and $J\in\{0\}^{2i}\times\bN_{\geq 1}\times\bN^{s-2i-2}$, there exist integers $z_{I,J}^{(s)}$'s, which are independent of $\vec g$, such that
  \[b_I=\sum_{J\in\{0\}^{2i}\times\bN_{\geq 1}\times\bN^{s-2i-2},|J|=|I|}z_{I,J}^{(s)}b_J+c_{0,I},\]
  where $c_{0,I}=c_{0,i_1,\dots,i_{s-1}}$ if $I=(i_1,\dots,i_{s-1})$.
\end{lem}
\begin{proof}
  In (\ref{Equ-image for s>2}), let $X_1=0$. Then we get
  \begin{equation}\label{Equ-constant-image-s>2-I}
  \begin{split}
      &\sum_{l_1,\dots,l_{s-1}\geq 0}c_{0,l_1,\dots,l_{s-1}}X_2^{[l_1]}\cdots X_s^{[l_{s-1}]}\\
      =&\sum_{k=2}^s(-1)^k\sum_{l_1,\dots,l_{k-2},l_k,\dots,l_{s-1}\geq 0}b_{0,l_1,\dots,l_{k-2},l_k,\dots,l_{s-1}}X_2^{[l_1]}\cdots X_{k-1}^{[l_{k-2}]}X_{k+1}^{[l_k]}\cdots X_s^{[l_{s-1}]}.
  \end{split}
  \end{equation}
  Now assume $i\geq 1$. For any $r\geq 0$ and $2i+2\leq t_1<\dots<t_r\leq s$, let $X_2=\cdots=X_{2i}=X_{t_1}=\dots=X_{t_r}=0$ in (\ref{Equ-constant-image-s>2-I}) and consider the coefficients of $\prod_{2i+1\leq j\leq s;j\neq t_1,\dots,t_r}X_{j}^{[l_{j-1}]}$ with all $l_{j-1}\geq 1$. Then we get
  \begin{equation*}
      \begin{split}
          &c_{l_{2i}E_{2i+1}+\sum_{2i+1\leq q\leq s-1;q+1\neq t_1,\dots,t_r}l_qE_{q+1}}\\
          =& b_{l_{2i}E_{2i}+\sum_{2i+1\leq q\leq s-1;q+1\neq t_1,\dots,t_r}l_qE_q}\\
          &+\sum_{k=1}^{r}(-1)^{t_k}b_{l_{2i}E_{2i+1}+\sum_{2i+1\leq q\leq t_k-2;1+q\neq t_1,\dots,t_{k-1}}l_qE_{q+1}
          +\sum_{t_k\leq q\leq s-1;1+q\neq t_{k+1},\dots,t_r}l_qE_q}.
      \end{split}
  \end{equation*}
  Equivalently, we have
  \begin{equation*}
      \begin{split}
          &b_{l_{2i}E_{2i}+\sum_{2i+1\leq q\leq s-1;q+1\neq t_1,\dots,t_r}l_qE_q}\\
          =& c_{l_{2i}E_{2i+1}+\sum_{2i+1\leq q\leq s-1;q+1\neq t_1,\dots,t_r}l_qE_{q+1}}\\
          &-\sum_{k=1}^{r}(-1)^{t_k}b_{l_{2i}E_{2i+1}+\sum_{2i+1\leq q\leq t_k-2;1+q\neq t_1,\dots,t_{k-1}}l_qE_{q+1}
          +\sum_{t_k\leq q\leq s-1;1+q\neq t_{k+1},\dots,t_r}l_qE_q}.
      \end{split}
   \end{equation*}
   Since $l_{2i}\geq 1$, the lemma follows.
\end{proof}

\begin{lem}\label{Lem-X1-image-s>2}
  Keep notations as above. Then one can recover $\{b_{i_1,\dots,i_{s-1}}\}_{i_1\geq 1,i_2,\dots,i_{s-1}\geq 0}$ from $\{c_{1,j_2,\dots,j_s}\}_{j_2,\dots,j_s\geq 0}$ and $\{b_{1,i_2,\dots,i_{s-1}}\}_{i_2,\dots,i_{s-1}\geq 0}$.

  More precisely, for any $I,L\in\bN_{\geq 1}\times\bN^{s-2}$ and any $J\in\bN^{s-1}$, there exist integers $z_{I,J}^{(s)}$ and $w_{I,L}^{(s)}$, which are independent of $\vec g$, such that
  \[b_I=\sum_{J\in\bN^{s-1},|J|\leq|I|}z_{I,J}^{(s)}\prod_{i=|J|}^{|I|-1}(A+i\alpha)c_{1,J}+\sum_{L\in\bN_{\geq 1}\times \bN^{s-2},|J|\leq|I|}w_{I,J}^{(s)}\prod_{i=|L|}^{|I|-1}(A+i\alpha)b_L,\]
  where we set $\prod_{i=|J|}^{|I|-1} (A+i\alpha)= 1$ if $|J|=|I|$ and $c_{1,J}=c_{1,j_1,\dots,j_{s-1}}$ for $J=(j_1,\dots,j_{s-1})$.
\end{lem}
\begin{proof}
  Comparing the coefficients (including $X_2,\dots,X_s$) in (\ref{Equ-image for s>2}), we get
  \begin{equation}\label{Equ-X1-image-s>2-I}
      \begin{split}
          &\sum_{l_1,\dots,l_{s-1}\geq 0}c_{1,l_1,\dots,l_{s-1}}X_2^{[l_1]}\cdots X_s^{[l_{s-1}]}\\
          =&\sum_{l_1,\dots,l_{s-1}\geq 0}((A+\sum_{j=1}^{s-1}l_j\alpha)b_{l_1,\dots,l_{s-1}}-(b_{l_1+1,l_2,\dots,l_{s-1}}+\dots,b_{l_1,l_2,\dots,l_{s-1}+1}))X_2^{[l_1]}\cdots X_s^{[l_{s-1}]}\\
          &+\sum_{k=2}^s(-1)^k\sum_{l_1,\dots,l_{k-2},l_k,\dots,l_{s-1}}b_{1,l_1,\dots,l_{k-2},l_k,\dots,l_{s-1}}X_2^{[l_1]}\cdots X_{k-1}^{[l_{k-2}]}X_{k+1}^{[l_k]}\cdots X_s^{[l_{s-1}]}.
      \end{split}
  \end{equation}
  In (\ref{Equ-X1-image-s>2-I}), comparing the coefficients of $X_2^{[l_1]}\cdots X_s^{[l_{s-1}]}$ for $I=(l_1,\dots,l_{s-1})\in\bN^{s-1}_{\geq 1}$, we get
  \begin{equation}\label{Equ-X1 of image-I}
      b_{I+E_1} = -c_{1,I}+(A+|I|\alpha)b_I-(\sum_{j=2}^{s-1}b_{I+E_j}).
  \end{equation}
  For any $r\geq 1$ and $2\leq t_1<\dots<t_r\leq s$, comparing the coefficients of $\prod_{2\leq j\leq s;j\neq t_1,\dots,t_r}X_j^{[l_{j-1}]}$ with all $l_{j-1}\geq 1$, we get
  \begin{equation*}
      \begin{split}
          &c_{\sum_{1\leq q\leq s-1;1+q\neq t_1,\dots,t_r}E_1+l_qE_{q+1}}\\
          =&(A+\sum_{1\leq q\leq s-1;1+q\neq t_1,\dots,t_r}l_q\alpha)b_{\sum_{1\leq q\leq s-1;1+q\neq t_1,\dots,t_r}l_qE_q}-(\sum_{j=1}^{s-1}b_{E_j+\sum_{1\leq q\leq s-1;1+q\neq t_1,\dots,t_r}l_qE_q})\\
          &+\sum_{k=1}^r(-1)^{t_k}b_{E_1+\sum_{1\leq q\leq t_k-2;q+1\neq t_1,\dots,t_{k-1}}l_qE_{q+1}+\sum_{t_k\leq q\leq s;1+q\neq t_{k+1},\dots,t_r}l_qE_q}.
      \end{split}
  \end{equation*}
  In other words, we have
  \begin{equation}\label{Equ-X1 of image-II}
      \begin{split}
          &b_{E_1+\sum_{1\leq q\leq s-1;1+q\neq t_1,\dots,t_r}l_qE_q}\\
          =&-c_{\sum_{1\leq q\leq s-1;1+q\neq t_1,\dots,t_r}E_1+l_qE_{q+1}}+(A+\sum_{1\leq q\leq s-1;1+q\neq t_1,\dots,t_r}l_q\alpha)b_{\sum_{1\leq q\leq s-1;1+q\neq t_1,\dots,t_r}l_qE_q}\\ &-(\sum_{j=2}^{s-1}b_{E_j+\sum_{1\leq q\leq s-1;1+q\neq t_1,\dots,t_r}l_qE_q})\\
          &+\sum_{k=1}^r(-1)^{t_k}b_{E_1+\sum_{1\leq q\leq t_k-2;q+1\neq t_1,\dots,t_{k-1}}l_qE_{q+1}+\sum_{t_k\leq q\leq s-1;1+q\neq t_{k+1},\dots,t_r}l_qE_q}.
      \end{split}
  \end{equation}
  Now the lemma follows from a similar inductive argument as in the proof of Lemma \ref{Lem-kernel-X1-s>2} by using (\ref{Equ-X1 of image-I}) and (\ref{Equ-X1 of image-II}) instead of (\ref{Equ-X1 of kernel-I}) and (\ref{Equ-X1 of kernel-II}), respectively.
\end{proof}
\begin{lem}\label{Lem-image-const vs X1}
   Keep notations as above. For any $i\geq 1$, put
   \[I=(1,0\dots,0,l_{2i+1},\dots,l_{s-1})\in\bN^{s-1}\]
   such that $l_{2i+1}\geq 1$. Then one can recover $b_I$ from
   $\{b_{1,0,\dots,0,k_{2i+2},\dots,k_{s-1}}\}_{k_{2i+2}\geq 1,k_{2i+3},\dots,k_{s-1}\geq 0}$, $\{b_L\}_{L\in\{0\}\times\bN^{s-2}}$ and $\{c_{1,0,\dots,0,j_{2i+2},\dots,j_s}\}_{j_{2i+2}\geq 1,j_{2i+3},\dots,j_s\geq 0}.$

   More precisely, if $l_{2i+2},\dots,l_{s-1}\geq 1$, then
   \[b_I =-c_{1,I-E_1}+(A+(|I|-1)\alpha)b_{I-E_1}-(\sum_{j=2}^{s-1}b_{I-E_1+E_j})\]
   and if $l_{2i+2}\cdots l_{s-1}=0$, then for any $J\in\{1\}\times\{0\}^{2i}\times\bN_{\geq 1}\times\bN^{s-2i-3}$, there exist an integer $z_{I,J}^{(s)}$, which is independent of $\vec g$, such that
   \[b_I=\sum_{J\in\{1\}\times\{0\}^{2i}\times\bN_{\geq 1}\times\bN^{s-2i-3},|J|=|I|}{z_{I,J}^{(s)}}b_J-c_{1,I-E_1}+(A+(|I|-1)\alpha)b_{I-E_1}-(\sum_{j=2}^{s-1}b_{I-E_1+E_j}),\]
   where for $J=(j_1,\dots,j_{s-1})\in\bN^{s-1}$, $c_{1,J}=c_{1,j_1,\dots,j_{s-1}}$.
\end{lem}
\begin{proof}
  Assume $i\geq 1$. In (\ref{Equ-X1 of image-II}), let $r\geq 2i$ and $t_1=2,\dots,t_{2i-1}=2i$. Then we get
  \begin{equation}\label{Equ-image-const vs X1-I}
      \begin{split}
          &\sum_{k=2i}^r(-1)^{t_k}b_{E_1+\sum_{2i\leq q\leq t_k-2;q+1\neq t_{2i},\dots,t_{k-1}}l_qE_{q+1}+\sum_{t_k\leq q\leq s-1;1+q\neq t_{k+1},\dots,t_r}l_qE_q}\\
          =&c_{\sum_{2i\leq q\leq s-1;1+q\neq t_{2i},\dots,t_r}E_1+l_qE_{q+1}}-(A+\sum_{2i\leq q\leq s-1;1+q\neq t_{2i},\dots,t_r}l_q\alpha)b_{\sum_{2i\leq q\leq s-1;1+q\neq t_{2i},\dots,t_r}l_qE_q}\\ &+(\sum_{j=2}^{s-1}b_{E_j+\sum_{1\leq q\leq s-1;1+q\neq t_1,\dots,t_r}l_qE_q}).
      \end{split}
  \end{equation}
  Let $r=2i$ and $t_{2i}=2i+1$ in (\ref{Equ-image-const vs X1-I}) and put $I=(1,0,\dots,0,l_{2i+1},\dots,l_{s-1})$. Then we get
  \[
    b_I=-c_{1,I-E_1}+(A+(|I|-1)\alpha)b_{I-E_1}-(\sum_{j=2}^{s-1}b_{I-E_1+E_j}),
  \]
  which implies the result in this case.

  Let $r\geq 2i+1$ and $t_{2i}=2i+1$, $t_{2i+1}\neq 2i+2$ in (\ref{Equ-image-const vs X1-I}). We get
  \begin{equation}\label{Equ-image-const vs X1-II}
      \begin{split}
          &b_{E_1+l_{2i+1}E_{2i+1}+\sum_{2i+2\leq q\leq s-1;1+q\neq t_{2i+1},\dots,t_r}l_qE_q}\\
          =&\sum_{k=2i+1}^r(-1)^{t_k}b_{E_1+l_{2i+1}E_{2i+2}+\sum_{2i+2\leq q\leq t_k-2;q+1\neq t_{2i+1},\dots,t_{k-1}}l_qE_{q+1}+\sum_{t_k\leq q\leq s-1;1+q\neq t_{k+1},\dots,t_r}l_qE_q}\\
          &-c_{\sum_{2i+1\leq q\leq s-1;1+q\neq t_{2i+1},\dots,t_r}E_1+l_qE_{q+1}}+(A+\sum_{2i+1\leq q\leq s-1;1+q\neq t_{2i+1},\dots,t_r}l_q\alpha)b_{\sum_{2i+1\leq q\leq s-1;1+q\neq t_{2i+1},\dots,t_r}l_qE_q}\\ &-(\sum_{j=2}^{s-1}b_{E_j+\sum_{2i+1\leq q\leq s-1;1+q\neq t_{2i+1},\dots,t_r}l_qE_q}).
      \end{split}
  \end{equation}
  So the result holds in this case as $l_{2i+1}\neq 0$.
\end{proof}

Now, we are prepared to prove $\tau^{\geq 2}\rR\Gamma_{\Prism}(\bM) = 0$.
For any $\underline e\vec f(\underline X)\in\Ker(d^s)$, write
\[\vec f(\underline X) = \sum_{I\in\bN^s}a_I\underline X^{[I]}.\]
 We define $\vec g(\underline X)=\sum_{I\in\bN^{s-1}}b_I\underline X^{[I]}\in\calO_K\{X_1,\dots,X_{s-1}\}^{\wedge}_{\pd}$ as follows:

 For any $i\geq 1$ and any $I\in\{0\}^{2i}\times\bN_{\geq 1}\times\bN^{s-2i-2}$, we put $b_I=0$.

 For any $i\geq 1$ and any $I\in\{0\}^{2i-1}\times\bN_{\geq 1}\times\bN^{s-2i-1}$, we put $b_I=a_{0,I}$.
 Then by Lemma \ref{Lem-constant-image-s>2}, for any $L\in\{0\}\times\bN^{s-2}$, $b_L$ is determined in the unique way. Moreover, since $\lim_{n\to+\infty}\prod_{i=0}^{n-1}(A+i\alpha) = 0$, we see that $\lim_{|L|\to+\infty}b_L=0$.

 For any $i\geq 0$ and any $I\in \{1\}\times\{0\}^{2i}\times\bN_{\geq 1}\times\bN^{s-2i-3}$, we put $b_I=0$.

 For any $i\geq 0$ and any $I = (1,0,\dots,0,l_{2i+1},\dots,l_{s-1})\in\bN^{s-1}$ with $l_{2i+1}\geq 1$, we define $b_I$ as follows:

 If $l_{2i+2},\dots,l_{s-1}\geq 1$, then put
 \[b_I= -a_{1,I-E_1}+(A+(|I|-1)\alpha)b_{I-E_1}-(\sum_{j=2}^{s-1}b_{I-E_1+E_j}).\]

 If $l_{2i+2}\cdots l_{s-1}=0$, let $z_{I,J}^{(s)}$'s be the integers defined in Lemma \ref{Lem-image-const vs X1}, then put
 \begin{equation*}
     \begin{split}
         b_I&=\sum_{J\in\{1\}\times\{0\}^{2i}\times\bN_{\geq 1}\times\bN^{s-2i-3},|J|=|I|}{z_{I,J}^{(s)}}b_J-a_{1,I-E_1}+(A+(|I|-1)\alpha)b_{I-E_1}-(\sum_{j=2}^{s-1}b_{I-E_1+E_j})\\
         &=-a_{1,I-E_1}+(A+(|I|-1)\alpha)b_{I-E_1}-(\sum_{j=2}^{s-1}b_{I-E_1+E_j}).
     \end{split}
 \end{equation*}
 Clearly, in this case, we still have $\lim_{|I|\to+\infty}b_I=0$.

 Finally, for any $I\in \bN_{\geq 2}\times\bN^{s-2}$, let $z_{I,J}^{(s)}$'s and $w_{I,J}^{(s)}$'s be the integers defined in Lemma \ref{Lem-X1-image-s>2}. Then we put
 \[b_I=\sum_{J\in\bN^{s-1},|J|\leq|I|}z_{I,J}^{(s)}\prod_{i=|J|}^{|I|-1}(A+i\alpha)c_{1,J}+\sum_{L\in\bN_{\geq 1}\times \bN^{s-2},|J|\leq|I|}w_{I,J}^{(s)}\prod_{i=|L|}^{|I|-1}(A+i\alpha)b_L.\]
 Again, since $\lim_{n\to+\infty}\prod_{i=0}^{n-1}(A+i\alpha) = 0$, we see that $\lim_{|I|\to+\infty}b_I=0$ in this case.
 So for any $I\in\bN^{s-1}$, $\lim_{|I|\to+\infty}b_I=0$, which implies that $\vec g(\underline X)$ is well-defined.

 It remains to show $d^{s-1}(\underline e\vec g(\underline X)) = \underline e\vec f(\underline X)$. Write
 \[d^{s-1}(\underline e\vec g(\underline X)) = \underline e\sum_{I\in\bN^s}c_I\underline X^{[I]}.\]
 Then by construction of $\vec g(\underline X)$ and Lemma \ref{Lem-constant-image-s>2}, Lemma \ref{Lem-X1-image-s>2} and Lemma \ref{Lem-image-const vs X1}, we see that
 \[c_I=a_I\]
 for any $I\in\Lambda_s$. Since $\Ima(d^{s-1})\subset\Ker(d^s)$, it follows from Proposition \ref{Prop-rigidty for kernel} that
 \[d^{s-1}(\underline e\vec g(\underline X)) = \underline e\vec f(\underline X).\]
 So we show that $\Ker(d^s)\subset\Ima(d^{s-1})$ for any $s\geq 2$, which implies $\rH^s_{\Prism}(\bM)=0$ as desired.

\end{document}